\documentclass[11pt,english,final]{amsart}
\usepackage[T1]{fontenc}
\usepackage[latin9]{inputenc}
\usepackage{babel}
\usepackage{amsbsy}
\usepackage{amstext}
\usepackage{amsthm}
\usepackage{amssymb}
\usepackage{stmaryrd}
\usepackage{graphicx}
\usepackage[letterpaper]{geometry}
\usepackage[pdfusetitle,
 bookmarks=true,bookmarksnumbered=false,bookmarksopen=false,
 breaklinks=false,pdfborder={0 0 1},backref=false,colorlinks=false]
 {hyperref}

\makeatletter
\numberwithin{equation}{section}
\numberwithin{figure}{section}
\theoremstyle{plain}
\newtheorem{thm}{\protect\theoremname}[section]
\theoremstyle{definition}
\newtheorem{defn}[thm]{\protect\definitionname}
\theoremstyle{remark}
\newtheorem{rem}[thm]{\protect\remarkname}
\theoremstyle{remark}
\newtheorem{notation}[thm]{\protect\notationname}
\theoremstyle{definition}
\newtheorem{example}[thm]{\protect\examplename}
\theoremstyle{plain}
\newtheorem{lem}[thm]{\protect\lemmaname}
\theoremstyle{plain}
\newtheorem{cor}[thm]{\protect\corollaryname}
\theoremstyle{plain}
\newtheorem{prop}[thm]{\protect\propositionname}
\theoremstyle{remark}
\newtheorem{claim}[thm]{\protect\claimname}

\@ifundefined{date}{}{\date{}}
\usepackage{amsmath, tikz-cd, amssymb, placeins}
\DeclareRobustCommand*\cal{\@fontswitch\relax\mathcal}

\usetikzlibrary{calc}
\usetikzlibrary{decorations.pathmorphing}

\tikzset{curve/.style={settings={#1},to path={(\tikztostart)
    .. controls ($(\tikztostart)!\pv{pos}!(\tikztotarget)!\pv{height}!270:(\tikztotarget)$)
    and ($(\tikztostart)!1-\pv{pos}!(\tikztotarget)!\pv{height}!270:(\tikztotarget)$)
    .. (\tikztotarget)\tikztonodes}},
    settings/.code={\tikzset{quiver/.cd,#1}
        \def\pv##1{\pgfkeysvalueof{/tikz/quiver/##1}}},
    quiver/.cd,pos/.initial=0.35,height/.initial=0}

\tikzset{tail reversed/.code={\pgfsetarrowsstart{tikzcd to}}}
\tikzset{2tail/.code={\pgfsetarrowsstart{Implies[reversed]}}}
\tikzset{2tail reversed/.code={\pgfsetarrowsstart{Implies}}}
\tikzset{no body/.style={/tikz/dash pattern=on 0 off 1mm}}

\makeatother

\providecommand{\claimname}{Claim}
\providecommand{\corollaryname}{Corollary}
\providecommand{\definitionname}{Definition}
\providecommand{\examplename}{Example}
\providecommand{\lemmaname}{Lemma}
\providecommand{\notationname}{Notation}
\providecommand{\propositionname}{Proposition}
\providecommand{\remarkname}{Remark}
\providecommand{\theoremname}{Theorem}

\begin{document}
\title{An unoriented skein exact triangle in unoriented link Floer homology}
\author{Gheehyun Nahm}
\thanks{The author was partially supported by the Simons Grant \emph{New structures in low-dimensional topology.}}
\begin{abstract}
We define band maps in unoriented link Floer homology and show that
they form an unoriented skein exact triangle. These band maps are
similar to the band maps in equivariant Khovanov homology given by
the Lee deformation.

As a key tool, we use a Heegaard Floer analogue of Bhat's recent $2$-surgery
exact triangle in instanton Floer homology, which may be of independent
interest. Unoriented knot Floer homology corresponds to $I^{\sharp}$
of the knot in our $2$-surgery exact triangle.
\end{abstract}

\address{Department of Mathematics, Princeton University, Princeton, New Jersey
08544, USA}
\email{gn4470@math.princeton.edu}

\maketitle
\tableofcontents{}

\section{\label{sec:Introduction}Introduction}

Heegaard Floer homology, defined by Ozsv\'{a}th and Szab\'{o} \cite{MR2113019},
is a powerful Floer theoretic invariant of closed, oriented three-manifolds.
A closely related invariant for knots is knot Floer homology, defined
by Ozsv\'{a}th and Szab\'{o} \cite{MR2065507} and independently
by Rasmussen \cite{MR2704683}, which generalizes to oriented links.
Link Floer homology categorifies the Alexander polynomial, and satisfies
an oriented skein exact triangle \cite[Theorem 10.2]{MR2065507} that
categorifies the oriented skein relation. For unoriented links, there
is a closely related invariant, unoriented link Floer homology, defined
by Ozsv\'{a}th, Stipsicz, and Szab\'{o} \cite{MR3694597}. In this
paper, we show that an algebraic variant of unoriented link Floer
homology\footnote{See Remark \ref{rem:curlyCFL-tensor} for a comparison with the invariant
defined in \cite{MR3694597}.}, $\boldsymbol{HFL'}^{-}(Y,L)$, satisfies an unoriented skein exact
triangle over the field $\mathbb{F}=\mathbb{Z}/2$ with two elements.
\begin{defn}
\label{def:unoriented-skein-triple-intro}Three links $L_{a},L_{b},L_{c}\subset Y$
form an \emph{unoriented skein triple }if the links are identical
outside of a ball $B^{3}\subset Y$, in which they differ as in Figure
\ref{fig:skein-moves}. 
\end{defn}

\begin{thm}
\label{thm:unoriented-exact-intro}Let $L_{a},L_{b},L_{c}\subset Y$
be an unoriented skein triple. Then there exists an $\mathbb{F}\llbracket U\rrbracket$-linear
exact triangle 
\[
\cdots\to\boldsymbol{HFL'}^{-}(Y,L_{a})\to\boldsymbol{HFL'}^{-}(Y,L_{b})\to\boldsymbol{HFL'}^{-}(Y,L_{c})\to\boldsymbol{HFL'}^{-}(Y,L_{a})\to\cdots.
\]
\end{thm}

A key tool in proving Theorem \ref{thm:unoriented-exact-intro} is
Claim \ref{claim:step-3}, a local version of Theorem \ref{thm:2-surgery-intro},
which may be of independent interest. Theorem \ref{thm:2-surgery-intro}
is a Heegaard Floer analogue of Bhat's recent $2$-surgery exact triangle
\cite{2311.04242} for $I^{\sharp}$, an instanton Floer theoretic
invariant \cite{floer1988instanton} for three-manifolds and links
defined by Kronheimer and Mrowka \cite{MR2805599}.
\begin{thm}
\label{thm:2-surgery-intro}Let $K$ be a knot in $Y$. Then given
any framing $\lambda$ of $K$, there exists an $\mathbb{F}\llbracket U\rrbracket$-linear
exact triangle
\[
\cdots\to\boldsymbol{HFL'}^{-}(Y,K)\to\boldsymbol{HF}^{-}(Y_{\lambda}(K))\to\boldsymbol{HF}^{-}(Y_{\lambda+2\mu}(K))\to\boldsymbol{HFL'}^{-}(Y,K)\to\cdots,
\]
where $\mu$ is the meridian, $Y_{\lambda}(K)$, resp., $Y_{\lambda+2\mu}(K)$
is the $\lambda$, resp., $\lambda+2\mu$-surgery of $Y$ along $K$,
and $\boldsymbol{HF}^{-}$ denotes the Heegaard Floer homology of
the three-manifold.
\end{thm}

These Heegaard Floer theoretic invariants come in various flavors,
for which our exact triangles also hold (see Theorems \ref{thm:unoriented-exact-triangle}
and \ref{thm:2surgery}). We propose the \emph{unreduced hat version},
$\widehat{HFL'}$ (Definition \ref{def:hat-version}) as the counterpart
of $I^{\sharp}$ in Heegaard Floer homology. Indeed, Theorem \ref{thm:unoriented-exact-intro}
is partly motivated by Kronheimer and Mrowka's skein exact triangle
for $I^{\sharp}$ \cite{MR2805599}. Another closely related motivation
is to define a spectral sequence from reduced Khovanov homology to
knot Floer homology. See Subsection \ref{subsec:Motivation} for further
discussion.

\begin{figure}[h]
\begin{centering}
\includegraphics[scale=2]{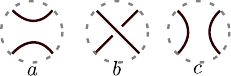}
\par\end{centering}
\caption{\label{fig:skein-moves}A local diagram for an unoriented skein triple}
\end{figure}

\begin{rem}
We thank Fan Ye for communicating to the author that a similar analogue
of $I^{\sharp}$ for knots in Heegaard Floer homology\footnote{For a knot $K$, put two link basepoints on it and consider the chain
complex over $\mathbb{Z}[H]/H^{2}$ where both basepoints have weight
$2H$.} has been suggested in his Miami talk \cite{miami}.
\end{rem}

\subsection{\label{subsec:Unoriented-link-Floer}Unoriented link Floer homology}

Let us define the unoriented link Floer homology groups that we study
in this paper.
\begin{defn}
\label{def:unoriented-intro}Let $L$ be a $k$-component link in
a closed, oriented three-manifold $Y$. To define the \emph{unoriented
link Floer homology} of $L$, we need to choose two basepoints $w_{i},z_{i}\in L_{i}$
for each component $L_{i}$ of $L$. Let $(\Sigma,\boldsymbol{\alpha},\boldsymbol{\beta},\boldsymbol{w},\boldsymbol{z})$
be an admissible, $2k$-pointed Heegaard diagram that represents this
pointed link $L\subset Y$. Then the \emph{unoriented link Floer homology
chain complex} is freely generated by intersection points ${\bf x}\in\mathbb{T}_{\boldsymbol{\alpha}}\cap\mathbb{T}_{\boldsymbol{\beta}}$:
\[
\boldsymbol{CFL'}^{-}(Y,L)=\bigoplus_{{\bf x}\in\mathbb{T}_{\boldsymbol{\alpha}}\cap\mathbb{T}_{\boldsymbol{\beta}}}\mathbb{F}\llbracket U_{1}^{1/2},\cdots,U_{k}^{1/2}\rrbracket{\bf x},
\]
and the differential is given by counting \emph{both basepoints} $w_{i}$
and $z_{i}$ with weight $U_{i}^{1/2}$:
\[
\partial^{-}{\bf x}=\sum_{{\bf y}\in\mathbb{T}_{\boldsymbol{\alpha}}\cap\mathbb{T}_{\boldsymbol{\beta}}}\sum_{\phi\in D({\bf x},{\bf y}),\ \mu(\phi)=1}\#\mathcal{M}(\phi)\prod_{i}U_{i}^{\frac{1}{2}\left(n_{w_{i}}(\phi)+n_{z_{i}}(\phi)\right)}{\bf y}.
\]
The \emph{unoriented link Floer homology of the link $L\subset Y$
}is $\boldsymbol{HFL'}^{-}(Y,L)$, the homology of $\boldsymbol{CFL'}^{-}(Y,L)$. 

The action of the $U_{i}$'s on $\boldsymbol{CFL'}^{-}(Y,L)$ are
homotopic by \cite[Equation (2.8)]{MR4845975}, and so we also view
$\boldsymbol{HFL'}^{-}(Y,L)$ as an $\mathbb{F}\llbracket U\rrbracket$-module,
where $U$ acts by multiplying by $U_{i}$.
\end{defn}

We define the \emph{unreduced }and\emph{ reduced hat versions} of
unoriented link Floer homology. We will be lax about issues related
to naturality in this section (in particular, basepoints) and postpone
the discussion to Section \ref{sec:band-maps}.
\begin{defn}
\label{def:hat-version}The \emph{(unreduced) hat version of unoriented
link Floer homology}, $\widehat{HFL'}(Y,L)$, is the homology of the
chain complex
\[
\widehat{CFL'}(Y,L)=\boldsymbol{CFL'}^{-}(Y,L)/U_{1}.
\]
\end{defn}

\begin{defn}
\label{def:reduced-hat}A \emph{marked link} $L$ is a link together
with a distinguished basepoint (which we call the \emph{marked point})
on $L$. The\emph{ reduced (hat version of) unoriented link Floer
homology of a marked link $L$, $\widetilde{HFL'}(Y,L)$,} is the
homology of
\[
\widetilde{CFL'}(Y,L)=\boldsymbol{CFL'}^{-}(Y,L)/U_{i}^{1/2},
\]
where $L_{i}$ is the link component that the marked point is on.

Given a basepoint on a link $L$, the corresponding \emph{basepoint
action}\footnote{Beware that in Heegaard Floer homology, the basepoint action usually
refers to the $\Phi_{w}$ and $\Psi_{z}$ maps \cite[Section 4.2]{MR3905679}.
We call multiplication by $U_{i}^{1/2}$ the basepoint action to match
the convention of Khovanov homology: see Remark \ref{rem:khovanov-convention}.} is multiplication by $U_{i}^{1/2}$ if the point is on the link component
$L_{i}$.
\end{defn}

For \emph{knots}, the reduced hat version, $\widetilde{HFK'}$, is
just knot Floer homology, $\widehat{HFK}$. We will study some examples
shortly, in Subsubsections \ref{subsec:hopf} and \ref{subsec:Trefoils}.

We end this subsection with some remarks that explain various aspects
of the definitions.
\begin{rem}
In the definition of the unreduced hat version, we could have quotiented
out by any $U_{i}$, since multiplication by $U_{1}$ is homotopic
to multiplication by $U_{i}$.
\end{rem}

\begin{rem}
\label{rem:curlyCFL-tensor}All these unoriented link Floer chain
complexes can be defined using a more general chain complex ${\cal CFL}^{-}$
defined over ${\cal R}=\mathbb{F}\llbracket Z_{1},\cdots,Z_{k},W_{1},\cdots,W_{k}\rrbracket$.
For instance, we have
\[
\boldsymbol{CFL'}^{-}={\cal CFL}^{-}\otimes_{{\cal R}}\mathbb{F}\llbracket U_{1}^{1/2},\cdots,U_{k}^{1/2},Z_{1},\cdots,Z_{k},W_{1},\cdots,W_{k}\rrbracket/(U_{i}^{1/2}=Z_{i}=W_{i}\ {\rm for\ all\ }i).
\]
In \cite{MR3694597}, they identify all the $U_{i}^{1/2}$'s together
and work with polynomial rings instead of power series rings: the
completion of their chain complex is 
\[
{\cal CFL}^{-}\otimes_{{\cal R}}\mathbb{F}\llbracket U^{1/2},Z_{1},\cdots,Z_{k},W_{1},\cdots,W_{k}\rrbracket/(U^{1/2}=Z_{i}=W_{i}\ {\rm for\ all\ }i).
\]
This is different from ours for links with more than one component,
although they are the same for knots.
\end{rem}

\begin{rem}
For simplicity, let us focus on links in $S^{3}$. Let $L\subset S^{3}$
be a $k$-component oriented link. Link Floer homology $\widehat{HFL}(Y,L)$
has the \emph{Maslov gradings}\footnote{There are other conventions that differ from ours by an overall shift
$\in\frac{1}{2}\mathbb{Z}$.} ${\rm gr}_{\boldsymbol{w}},{\rm gr}_{\boldsymbol{z}}\in\mathbb{Z}$
and \emph{the Alexander grading} \cite{MR2443092} which takes values
in $\mathbb{H}=h+\mathbb{Z}^{k}\subset(\frac{1}{2}\mathbb{Z})^{k}$
for some $h$\footnote{If $L$ is a knot, then $\mathbb{H}=\mathbb{Z}$.}.
The \emph{collapsed Alexander grading}, obtained by summing the coordinates
of $\mathbb{H}$, is $({\rm gr}_{\boldsymbol{w}}-{\rm gr}_{\boldsymbol{z}})/2\in\mathbb{Z}$.

Hence, we can define $\chi_{{\rm gr}_{\boldsymbol{z}}}(\widehat{HFL}(S^{3},L))$,
the Euler characteristic of $\widehat{HFL}(S^{3},L)$ with respect
to ${\rm gr}_{\boldsymbol{z}}$, as an element of the free, rank $1$
module $\mathbb{Z}[\mathbb{H}]$ over the group ring $\mathbb{Z}[\mathbb{Z}^{k}]$.
This Euler characteristic can be written in terms of the multi-variable
Alexander polynomial.

Our unoriented link Floer homology group is singly graded, by the
$\delta$-grading $({\rm gr}_{\boldsymbol{w}}+{\rm gr}_{\boldsymbol{z}})/2\in\mathbb{Z}$
(as in \cite[Subsection 2.1]{MR3694597}), since we identify $Z_{i}=W_{i}$
in the sense of Remark \ref{rem:curlyCFL-tensor}. However, the Alexander
grading modulo $(2\mathbb{Z})^{k}$ still exists. Let $\mathbb{H}^{un}:=h+(\mathbb{Z}/2)^{\oplus k}\subset\left(\frac{1}{2}\mathbb{Z}/2\mathbb{Z}\right)^{\oplus k}$
be the corresponding $(\mathbb{Z}/2)^{\oplus k}$-torsor.

Hence, we can consider the Euler characteristic\footnote{If we want to consider $L$ as an unoriented link, then the spaces
$\mathbb{H}$ and $\mathbb{H}^{un}$ should be thought of as $H_{1}(Y\backslash L;\mathbb{Z})$-
and $H_{1}(Y\backslash L;\mathbb{Z}/2)$-torsors living in $H_{1}(Y\backslash L;\frac{1}{2}\mathbb{Z})$
and $H_{1}(Y\backslash L;\frac{1}{2}\mathbb{Z}/2\mathbb{Z})$, respectively.
Also, changing the orientation of a link component changes the $\delta$-grading
by an overall shift, and hence changes the Euler characteristic by
a factor of $\pm1$. (Compare \cite[Proposition 7.1]{MR3694597}.)
Viewed this way, these Euler characteristics, up to sign, do not depend
on the orientation of $L$.} of the unreduced and reduced hat versions (let us mark the $j$th
component of $L$) with respect to the $\delta$-grading as an element
of the free, rank $1$ module $\mathbb{Z}[\mathbb{H}^{un}]$ over
the group ring $\mathbb{Z}[(\mathbb{Z}/2)^{\oplus k}]$. Then, we
have
\[
\chi_{\delta}(\widehat{HFL'}(S^{3},L))=(1-m_{j})\chi_{\delta}(\widetilde{HFL'}(S^{3},L)),\ \chi_{\delta}(\widehat{HFL}(S^{3},L))=\prod_{i\neq j}(1+m_{i})\chi_{\delta}(\widetilde{HFL'}(S^{3},L)),
\]
where $m_{i}$ is the $i$th unit vector of $(\mathbb{Z}/2)^{\oplus k}$.
(We thank Jacob Rasmussen for this remark.)
\end{rem}

\begin{rem}
\label{rem:khovanov-convention}We assigned the weights $U_{i}^{1/2}$
to the basepoints and called multiplication by $U_{i}^{1/2}$ the
basepoint action, to match the convention of equivariant Khovanov
homology \cite{MR1740682,MR2173845,MR4504654}: the homology groups
have a $U$-action, and the basepoint action squared is the $U$-action.
Similarly, the unreduced and reduced hat versions are defined analogously
to Khovanov homology, which we now recall.

The \emph{equivariant Khovanov homology given by the Lee deformation}
uses the Frobenius algebra ${\cal A}=\mathbb{F}[x,U]/(x^{2}=U)$ over
$\mathbb{F}[U]$ and the multiplication and comultiplication maps
are $\mathbb{F}[U]$-linear maps given by the following.
\begin{gather*}
m:{\cal A}\otimes_{\mathbb{F}[U]}{\cal A}\to{\cal A}:1\otimes1\mapsto1,\ 1\otimes x\mapsto x,\ x\otimes1\mapsto x,\ x\otimes x\mapsto U1\otimes1\\
\Delta:{\cal A}\to{\cal A}\otimes_{\mathbb{F}[U]}{\cal A}:1\mapsto1\otimes x+x\otimes1,\ x\mapsto x\otimes x+U1\otimes1
\end{gather*}
Given a link diagram, one considers the cube of resolutions, and uses
this Frobenius algebra to build a chain complex. In particular, the
vertices of the cube of resolutions chain complex are the equivariant
Khovanov homology of unlinks, which is
\[
\mathbb{F}[x_{1},\cdots,x_{k},U]/(x_{1}^{2}=\cdots=x_{k}^{2}=U)
\]
for the $k$-component unlink. The homology of this chain complex
is \emph{equivariant Khovanov homology}, which corresponds to the
minus version\footnote{Since we work over a field of characteristic $2$, the homology is
not as interesting: as $\mathbb{F}[U]$-modules, it is isomorphic
to unreduced Khovanov homology tensored with $\mathbb{F}[U]$ over
$\mathbb{F}$.}. To get \emph{unreduced Khovanov homology}, or the hat version, we
quotient out the chain complex by $U$. If the link has a marked point,
then the chain complex can be viewed as a chain complex over $\mathbb{F}[x,U]/(x^{2}=U)$,
where the action of $x$ is given by the marked point in each resolution.
The \emph{reduced Khovanov homology,} or the reduced hat version,
is given by quotienting out the chain complex by $x$.
\end{rem}

\begin{rem}
\label{rem:minus-polynomial}We can define the minus version of unoriented
link Floer homology over a polynomial ring instead of a power series
ring; this involves assuming \emph{strong $\mathfrak{s}$-admissibility},
which depends on the ${\rm Spin}^{c}$-structure $\mathfrak{s}$.
The unoriented skein exact triangle, Theorem \ref{thm:unoriented-exact},
holds in this version as well (Remark \ref{rem:skein-exact-hfl}),
although the $2$-surgery exact triangle, Theorem \ref{thm:2-surgery-intro},
does not (Remark \ref{rem:2-exact-hfl}). We mainly work with power
series rings and suppress ${\rm Spin}^{c}$-structures to simplify
the exposition. 
\end{rem}

\subsection{An unoriented skein exact triangle}

We will continue to be lax about basepoints in this subsection.
\begin{defn}
\label{def:unoriented-skein-triple}Three (marked) links $L_{a},L_{b},L_{c}\subset Y$
form an \emph{unoriented skein triple }if the links and the marked
points (if the links are marked) are identical outside of a ball $B^{3}\subset Y$,
in which they differ as in Figure \ref{fig:skein-moves}. 
\end{defn}

\begin{defn}
A \emph{band on a link $L\subset Y$} is an ambient $2$-dimensional
$1$-handle, i.e. it is an embedding $\iota:[0,1]\times[0,1]\hookrightarrow Y$
such that $\iota^{-1}(L)=[0,1]\times\{0,1\}$. A link $L'$ is obtained
by a \emph{band move on a link $L$} if it is given by surgering $L$
along a band on $L$.

The following are the three types of bands:
\begin{itemize}
\item a band is a \emph{non-orientable band} if it does not change the number
of connected components;
\item a band is a \emph{split band} if it increases the number of connected
components;
\item a band is a \emph{merge band} if it decreases the number of connected
components.
\end{itemize}
\end{defn}

Unoriented skein triples are related by band moves, as in Figure \ref{fig:skein-moves-bands},
and the cyclic order of the three bands is always non-orientable,
split, and merge.

\begin{figure}[h]
\begin{centering}
\includegraphics[scale=2]{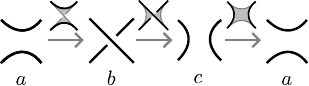}
\par\end{centering}
\caption{\label{fig:skein-moves-bands}A local diagram for an unoriented skein
triple together with the band maps}
\end{figure}

\begin{thm}
\label{thm:unoriented-exact}Let $L_{a},L_{b},L_{c}\subset Y$ be
an unoriented skein triple. Then there exist exact triangles in the
various versions
\begin{gather*}
\cdots\to\boldsymbol{HFL'}^{-}(Y,L_{a})\to\boldsymbol{HFL'}^{-}(Y,L_{b})\to\boldsymbol{HFL'}^{-}(Y,L_{c})\to\boldsymbol{HFL'}^{-}(Y,L_{a})\to\cdots\\
\cdots\to\widehat{HFL'}(Y,L_{a})\to\widehat{HFL'}(Y,L_{b})\to\widehat{HFL'}(Y,L_{c})\to\widehat{HFL'}(Y,L_{a})\to\cdots\\
\cdots\to\widetilde{HFL'}(Y,L_{a})\to\widetilde{HFL'}(Y,L_{b})\to\widetilde{HFL'}(Y,L_{c})\to\widetilde{HFL'}(Y,L_{a})\to\cdots,
\end{gather*}
where the maps can be interpreted as band maps.
\end{thm}

If $Y=S^{3}$, then the above band maps are homogeneous with respect
to the $\delta$-grading\footnote{In general, they preserve relative homological gradings; see Subsection
\ref{subsec:spinc-structures-and-grading}.}.

We show in Section \ref{sec:Computation-for-unlinks} that band maps
for planar links in unoriented link Floer homology agree with those
in Khovanov homology.

\subsubsection{\label{subsec:hopf}Unlinks and the Hopf link}

Recall that link Floer homology $\widehat{HFL}$ cannot have a skein
exact triangle without any modifications, as there should be an exact
triangle involving the unknot, the Hopf link, and the unknot\footnote{In Manolescu's unoriented skein exact triangle \cite{MR2350128},
the unknots have two extra basepoints each; this doubles the rank
of $\widehat{HFL}({\rm unknot})$.}, but $\widehat{HFL}({\rm unknot})$ has rank $1$ and $\widehat{HFL}({\rm Hopf\ link})$
has rank $4$.

In contrast, the unreduced hat version of unoriented link Floer homology
of the unlink $UL_{k}$ with $k$ components has rank $2^{k}$, and
that of the Hopf link $H$ has rank $4$. The Hopf link and two unknots
indeed form an exact triangle, and in fact all these homology groups
are the same as Khovanov homology, as vector spaces.

\subsubsection{\label{subsec:Trefoils}Trefoil knots}

The unoriented knot Floer homology of a trefoil in $S^{3}$ is 
\[
\mathbb{F}\llbracket U^{1/2}\rrbracket\oplus\mathbb{F}\llbracket U^{1/2}\rrbracket/(U^{1/2}),
\]
and so the unreduced hat version has rank $4$ and the reduced hat
version has rank $3$. The rank of the reduced version is the same
as the ranks of both reduced Khovanov and reduced instanton Floer
homology over any field, but the rank of the unreduced version is
the same as the ranks of unreduced Khovanov and unreduced instanton
Floer homology over $\mathbb{Q}$, but not over $\mathbb{F}=\mathbb{Z}/2$.
This is as expected: similarly, it is conjectured that the $\widehat{HF}$
and $I^{\sharp}$ of a three-manifold are isomorphic over $\mathbb{Q}$
(\cite[Conjecture 7.24]{MR2652464}, \cite[Conjecture 1.1]{MR4407491}),
but they are not isomorphic over $\mathbb{F}$ (\cite[Theorem 1.5]{2311.04242},
\cite{li20242torsioninstantonfloerhomology}).

\subsubsection{Comparing with Manolescu's maps \cite{MR2350128}}

Manolescu's unoriented skein exact triangle and ours both involve
pointed links and maps between them. Let us discuss them for one specific
example: a non-orientable band $B:K_{a}\to K_{b}$ between two unknots
in $S^{3}$. In this case, the number of basepoints on the knots and
the rank of the maps are different: Manolescu's map has half rank,
and our map is zero (as in Khovanov homology).

Recall that if we have an admissible Heegaard diagram with three attaching
curves $\boldsymbol{\alpha},\boldsymbol{\beta}_{a},\boldsymbol{\beta}_{b}$,
then we can define a holomorphic triangle counting map 
\[
\mu_{2}:CF(\boldsymbol{\alpha},\boldsymbol{\beta}_{a})\otimes CF(\boldsymbol{\beta}_{a},\boldsymbol{\beta}_{b})\to CF(\boldsymbol{\alpha},\boldsymbol{\beta}_{b}).
\]
The band maps are $\mu_{2}(-\otimes\Theta_{B})$ for some cycle $\Theta_{B}\in CF(\boldsymbol{\beta}_{a},\boldsymbol{\beta}_{b})$.

Manolescu's unoriented skein exact triangle requires the two unknots
to each have at least four basepoints. See Figure \ref{fig:non-orientable-unknot-manolescu}\footnote{The crossing of $K_{b}$ is drawn correctly; see Remarks \ref{rem:orientations}
and \ref{rem:orientation2} for orientation conventions.}: we consider the simplest case where there are exactly four. First,
the attaching curves $\boldsymbol{\alpha}$ and $\boldsymbol{\beta}_{i}$
describe $K_{i}$ for $i\in\{a,b\}$, i.e. $\widehat{HFK}(S^{3},K_{i})=\widehat{HF}(\boldsymbol{\alpha},\boldsymbol{\beta}_{i})$\footnote{By $\widehat{HF}(\boldsymbol{\alpha},\boldsymbol{\beta}_{i})$, we
mean the homology of the chain complex freely generated over $\mathbb{F}$
by intersection points $\mathbb{T}_{\boldsymbol{\alpha}}\cap\mathbb{T}_{\boldsymbol{\beta}_{i}}$
and where the differential counts holomorphic bigons that do not cross
any basepoints.}. Under this identification, the cycle $\Theta_{B}=\tau_{0}+\tau_{1}\in\widehat{CF}(\boldsymbol{\beta}_{a},\boldsymbol{\beta}_{b})$
induces the band map $\widehat{HFK}(S^{3},K_{a})\to\widehat{HFK}(S^{3},K_{b})$.
It turns out that all the maps in Manolescu's exact triangle come
from similar looking diagrams.

In contrast, our exact triangle requires the two knots to have exactly
two basepoints, and they need to be on the same side of the knot with
respect to the band, as in Figure \ref{fig:non-orientable-unknot-ours}
(we could have chosen $w_{1},z_{1}$ instead). See Figure \ref{fig:non-orientable-unknot-ours}:
similarly, $\boldsymbol{HFK'}^{-}(S^{3},K_{i})=\boldsymbol{HF}^{-}(\boldsymbol{\alpha},\boldsymbol{\beta}_{i})$\footnote{By $\boldsymbol{HF}^{-}(\boldsymbol{\alpha},\boldsymbol{\beta}_{i})$,
we mean the homology of the chain complex freely generated over $\mathbb{F}\llbracket U^{1/2}\rrbracket$
by intersection points $\mathbb{T}_{\boldsymbol{\alpha}}\cap\mathbb{T}_{\boldsymbol{\beta}_{i}}$
and where the differential counts every holomorphic bigon, where both
basepoints have weight $U^{1/2}$, as in Definition \ref{def:unoriented-intro}.}, and under this identification, the cycle $\tau_{1}\in\boldsymbol{CF}^{-}(\boldsymbol{\beta}_{a},\boldsymbol{\beta}_{b})$
induces the band map $\boldsymbol{HFK'}^{-}(S^{3},K_{a})\to\boldsymbol{HFK'}^{-}(S^{3},K_{b})$
(and also in the unreduced and reduced hat versions). It turns out
that all the non-orientable band maps in our exact triangle come from
similar looking diagrams; however, the diagrams for the split and
merge band maps are different, and these band maps are more complicated
to describe.

\begin{figure}[h]
\begin{centering}
\includegraphics[scale=1.5]{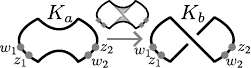}\includegraphics[scale=0.75]{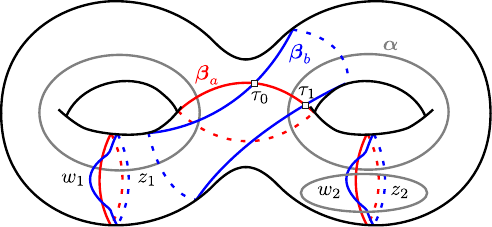}
\par\end{centering}
\caption{\label{fig:non-orientable-unknot-manolescu}A non-orientable band
between two unknots for Manolescu's exact triangle, and a Heegaard
diagram for it}
\end{figure}

\begin{figure}[h]
\begin{centering}
\includegraphics[scale=1.5]{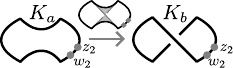}\includegraphics[scale=0.75]{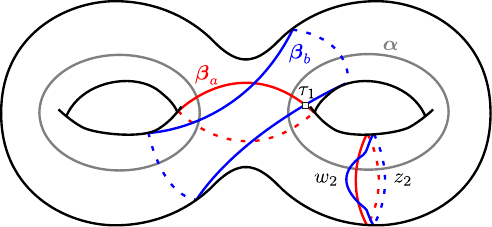}
\par\end{centering}
\caption{\label{fig:non-orientable-unknot-ours}A non-orientable band between
two unknots for our exact triangle, and a Heegaard diagram for it}
\end{figure}

\subsection{\label{subsec:The-proof}A word on the proof}

We discuss some key ideas involved in the definition of the split
and merge band maps and the proof of Theorem \ref{thm:unoriented-exact}.

The split and merge band maps involve links with different numbers
of link components, and it is not obvious how to define them since
we want exactly two basepoints on each link component. (Compare \cite{MR2574747}.)
We achieve this by adding a basepoint in the three-manifold when we
consider the links with one fewer component (this does not change
the quasi-isomorphism type of the chain complex \cite[Proposition 6.5]{MR2443092}).
This basepoint becomes two link basepoints on the link with more components.

See Figure \ref{fig:genus2-everything}: if we consider only $\boldsymbol{\beta}_{a}$
or $\boldsymbol{\beta}_{b}$, then we treat the basepoints $w_{1}$
and $z_{1}$ as one basepoint (with weight $U_{1})$ on the three-manifold.
When we consider $\boldsymbol{\beta}_{c}$, this basepoint becomes
two basepoints (with weight $U_{1}^{1/2}$ each) on the extra link
component.

As usual, we prove the exact triangle by essentially using the triangle
detection lemma \cite[Lemma 4.2]{MR2141852} and reducing it to a
local computation. We have to work in ${\rm Sym}^{2}(\mathbb{T}^{2})$,
i.e. we have to compute the number of certain kinds of holomorphic
disks, triangles, and quadrilaterals in ${\rm Sym}^{2}(\mathbb{T}^{2})$,
which makes the local computation challenging. Instead of directly
carrying out the local computation in ${\rm Sym}^{2}(\mathbb{T}^{2})$,
we use a trick that reduces it to two local computations in $\mathbb{T}^{2}$,
which are combinatorial and well understood. This trick utilizes Claim
\ref{claim:step-3}, a local version of Theorem \ref{thm:2-surgery-intro},
which is a Heegaard Floer analogue of Bhat's $2$-surgery exact triangle
\cite{2311.04242}. (See the discussion right before Theorem \ref{thm:2surgery}
and Section \ref{sec:The-plan}.)

\begin{figure}[h]
\begin{centering}
\includegraphics{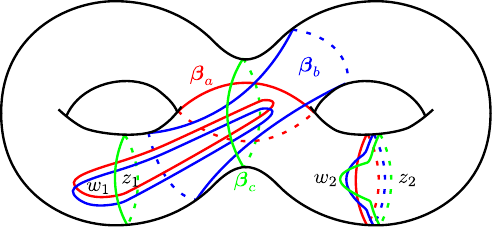}
\par\end{centering}
\caption{\label{fig:genus2-everything}We reduce the proof of Theorem \ref{thm:unoriented-exact}
to a local computation, Theorem \ref{thm:sym2-t2}.}
\end{figure}

\begin{rem}
We prove Theorem \ref{thm:2-surgery-intro} directly, by doing a model
computation on the torus in the spirit of \cite{MR2113020,MR2141852}.
We thank Ian Zemke for communicating to the author that Theorem \ref{thm:2-surgery-intro}
can also be proven using his bordered $\boldsymbol{HF}^{-}$ theory
\cite{2109.11520,1011.1317,MR2377279}.
\end{rem}

\begin{rem}
\label{rem:orbifold}The instanton invariants $I^{\sharp}(Y,L)$ come
from the orbifold $Y(L)$ whose underlying three-manifold is $Y$
and $L$ is the $\mathbb{Z}/2$-orbifold points, and the maps in Bhat's
$2$-surgery exact triangle \cite{2311.04242} come from orbifold
cobordism maps.

We show that $\boldsymbol{CFL'}^{-}(Y,L)$ splits into ${\rm Spin}^{c}(Y(L))$-summands
where ${\rm Spin}^{c}(Y(L))$ is an $H_{1}^{orb}(Y(L);\mathbb{Z})$-torsor,
but we do not discuss further connections with orbifolds in this paper:
we do not interpret the maps in Theorem \ref{thm:2-surgery-intro}
as orbifold cobordism maps. However, our maps come from triple Heegaard
diagrams, and so it is possible to interpret them as a composition
of a link cobordism map \cite{MR3905679} and some algebraic maps
such as the creation and annihilation maps that we define in Subsubsection
\ref{subsec:Birth-and-death}.
\end{rem}

\subsection{\label{subsec:Motivation}Motivation}

Link Floer homology and instanton Floer homology are closely related
to Khovanov homology \cite{MR1740682}, which is defined combinatorially
from an unoriented cube of resolutions. Rasmussen \cite{MR2189938}
conjectured that the rank of reduced Khovanov homology is always greater
than or equal to the rank of knot Floer homology; Dowlin \cite{MR4777638}
proved this conjecture over $\mathbb{Q}$.

One challenge in relating Khovanov homology to knot Floer homology
lies in the difficulty of finding a suitable generalization of knot
Floer homology for links that satisfies an unoriented skein exact
triangle\footnote{Dowlin's spectral sequence uses an oriented cube of resolutions which
involves singular knots.}. Manolescu's unoriented skein exact triangle \cite{MR2350128} in
link Floer homology requires additional basepoints on the links. As
a result, the $E^{2}$ page of the corresponding cube of resolutions
is not a link invariant, which was shown by Baldwin and Levine \cite{MR2964628}.
Also, Baldwin, Levine, and Sarkar \cite{MR3604486} suggested that
one should take the Koszul resolution with respect to the basepoint
actions on the Khovanov side.

In instanton Floer homology, Kronheimer and Mrowka \cite{MR2805599}
defined a spectral sequence from unreduced Khovanov homology to $I^{\sharp}$
by defining and iterating an unoriented skein exact triangle for $I^{\sharp}$.
They also defined a spectral sequence from reduced Khovanov homology
to $I^{\natural}$, a reduced version of $I^{\sharp}$, which is conjecturally
isomorphic to $\widehat{HFK}$ for \emph{knots} over $\mathbb{Q}$
(\cite[Conjecture 7.24]{MR2652464}, \cite[Conjecture 1.1]{MR4407491}).
Hence, a Heegaard Floer analogue of $I^{\sharp}$ and $I^{\natural}$
for (marked) links would be a good candidate for a suitable generalization
of knot Floer homology from the perspective of Khovanov homology.

Our motivation is to develop a theory in Heegaard Floer homology that
is analogous to $I^{\sharp}$ and $I^{\natural}$ and construct a
spectral sequence from Khovanov homology to this theory. We propose
the unreduced and reduced hat versions of unoriented link Floer homology
(Definitions \ref{def:hat-version} and \ref{def:reduced-hat}) as
the counterparts of $I^{\sharp}$ and $I^{\natural}$, and show Theorems
\ref{thm:2-surgery-intro} and \ref{thm:unoriented-exact} which are
counterparts of theorems in instanton Floer homology. In Section \ref{sec:Computation-for-unlinks},
we show that for planar links, the band maps in the various versions
of unoriented link Floer homology agree with the band maps in (equivariant)
Khovanov homology, and hence also with the band maps in $I^{\sharp}$
and $I^{\natural}$. Hence, if one could iterate our unoriented skein
exact triangle, then one would get spectral sequences from Khovanov
homology to unoriented link Floer homology. We study this in \cite{gn_spectral};
see Subsection \ref{subsec:Future-directions}.

It is interesting to compare our proposal with the existing conjectures
that relate Heegaard Floer homology and $I^{\#},I^{\natural}$. First,
for knots, the counterpart of $I^{\natural}$ that we propose is indeed
$\widehat{HFK}$. Also, for a three-manifold $Y$, $I^{\sharp}(Y)$
is conjecturally isomorphic to $\widehat{HF}(Y)$ over $\mathbb{Q}$,
and a version of our $2$-surgery exact triangle involves $\widehat{HF}(Y_{\lambda}(K))$,
$\widehat{HF}(Y_{\lambda+2\mu}(K))$, and $\widehat{HFK'}(Y,K)$ (over
$\mathbb{F}$), whereas Bhat's $2$-surgery exact triangle \cite{2311.04242}
involves $I^{\sharp}(Y_{\lambda}(K))$, $I^{\sharp}(Y_{\lambda+2\mu}(K))$,
and $I^{\sharp}(Y,K)$.

\subsection{\label{subsec:Future-directions}Future directions}

It is natural to ask whether one can define unoriented link cobordism
maps in unoriented link Floer homology, just like $I^{\sharp}$. We
show in a forthcoming work \cite{gn_spectral} that the band maps
as defined in this paper sometimes do not commute, but also that it
is possible to modify them so that they commute. Using this as a key
step, we iterate a modified version of our unoriented skein exact
triangle and obtain spectral sequences from Khovanov homology to unoriented
link Floer homology in the minus, unreduced hat, and reduced hat versions,
and hence in particular a spectral sequence from reduced Khovanov
homology to knot Floer homology.

It would also be interesting to be able to compute these band maps
(or more ambitiously, the spectral sequence), either using ideas from
bordered Heegaard Floer homology \cite{MR3827056} or grid homology
\cite{MR3381987}. One difficulty in interpreting our maps in grid
homology comes from that we only consider links with exactly two basepoints
on each component. We do not know whether there are similar band maps
in the case where there are more basepoints, such that they form an
unoriented skein exact triangle.

Also, we do not deal with signs in this paper and work over a field
of characteristic $2$. In contrast, Dowlin \cite{MR4777638} has
to divide by $2$. It would be interesting to see whether the theorems
in this paper hold over $\mathbb{Z}$.

There are two different $2$-surgery exact triangles in Heegaard Floer
homology: there is Theorem \ref{thm:2-surgery-intro}, and also Ozsv\'{a}th
and Szab\'{o}'s integer surgery exact triangle \cite[Section 9.5]{MR2113020},
which has two copies of $\boldsymbol{HF}^{-}(Y)$ instead of $\boldsymbol{HFL'}^{-}(Y,K)$
(and the maps are different). In \cite{gn}, we show analogous statements
for any positive rational surgery slope: we get (at least) $p$ different
$p/q$-surgery exact triangles, one of which generalizes Ozsv\'{a}th
and Szab\'{o}'s $n$ or $1/n$-surgery exact triangle. The others
involve algebraic modifications of knot Floer homology; more precisely,
these are variants (if necessary) of $2i/p$-modified knot Floer homology
\cite{1407.1795} for $i=0,1,\cdots,p-1$. Note that $1$-modified
knot Floer homology is unoriented knot Floer homology.

\subsection{Organization}

In Section \ref{sec:Preliminaries}, we start by setting notations
and recalling some definitions in Heegaard Floer theory in Subsection
\ref{subsec:Heegaard-Floer-homology}. In order to prove Theorems
\ref{thm:2-surgery-intro} and \ref{thm:unoriented-exact}, we consider
local systems that a priori involve negative powers of $U$. We define
the type of local systems we consider in Subsection \ref{subsec:unoriented-local-system},
and discuss\emph{ weak admissibility} and \emph{positivity} in Subsection
\ref{subsec:Positivity-and-admissibility}, which ensure that the
differential and (higher) composition maps are well-defined. We use
the language of $A_{\infty}$-categories to simplify the homological
algebra; we recall this in Subsection \ref{subsec:Twisted-complexes}.
We set up notations for standard translates in Subsection \ref{subsec:standard-translates}
and recall theorems on stabilizations in Subsection \ref{subsec:Stabilizations}.
In Subsection \ref{subsec:link-heegaard-diagram}, we clarify our
conventions for link Floer homology, and we define ${\rm Spin}^{c}$-structures
for unoriented links and \emph{strong admissibility} in Subsection
\ref{subsec:spinc-link}. We set up conventions for gradings in Subsection
\ref{subsec:Gradings}, especially when nontrivial local systems are
present. We introduce a simplifying assumption in Subsection \ref{subsec:Homologous-attaching-curves},
in which case checking whether gradings exist is simpler.

In Section \ref{sec:band-maps}, we precisely define the objects we
consider, define the band maps, and show that they are well-defined.
Using these notions, we state the main theorem, the unoriented skein
exact triangle (and the $2$-surgery exact triangle), in Section \ref{sec:An-unoriented-skein}.
In Section \ref{sec:Computation-for-unlinks}, we use the unoriented
skein exact triangle to deduce that band maps for planar links are
the same as Khovanov homology.

In Section \ref{sec:The-plan}, we briefly discuss the key steps of
the proof, and we carry out the local computations in Sections \ref{sec:proof-step1}
and \ref{sec:proof-step3}. Section \ref{sec:proof-step3} is also
the local computation for the $2$-surgery exact triangle. We use
these computations in Section \ref{sec:composition} and finish off
the proof.

\subsection{Acknowledgements}

We thank Peter Ozsv\'{a}th for his continuous support, explaining
a lot of the arguments in this paper, and helpful discussions. We
also thank Ian Zemke for his continuous support, teaching the author
a lot of previous works, especially \cite{2308.15658}, and helpful
discussions. We thank Deeparaj Bhat for sharing his work on the $2$-surgery
exact triangle back in March 2023. We thank John Baldwin, Deeparaj
Bhat, Fraser Binns, Evan Huang, Seungwon Kim, Yonghwan Kim, Jae Hee
Lee, Adam Levine, Jiakai Li, Glen Lim, Ayodeji Lindblad, Robert Lipshitz,
Marco Marengon, Sucharit Sarkar, Zolt\'{a}n Szab\'{o}, Alison Tatsuoka,
Jacob Rasmussen, Joshua Wang, and Fan Ye for helpful discussions.
We also thank Robert Lipshitz, Peter Ozsv\'{a}th, Jacob Rasmussen,
Joshua Wang, and Ian Zemke for helpful comments on earlier drafts
of this paper.

\section{\label{sec:Preliminaries}Preliminaries}

\subsection{\label{subsec:Heegaard-Floer-homology}Heegaard Floer homology and
local systems}

We set up notations and notions that we use throughout this paper.
We use local systems that \emph{a priori} involve negative powers
of $U$, as in \cite{2308.15658}. Although it is not necessary to
consider nontrivial local systems to define the unoriented link Floer
homology groups and the band maps, we will use them to prove the exact
triangles.

We do not need any new analytic input, and thus we will refrain from
discussing the analytic foundations and refer the readers to \cite{MR2113019,MR2113020,MR2240908,MR3509974,2011.00113}.
\begin{defn}[{\cite[Definition 3.1]{MR2443092}}]
Given a closed, oriented, genus $g$ surface $\Sigma$ together with
a set of \emph{basepoints} $\boldsymbol{p}=\{p_{i}\}_{i=1}^{l}$ on
$\Sigma$, an \emph{attaching curve} is a set $\boldsymbol{\alpha}$
of pairwise disjoint, simple closed curves on $\Sigma\backslash\boldsymbol{p}$
whose images span a $g$-dimensional subspace of $H_{1}(\Sigma)$.
A \emph{multi-Heegaard diagram}\footnote{This definition will be slightly modified in Subsection \ref{subsec:unoriented-local-system}.}
is a collection $(\Sigma,\boldsymbol{\alpha}_{0},\cdots,\boldsymbol{\alpha}_{m},\boldsymbol{p})$
of such a pointed surface together with attaching curves. We assume
that attaching curves intersect transversely, and that there are no
triple intersections. 
\end{defn}

\begin{notation}
We write $\alpha$ for \emph{one} alpha circle. The boldsymbol $\boldsymbol{\alpha}$
means a set of alpha circles, and write $\boldsymbol{\alpha}=\{\alpha^{1},\cdots,\alpha^{g}\}$.
If the curves $\alpha^{i}$ lie in $\Sigma$, then, by abuse of notation,
$\boldsymbol{\alpha}$ also means the image of $\mathbb{T}_{\boldsymbol{\alpha}}=\alpha^{1}\times\cdots\times\alpha^{g}$
in ${\rm {\rm Sym}^{g}(\Sigma)}$ and $\alpha^{1}\cup\cdots\cup\alpha^{g}\subset\Sigma$.
\end{notation}

\begin{defn}
Given a multi-Heegaard diagram $(\Sigma,\boldsymbol{\alpha}_{0},\boldsymbol{\alpha}_{1},\cdots,\boldsymbol{\alpha}_{m},\boldsymbol{p})$,
an\emph{ elementary two-chain }is a connected component of $\Sigma\backslash\left(\boldsymbol{\alpha}_{0}\cup\cdots\cup\boldsymbol{\alpha}_{m}\right)$,
and a \emph{two-chain} is a formal sum ${\cal D}$ of elementary two-chains.
A two-chain is \emph{nonnegative} if its local multiplicities are
nonnegative. Denote the \emph{$\boldsymbol{\alpha}_{i}$-boundary
of ${\cal D}$} as $\partial_{\boldsymbol{\alpha}_{i}}{\cal D}$ (which
is a one-chain). A \emph{cornerless two-chain} is a two-chain ${\cal D}$
for which $\partial_{\boldsymbol{\alpha}_{i}}{\cal D}$ is a cycle
for all $i$. A cornerless two-chain is \emph{periodic} if it does
not contain any basepoints.
\end{defn}

\begin{defn}
A \emph{domain} is a two-chain ${\cal D}$ together with an ordered
sequence of \emph{vertices} ${\bf x}_{0},\cdots,{\bf x}_{k}$ where
$k\ge1$ and ${\bf x}_{j}\in\boldsymbol{\alpha}_{i_{j}}\cap\boldsymbol{\alpha}_{i_{j+1}}$
for $j=0,\cdots,k$ ($i_{k+1}=i_{0}$), such that 
\[
\partial(\partial_{\boldsymbol{\alpha}_{i_{j}}}{\cal D})={\bf x}_{j+1}-{\bf x}_{j},\ \partial_{\boldsymbol{\alpha}_{i}}{\cal D}=0
\]
for $j=0,\cdots,k$ and $i\neq i_{0},\cdots,i_{k}$\footnote{Not all two-chains \emph{lift} to a domain.}.
Let $D({\bf x}_{0},\cdots,{\bf x}_{k})$ be the set of \emph{domains
with vertices ${\bf x}_{0},\cdots,{\bf x}_{k}$}. 

An \emph{$\boldsymbol{\alpha}_{i_{0}}\cdots\boldsymbol{\alpha}_{i_{k}}$-domain}
is a domain with vertices ${\bf x}_{0},\cdots,{\bf x}_{k}$ for some
intersection points ${\bf x}_{j}\in\boldsymbol{\alpha}_{i_{j}}\cap\boldsymbol{\alpha}_{i_{j+1}}$,
and a\emph{ domain with $k+1$ vertices }is an $\boldsymbol{\alpha}_{i_{0}}\cdots\boldsymbol{\alpha}_{i_{k}}$-domain
for some $i_{0},\cdots,i_{k}$.
\end{defn}

By abuse of notation, we also denote domains as ${\cal D}$. Also,
by abuse of notion, we say a two-chain ${\cal D}$ is in $D({\bf x}_{0},\cdots,{\bf x}_{k})$
if there exists a domain in $D({\bf x}_{0},\cdots,{\bf x}_{k})$ whose
two-chain is ${\cal D}$, and we may identify ${\cal D}$ with the
corresponding domain.
\begin{rem}
Note that the two-chain of a domain uniquely determines the domain
if the domain has at least three vertices. This may not be the case
otherwise: for instance, given any cornerless two-chain whose boundary
lies in $\boldsymbol{\alpha}\cup\boldsymbol{\beta}$ for some $\boldsymbol{\alpha},\boldsymbol{\beta}$,
then for any ${\bf x}\in\boldsymbol{\alpha}\cap\boldsymbol{\beta}$,
there exists a domain in $D({\bf x},{\bf x})$ whose underlying two-chain
is the given two-chain\footnote{A cornerless two-chain can be lifted to a domain if and only if it
has exactly two vertices (which are identical).}. (These are not the only examples.)
\end{rem}

The Maslov index $\mu({\cal D})$ of a domain ${\cal D}$ can be computed
(and/or defined) combinatorially, using the formulas in \cite[Section 4]{MR2240908}
and \cite{MR2811652}. Note that $\mu({\cal D})$, in general, cannot
be defined just from the underlying two-chain of ${\cal D}$\footnote{Similarly, the monodromy (Definition \ref{def:hf-chain-cpx}) also
depends on the vertices of ${\cal D}$.}.

It is shown in \cite[Theorems 3.2, 3.3]{MR2811652} that $\mu({\cal D})$
is cyclically symmetric in the vertices of ${\cal D}$, and that $\mu$
is additive, i.e. $\mu({\cal D}_{1})+\mu({\cal D}_{2})=\mu({\cal D}_{1}+{\cal D}_{2})$,
with respect to the following composition of domains: given ${\cal D}_{1}\in D({\bf x}_{a},\cdots,{\bf x}_{a+b-1},{\bf y})$
and ${\cal D}_{2}\in D({\bf x}_{0},\cdots,{\bf x}_{a-1},{\bf y},{\bf x}_{a+b},\cdots,{\bf x}_{k})$
($a\ge0,b\ge1$), their composition is ${\cal D}_{1}+{\cal D}_{2}\in D({\bf x}_{0},\cdots,{\bf x}_{k})$.

To define the Heegaard Floer chain complex, we need to choose a coefficient
ring and assign weights to basepoints. We also equip attaching curves
with local systems.
\begin{defn}
If $R$ is a power series ring $R=\mathbb{F}\llbracket X_{1},\cdots,X_{n}\rrbracket$,
denote 
\[
R^{\infty}=\mathbb{F}\llbracket X_{1},\cdots,X_{n}\rrbracket[X_{1}^{-1},\cdots,X_{n}^{-1}].
\]
Note that this depends on the identification $R=\mathbb{F}\llbracket X_{1},\cdots,X_{n}\rrbracket$.
For an $R=\mathbb{F}\llbracket X_{1},\cdots,X_{n}\rrbracket$-module
$E$, write $E^{\infty}:=E\otimes_{R}R^{\infty}.$

A \emph{coefficient ring} is a power series ring $R=\mathbb{F}\llbracket X_{1},\cdots,X_{n}\rrbracket$,
its quotient, or $R^{\infty}$.
\end{defn}

\begin{defn}
If $\boldsymbol{p}$ is the set of all basepoints and $R$ is a coefficient
ring, then  a \emph{weight function} is a function $w:\boldsymbol{p}\to R$.
Given a weight function $w$, define the \emph{weight of a two-chain
${\cal D}$} as 
\[
w({\cal D}):=\prod_{p\in\boldsymbol{p}}w(p)^{n_{p}({\cal D})}.
\]
\end{defn}

\begin{defn}
A \emph{($R$-)local system on an attaching curve $\boldsymbol{\alpha}$}
is a pair $(E,\Phi)$ where $E$ is a free $R$-module, and $\Phi$
is the \emph{monodromy}: it is a groupoid representation 
\[
\Phi:\Pi_{1}(\boldsymbol{\alpha})\to{\rm Hom}_{R}(E,E),
\]
where $\Pi_{1}(\boldsymbol{\alpha})$ is the fundamental groupoid
of $\boldsymbol{\alpha}=\mathbb{T}_{\boldsymbol{\alpha}}$, and ${\rm Hom}_{R}(E,E)$
is a groupoid with one object whose automorphism group is ${\rm Hom}_{R}(E,E)$.
Write $\boldsymbol{\alpha}^{(E,\Phi)}$, or simply $\boldsymbol{\alpha}^{E}$,
to signify that $\boldsymbol{\alpha}$ is equipped with the local
system $(E,\Phi)$.
\end{defn}

The local systems in this paper will all be specified by an oriented
arc $G$ on the Heegaard surface that satisfies certain conditions.
We define this in Subsection \ref{subsec:unoriented-local-system}
and also modify the definition of a Heegaard diagram to a collection
$(\Sigma,\boldsymbol{\alpha}_{0},\cdots,\boldsymbol{\alpha}_{m},\boldsymbol{p})$
or $(\Sigma,\boldsymbol{\alpha}_{0},\cdots,\boldsymbol{\alpha}_{m},\boldsymbol{p},G)$.
\begin{defn}
A \emph{Heegaard datum} is a tuple that consists of a multi-Heegaard
diagram $(\Sigma,\boldsymbol{\alpha}_{0},\cdots,\boldsymbol{\alpha}_{m},\boldsymbol{p})$,
a coefficient ring $R$, a weight function $w:\boldsymbol{p}\to R$,
and local systems $(E_{i},\Phi_{i})$ on $\boldsymbol{\alpha}_{i}$
for $i=0,\cdots,m$.
\end{defn}

Later, we will define properties of Heegaard diagrams. A Heegaard
datum is said to have a given property if its underlying Heegaard
diagram possesses that property.
\begin{defn}
\label{def:hf-chain-cpx}Given a Heegaard datum with Heegaard diagram
$(\Sigma,\boldsymbol{\alpha},\boldsymbol{\beta},\boldsymbol{p})$,
coefficient ring $R$, weight function $w:\boldsymbol{p}\to R$, and
local systems $(E_{\boldsymbol{\alpha}},\Phi_{\boldsymbol{\alpha}}),(E_{\boldsymbol{\beta}},\Phi_{\boldsymbol{\beta}})$
on $\boldsymbol{\alpha},\boldsymbol{\beta}$, respectively, define
the group $CF(\boldsymbol{\alpha}^{E_{\boldsymbol{\alpha}}},\boldsymbol{\beta}^{E_{\boldsymbol{\beta}}})$
as a direct sum of ${\rm Hom}_{R}(E_{\boldsymbol{\alpha}},E_{\boldsymbol{\beta}})$'s:
\[
CF_{R,w}(\boldsymbol{\alpha}^{E_{\boldsymbol{\alpha}}},\boldsymbol{\beta}^{E_{\boldsymbol{\beta}}})=\bigoplus_{{\bf x}\in\boldsymbol{\alpha}\cap\boldsymbol{\beta}}{\rm Hom}_{R}(E_{\boldsymbol{\alpha}},E_{\boldsymbol{\beta}}){\bf x},
\]
and define the differential $\partial$ as the $R$-linear map such
that 
\begin{equation}
\partial(e{\bf x})=\sum_{{\cal D}\in D({\bf x},{\bf y}),\ \mu({\cal D})=1}\#{\cal M}({\cal D})w({\cal D})\rho({\cal D})(e){\bf y},\label{eq:differential}
\end{equation}
where $e\in{\rm Hom}_{R}(E_{\boldsymbol{\alpha}},E_{\boldsymbol{\beta}})$
and $\rho({\cal D}):{\rm Hom}_{R}(E_{\boldsymbol{\alpha}},E_{\boldsymbol{\beta}})\to{\rm Hom}_{R}(E_{\boldsymbol{\alpha}},E_{\boldsymbol{\beta}})$
is the \emph{monodromy of ${\cal D}$}: 
\[
\rho({\cal D})=\Phi_{\boldsymbol{\beta}}(\partial_{\boldsymbol{\beta}}{\cal D})\circ e\circ\Phi_{\boldsymbol{\alpha}}(\partial_{\boldsymbol{\alpha}}{\cal D}).
\]
\end{defn}

We refer to elements of $CF_{R,w}(\boldsymbol{\alpha}^{E_{\boldsymbol{\alpha}}},\boldsymbol{\beta}^{E_{\boldsymbol{\beta}}})$
as \emph{morphisms} from $\boldsymbol{\alpha}^{E_{\boldsymbol{\alpha}}}$
to $\boldsymbol{\beta}^{E_{\boldsymbol{\beta}}}$. In most cases,
we omit the subscripts $R$ and $w$. Also, we often decorate the
symbol $CF$: for instance, $\boldsymbol{CF}^{-}$ means that we are
working over a power series ring, and $\widehat{CF}$ means that we
are working over some quotient of the power series ring. We will mostly
focus on the $\boldsymbol{CF}^{-}$ case.

We get (higher) composition maps $\mu_{n}$\footnote{The differential $\partial=\mu_{1}$ is also a composition map.}
when we have many attaching curves. If $\boldsymbol{\alpha}_{0}^{(E_{0},\Phi_{0})},\cdots,\boldsymbol{\alpha}_{d}^{(E_{d},\Phi_{d})}$
are attaching curves with local systems, then the composition map
is given by 
\begin{multline}
\mu_{d}:CF(\boldsymbol{\alpha}_{0}^{E_{0}},\boldsymbol{\alpha}_{1}^{E_{1}})\otimes\cdots\otimes CF(\boldsymbol{\alpha}_{d-1}^{E_{d-1}},\boldsymbol{\alpha}_{d}^{E_{d}})\to CF(\boldsymbol{\alpha}_{0}^{E_{0}},\boldsymbol{\alpha}_{d}^{E_{d}}):\\
e_{1}{\bf x}_{1}\otimes\cdots\otimes e_{d}{\bf x}_{d}\mapsto\sum_{{\cal D}\in D({\bf x}_{1},\cdots,{\bf x}_{d},{\bf y}),\ \mu({\cal D})=2-d}\#{\cal M}({\cal D})w({\cal D})\rho({\cal D})(e_{1}\otimes\cdots\otimes e_{d}){\bf y},\label{eq:higher-multiplication}
\end{multline}
where $\rho({\cal D})$ is the \emph{monodromy} defined as follows:
given a collection of paths $p_{j}:{\bf x}_{j}\to{\bf x}_{j+1}$ in
$\boldsymbol{\alpha}_{j}$ ($j=0,\cdots,d$) for intersection points
${\bf x}_{j}\in\boldsymbol{\alpha}_{j-1}\cap\boldsymbol{\alpha}_{j}$
($j=0,\cdots,d$, $\boldsymbol{\alpha}_{-1}:=\boldsymbol{\alpha}_{d}$),
define the\emph{ monodromy of $p_{0},\cdots,p_{d}$} as 
\begin{multline*}
\rho(p_{0},\cdots,p_{d}):{\rm Hom}_{R}(E_{0},E_{1})\otimes\cdots\otimes{\rm Hom}_{R}(E_{d-1},E_{d})\to{\rm Hom}_{R}(E_{0},E_{d}):\\
e_{1}\otimes\cdots\otimes e_{d}\mapsto\Phi_{d}(p_{d})\circ e_{d}\circ\Phi_{d-1}(p_{d-1})\circ e_{d-1}\circ\cdots\circ\Phi_{1}(p_{1})\circ e_{1}\circ\Phi_{0}(p_{0}).
\end{multline*}
Given an $\boldsymbol{\alpha}_{0},\cdots,\boldsymbol{\alpha}_{d}$-domain
${\cal D}$, its \emph{monodromy} is $\rho({\cal D})=\rho(\partial_{\boldsymbol{\alpha}_{0}}({\cal D}),\cdots,\partial_{\boldsymbol{\alpha}_{d}}({\cal D}))$.
\begin{rem}
We need the Heegaard datum to satisfy an \emph{admissibility criterion}
(Subsubsection \ref{subsec:Admissibility}) for the sum in the definition
of $\mu_{n}$ to be well-defined. Also, the groupoid representations
we consider in this paper will \emph{a priori} have negative powers
of $U$: indeed, they will be maps into ${\rm Hom}_{R}(E_{\boldsymbol{\alpha}}^{\infty},E_{\boldsymbol{\beta}}^{\infty})\simeq{\rm Hom}_{R}(E_{\boldsymbol{\alpha}},E_{\boldsymbol{\beta}})\otimes_{R}R^{\infty}$.
We show in Subsection \ref{subsec:Positivity} that this is not a
problem for the Heegaard data that we consider in this paper (see
Subsection \ref{subsec:unoriented-local-system}).
\end{rem}

We need to choose a family of almost complex structure (on ${\rm Sym}^{d}(\Sigma)$,
or on $\Sigma\times D_{n}$, where $D_{n}$ is a disk with $n$ punctures,
if we work in the cylindrical reformulation) to define the Heegaard
Floer chain complex and the composition maps. In this paper, we will
have to work in various different (families of) almost complex structures
(for instance, see Section \ref{sec:The-plan}). Using the arguments
of \cite[Section 10c]{MR2441780} and \cite[Section 3.3]{MR3509974},
it is possible to show that a generic one-parameter family that interpolates
two such (families of) almost complex structures induces an $A_{\infty}$-functor
between the $A_{\infty}$-categories given by the two choices, and
that this $A_{\infty}$-functor induces isomorphisms on the homology
groups. (See Remark \ref{rem:ainf-pedantry} for a definition of the
$A_{\infty}$-categories we consider.)
\begin{rem}
\label{rem:orientations}We use the convention that almost complex
structures on surfaces rotate the tangent spaces in a counterclockwise
direction. This is the same as \cite{MR2350128} but is different
from \cite{MR2964628,MR3604486}.
\end{rem}

\subsection{\label{subsec:unoriented-local-system}Heegaard data for unoriented
link Floer homology}

In this subsection, we describe the kinds of Heegaard data that we
consider in this paper. The local systems will be specified by an
oriented arc $G\subset\Sigma$ on the Heegaard surface. Hence, instead
of including local systems as a part of a Heegaard datum, we simply
incorporate $G$ into the Heegaard diagram: modify the definition
of a multi-Heegaard diagram to be a collection of a pointed surface,
attaching curves, and \emph{possibly also an oriented arc $G\subset\Sigma$}
(this will always be drawn dashed and grey). Also, the coefficient
rings and the weight functions will mostly be determined by the Heegaard
diagram.

\subsubsection*{The Heegaard diagram, ignoring $G$}

Let $\boldsymbol{p}$ be the set of basepoints. For any two attaching
curves $\boldsymbol{\alpha}$ and $\boldsymbol{\beta}$, the sub Heegaard
diagram $(\Sigma,\boldsymbol{\alpha},\boldsymbol{\beta},\boldsymbol{p})$
represents a pointed link in a pointed three-manifold, where there
are exactly two basepoints in each link component. Equivalently, there
exists an \emph{equivalence relation} on $\boldsymbol{p}$ where the
equivalence classes have size $1$ or $2$, such that for each attaching
curve $\boldsymbol{\gamma}$ and each connected component ${\cal D}$
of $\Sigma\backslash\boldsymbol{\gamma}$, the intersection ${\cal D}\cap\boldsymbol{p}$
is an equivalence class.

\subsubsection*{The oriented arc $G$}

The negative boundary $\partial_{-}G$ of $G$ is a basepoint. For
any attaching curve $\boldsymbol{\alpha}$, there is at most one circle
$\alpha^{i}\in\boldsymbol{\alpha}$ that intersects $G$, and the
boundary $\{\partial_{+}G,\partial_{-}G\}$ of $G$ is contained in
the same connected component of $\Sigma\backslash\boldsymbol{\alpha}$.
Also, we assume that there are at most two circles that intersect
$G$, and that they are standard, small translates\footnote{These two circles intersect each other at two points, and intersect
other circles in the ``same way'': see Definition \ref{def:standard-translate}.} of each other, that intersect each other at two points, right above
and below $G$. See Figure \ref{fig:local-system} for a schematic
near $G$.

\subsubsection*{The coefficient ring $R$}

If $G$ does not exist, $R=\mathbb{F}\llbracket X_{1},\cdots,X_{n}\rrbracket$;
if $G$ exists, $R=\mathbb{F}\llbracket U,X_{1},\cdots,X_{n}\rrbracket$,
where $X_{i}=U_{i}$ or $U_{i}^{1/2}$. We call $U$ the \emph{distinguished
variable}. Sometimes, we consider $R^{\infty}$ or a quotient of $R$
by some $U_{i}$ or $U_{i}^{1/2}$'s.

\subsubsection*{The weight function}

If $G$ exists, then $\partial_{-}G$ has weight $U$. If $X_{i}=U_{i}^{1/2}$,
then there are exactly two basepoints with weight $U_{i}^{1/2}$,
and they belong to the same equivalence class of $\boldsymbol{p}$.
We call these basepoints \emph{link basepoints}. If $X_{i}=U_{i}$,
then there is exactly one basepoint with weight $U_{i}$, and it forms
a (singleton) equivalence class of $\boldsymbol{p}$. We call these
\emph{free basepoints}\footnote{The basepoint $\partial_{-}G$ is neither a link basepoint nor a free
basepoint.}. These are all the basepoints.

\subsubsection*{The local systems}

If an attaching curve $\boldsymbol{\alpha}$ does not intersect $G$,
then it has the trivial local system. If some $\alpha^{i}\in\boldsymbol{\alpha}$
intersects $G$, then we consider the local system $(E,\phi)$ where
$E$ is the rank $2$ free $R$-module $E=e_{0}R\oplus e_{1}R$, and
the monodromy sends a path $p$ in $\boldsymbol{\alpha}$ to $\phi^{\#(G\cap p)}$,
where 
\[
\phi=e_{1}e_{0}^{\ast}+Ue_{0}e_{1}^{\ast}\in{\rm Hom}_{R}(E^{\infty},E^{\infty}).
\]
The intersection $G\cap p$ is positive if $p$ goes from top to bottom
in Figure \ref{fig:local-system}.

\begin{figure}[h]
\begin{centering}
\includegraphics[scale=1.5]{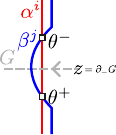}
\par\end{centering}
\caption{\label{fig:local-system}A local diagram for the Heegaard diagram
near $G$}
\end{figure}

\begin{defn}
We chose a basis for the nontrivial local system $E$. Choose the
trivial basis $\{1\}$ for the trivial local system. These induce
$R$-bases on the $R$-Hom spaces. An \emph{$R$-generator} (or simply
a \emph{generator}) of $\boldsymbol{CF}^{-}(\boldsymbol{\alpha}^{E_{\boldsymbol{\alpha}}},\boldsymbol{\beta}^{E_{\boldsymbol{\beta}}})$
is of the form $f{\bf x}$, where $f$ is an $R$-basis element of
${\rm Hom}_{R}(E_{\boldsymbol{\alpha}},E_{\boldsymbol{\beta}})$ and
${\bf x}\in\boldsymbol{\alpha}\cap\boldsymbol{\beta}$.
\end{defn}

\begin{defn}
\label{def:total-multiplicity}Define the \emph{total multiplicity
of a two-chain ${\cal D}$ as 
\[
P({\cal D}):=n_{link}({\cal D})+2n_{free}({\cal D}),
\]
where} $link$ is the set of link basepoints together with $\partial_{+}G,\partial_{-}G$
(if $G$ exists), and $free$ is the set of all free basepoints. 
\end{defn}

\begin{example}
\label{exa:local-system}Let us consider the first two genus $1$
Heegaard diagrams in Figure \ref{fig:local-system-ex}, where we work
over $\mathbb{F}\llbracket U\rrbracket$, and $U$ is the distinguished
variable.

Let us consider the first diagram. The attaching curve $\alpha$ has
a nontrivial local system, but $\beta$'s local system is trivial.
The chain complex $\boldsymbol{CF}^{-}(\alpha^{E_{\alpha}},\beta)$
is freely generated by $e_{0}a,e_{1}a,e_{0}b,e_{1}b,e_{0}c,e_{1}c$
as an $\mathbb{F}\llbracket U\rrbracket$-module. The differential
acts on the generators by 
\[
e_{0}a\mapsto e_{0}b,\ e_{1}a\mapsto e_{1}b,\ e_{0}a\mapsto e_{1}c,\ e_{1}a\mapsto Ue_{0}c,
\]
where the first two are given by the two-chain $D_{1}$ and the last
two are given by $D_{2}$. Note that if we identify $e_{0}x$ with
$x$ and $e_{1}x$ with $U^{1/2}x$ for $x=a,b,c$, then this chain
complex is isomorphic to the chain complex $\boldsymbol{CF}^{-}(\alpha,\beta)$
given by the third diagram, where we work over $\mathbb{F}\llbracket U^{1/2}\rrbracket$,
and where we assign the weight $U^{1/2}$ to both $z$ and $w$.

Let us consider the second diagram. Both attaching curves $\alpha$
and $\beta$ have a nontrivial local system. The module $\boldsymbol{CF}^{-}(\alpha^{E_{\alpha}},\beta^{E_{\beta}})$
is freely generated by $e_{i}e_{j}^{\ast}\theta^{+},e_{i}e_{j}^{\ast}\theta^{-}$
($i,j=0,1$) as an $\mathbb{F}\llbracket U\rrbracket$-module. The
differential is as follows:
\[
\begin{gathered}e_{0}e_{0}^{\ast}\theta^{+}\mapsto e_{1}e_{1}^{\ast}\theta^{-}+e_{0}e_{0}^{\ast}\theta^{-},\ e_{1}e_{1}^{\ast}\theta^{+}\mapsto e_{0}e_{0}^{\ast}\theta^{-}+e_{1}e_{1}^{\ast}\theta^{-},\\
e_{1}e_{0}^{\ast}\theta^{+}\mapsto Ue_{0}e_{1}^{\ast}\theta^{-}+e_{1}e_{0}^{\ast}\theta^{-},\ e_{0}e_{1}^{\ast}\theta^{+}\mapsto U^{-1}e_{1}e_{0}^{\ast}\theta^{-}+e_{0}e_{1}^{\ast}\theta^{-}.
\end{gathered}
\]
In each case, the first summand is the contribution of $D_{1}$; the
second summand is that of $D_{2}$. Note that the differential is
well defined on $\boldsymbol{CF}^{\infty}(\alpha^{E_{\alpha}},\beta^{E_{\beta}})$,
but \emph{not} well defined on $\boldsymbol{CF}^{-}(\alpha^{E_{\alpha}},\beta^{E_{\beta}})$
as we get a negative power of $U$. Our solution (see Subsubsection
\ref{subsec:Positivity}) will be to consider the free submodule $\boldsymbol{CF}_{fil}^{-}(\alpha^{E_{\alpha}},\beta^{E_{\beta}})\le\boldsymbol{CF}^{-}(\alpha^{E_{\alpha}},\beta^{E_{\beta}})$
generated by $Ue_{0}e_{1}^{\ast}\theta^{+}$ together with all $e_{i}e_{j}^{\ast}\theta^{+},e_{i}e_{j}^{\ast}\theta^{-}$
for $i,j=0,1$ except $e_{0}e_{1}^{\ast}\theta^{+}$ (i.e. replace
$e_{0}e_{1}^{\ast}\theta^{+}$ by $Ue_{0}e_{1}^{\ast}\theta^{+}$).
Then, $\boldsymbol{CF}_{fil}^{-}(\alpha^{E_{\alpha}},\beta^{E_{\beta}})$
will be a sub chain complex of $\boldsymbol{CF}^{\infty}(\alpha^{E_{\alpha}},\beta^{E_{\beta}})$.
\begin{figure}[h]
\begin{centering}
\includegraphics[scale=1.5]{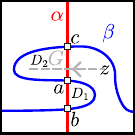}\qquad{}\includegraphics[scale=1.5]{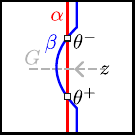}\qquad{}\includegraphics[scale=1.5]{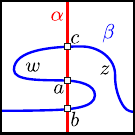}
\par\end{centering}
\caption{\label{fig:local-system-ex}Some genus 1 examples}
\end{figure}
\end{example}

\begin{rem}
\label{rem:same-as-unoriented}One can define the composition maps
$\mu_{n}$ without mentioning local systems, if there is at most one
attaching curve with a nontrivial local system. Let us demonstrate
this for $\mu_{1}$ (we have observed a special case in the first
example of Example \ref{exa:local-system}). Assume that exactly one
of $\boldsymbol{\alpha}$ and $\boldsymbol{\beta}$ has a nontrivial
local system, and let the coefficient ring be $S=\mathbb{F}\llbracket U^{1/2},X_{1},X_{2},\cdots,X_{k}\rrbracket$.
Then $\boldsymbol{CF}_{S}^{-}(\boldsymbol{\alpha}^{E_{\boldsymbol{\alpha}}},\boldsymbol{\beta}^{E_{\boldsymbol{\beta}}})$
is $S$-linearly isomorphic to the chain complex $\boldsymbol{CF}_{R}^{-}(\boldsymbol{\alpha},\boldsymbol{\beta})$,
where $R=\mathbb{F}\llbracket U^{1/2},X_{1},X_{2},\cdots,X_{k}\rrbracket$,
and we add $w:=\partial_{+}G$ to the set of basepoints, assign the
weight $U^{1/2}$ to $w$ and $z$, and assign the same weights as
before to all the basepoints except $w$ and $z$. The $S$-linear
chain isomorphism is given by mapping $e_{0}{\bf x}\mapsto{\bf x}$,
$e_{1}{\bf x}\mapsto U^{1/2}{\bf x}$ if $E_{\boldsymbol{\beta}}$
is nontrivial; and $e_{1}^{\ast}{\bf x}\mapsto{\bf x}$, $e_{0}^{\ast}{\bf x}\mapsto U^{1/2}{\bf x}$
if $E_{\boldsymbol{\alpha}}$ is nontrivial.

We will see an instance of this for $\mu_{2}$ in Subsubsection \ref{subsec:Merge-and-split}:
compare Remark \ref{rem:band-non-trivial-local-system} and the discussion
right before it.
\end{rem}

\subsection{\label{subsec:Positivity-and-admissibility}Weak admissibility and
positivity}

We mainly work in the minus version. Recall that in Subsection \ref{subsec:unoriented-local-system},
we have specialized to specific kinds of Heegaard data, which involved
local systems that a priori involve negative powers of $U$. Under
these assumptions, we first deal with \emph{admissibility}, which
ensures that the sums in Equations (\ref{eq:differential}) and (\ref{eq:higher-multiplication})
are well-defined in $\boldsymbol{CF}^{\infty}$.

We then deal with \emph{positivity}, i.e. we show that our maps only
involve nonnegative powers of $U$. More precisely, we define a $\boldsymbol{CF}^{-}$-type
sub \emph{chain complex} of $\boldsymbol{CF}^{\infty}$, which we
call the space of \emph{filtered maps} $\boldsymbol{CF}_{fil}^{-}$
as in \cite{2308.15658}, on which the other composition maps ($\mu_{n}$
for $n\ge2$) are defined as well. By $\boldsymbol{CF}^{-}$-type,
we mean that $\boldsymbol{CF}_{fil}^{-}$ is a (free) sub $R$-module
of $\boldsymbol{CF}^{-}$, and that $R^{\infty}\boldsymbol{CF}_{fil}^{-}=\boldsymbol{CF}^{\infty}$.
In our case, positivity is ensured because we have \emph{at most two}
attaching curves with nontrivial local systems\footnote{When we prove the exact triangle, we need to consider two, and fortunately
only two.}. If at most one of $E_{\boldsymbol{\alpha}}$ or $E_{\boldsymbol{\beta}}$
is a trivial local system, then we will have $\boldsymbol{CF}_{fil}^{-}(\boldsymbol{\alpha}^{E_{\boldsymbol{\alpha}}},\boldsymbol{\beta}^{E_{\boldsymbol{\beta}}})=\boldsymbol{CF}^{-}(\boldsymbol{\alpha}^{E_{\boldsymbol{\alpha}}},\boldsymbol{\beta}^{E_{\boldsymbol{\beta}}})$.

We first record a useful lemma.
\begin{defn}
Define a grading
\[
{\rm gr}_{U}:{\rm Hom}_{R}(A,B)\backslash\{0\}\to\frac{1}{2}\mathbb{Z}
\]
for $A,B\in\{E^{\infty},R^{\infty}\}$ as follows.

First, for $r\in R^{\infty}\backslash\{0\}$, let ${\rm gr}_{U}(r)$
be the largest integer $n$ such that $r\in U^{n}R[X_{1}^{-1},\cdots,X_{k}^{-1}]$,
and define ${\rm gr}_{U}:{\rm Hom}_{R}(R^{\infty},R^{\infty})\backslash\{0\}\to\mathbb{Z}$
by identifying ${\rm Hom}_{R}(R^{\infty},R^{\infty})\simeq R^{\infty}$.

Then, define ${\rm gr}_{U}$ for the basis elements $e_{0},e_{1},e_{0}^{\ast},e_{1}^{\ast},e_{0}e_{0}^{\ast},\cdots$
such that it is multiplicative and ${\rm gr}_{U}e_{0}=0$, ${\rm gr}_{U}e_{1}=1/2$
(for instance, ${\rm gr}_{U}e_{1}^{\ast}=-1/2$). For $\sum r_{i}f_{i}\in{\rm Hom}_{R}(A,B)$
where $f_{i}$ are distinct basis elements and $r_{i}\in R^{\infty}\backslash\{0\}$,
let
\[
{\rm gr}_{U}\left(\sum r_{i}f_{i}\right):=\min_{i}({\rm gr}_{U}(r_{i}){\rm gr}_{U}(f_{i})).
\]
\end{defn}

\begin{figure}[h]
\begin{centering}
\includegraphics[scale=1.5]{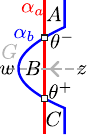}
\par\end{centering}
\caption{\label{fig:local-system-lemma}A local diagram for Lemmas \ref{lemma:deg-output}
and \ref{lem:positivity}}
\end{figure}

\begin{lem}
\label{lemma:deg-output}Let $\boldsymbol{\alpha}_{0}^{E_{0}},\cdots,\boldsymbol{\alpha}_{d}^{E_{d}}$
be attaching curves with local systems. Consider generators $f_{i}{\bf x}_{i}\in\boldsymbol{CF}^{-}(\boldsymbol{\alpha}_{i-1}^{E_{i-1}},\boldsymbol{\alpha}_{i}^{E_{i}})$
for $i=1,\cdots,d$ and ${\bf y}\in\boldsymbol{\alpha}_{0}\cap\boldsymbol{\alpha}_{d}$.
Assume that there is a domain ${\cal D}\in D({\bf x}_{1},\cdots,{\bf x}_{d},{\bf y})$,
and let 
\[
g=w({\cal D})\rho({\cal D})(f_{1}\otimes\cdots\otimes f_{d}).
\]
If $g\neq0$, then 
\[
{\rm gr}_{U}g=\frac{1}{2}(n_{w}({\cal D})+n_{z}({\cal D}))+\sum{\rm gr}_{U}f_{i},
\]
where $w=\partial_{+}G$ and $z=\partial_{-}G$.
\end{lem}

\begin{proof}
This follows from inspecting Figure \ref{fig:local-system-lemma}.
Note that for this claim, we do not need that only at most two of
$\boldsymbol{\alpha}_{0},\cdots,\boldsymbol{\alpha}_{d}$ have nontrivial
local systems.
\end{proof}

\subsubsection{\label{subsec:Admissibility}Weak admissibility}

There are two kinds of admissibility conditions: \emph{strong admissibility}
and \emph{weak admissibility}. Strong admissibility lets us work with
polynomial coefficient rings by ensuring that for each $m$, there
are only finitely many domains with Maslov index $m$, if we restrict
to a specific ${\rm Spin}^{c}$-structure. Weak admissibility lets
us work with power series rings $R=\mathbb{F}\llbracket Y_{1},\cdots,Y_{k}\rrbracket$
by ensuring that there are only finitely many domains that contribute
when the coefficient ring is $R/(Y_{1}^{N},\cdots,Y_{k}^{N})$, for
each $N$\footnote{We have not discussed positivity yet, so to be precise, we can only
talk about $\boldsymbol{CF}^{\infty}$; in this case, weak admissibility
ensures a finite count for each $N$, if the coefficient \emph{module}
is $R^{\infty}/(Y_{1}^{N}R+\cdots+Y_{k}^{N}R)$.}.

We discuss strong admissibility when there are two attaching curves
in Subsubsection \ref{subsec:Strong-admissibility-2}; for more than
two attaching curves, we only study a restrictive case in Subsubsection
\ref{subsec:Strong-admissibility-c}. We can work with either weak
or strong admissibility for the purpose of this paper if we only consider
two attaching curves, but when there are more than two, we will mainly
work with weak admissibility, since in some cases, we have to consider
the sum of infinitely many polygon maps, as in \cite[Section 9]{MR2113020}.
\begin{defn}
Let $\boldsymbol{q}$ be the set of basepoints together with $\partial_{+}G$.
The Heegaard diagram is \emph{weakly admissible} if every cornerless
two-chain ${\cal D}$ such that $n_{\boldsymbol{q}}({\cal D})=0$
has both positive and negative local multiplicities.
\end{defn}

\begin{lem}
If the Heegaard datum is weakly admissible, then the sums in Equations
(\ref{eq:differential}) and (\ref{eq:higher-multiplication}) are
well-defined for $\boldsymbol{CF}^{\infty}$.
\end{lem}

\begin{proof}
Fix generators $f_{i}{\bf x}_{i}\in\boldsymbol{CF}^{\infty}(\boldsymbol{\alpha}_{i-1}^{E_{i-1}},\boldsymbol{\alpha}_{i}^{E_{i}})$
for $i=1,\cdots,d$ and ${\bf y}\in\boldsymbol{\alpha}_{0}\cap\boldsymbol{\alpha}_{d}$.
For each $M\in\mathbb{Z}$, we would like to show that there are only
finitely many nonnegative domains ${\cal D}\in D({\bf x}_{1},\cdots,{\bf x}_{d},{\bf y})$
such that 
\[
g\notin U^{M}{\rm Hom}_{R}(E_{0},E_{d})+X_{1}^{M}{\rm Hom}_{R}(E_{0},E_{d})+\cdots+X_{k}^{M}{\rm Hom}_{R}(E_{0},E_{d}),
\]
where $g=w({\cal D})\rho({\cal D})(f_{1}\otimes\cdots\otimes f_{d})$\footnote{Here, identify ${\rm Hom}_{R}(E_{i}^{\infty},E_{j}^{\infty})\simeq{\rm Hom}_{R}(E_{i},E_{j})\otimes_{R}R^{\infty}$.}.

The proof of \cite[Lemma 4.13]{MR2113019} shows that for each $N$,
there are only finitely many nonnegative $\boldsymbol{\alpha}_{0},\cdots,\boldsymbol{\alpha}_{d}$-domains
${\cal D}$ such that for all $q\in\boldsymbol{q}$, we have $n_{q}({\cal D})\le N$.
Hence the above follows, from Lemma \ref{lemma:deg-output} and that
$g\in U^{\left\lfloor {\rm gr}_{U}g\right\rfloor }{\rm Hom}_{R}(E_{0},E_{d})$.
\end{proof}
We can now define $\boldsymbol{CF}^{\infty}(\boldsymbol{\alpha}^{E_{\boldsymbol{\alpha}}},\boldsymbol{\beta}^{E_{\boldsymbol{\beta}}})$
and show that it is a chain complex. Similarly, one can check that
the higher $A_{\infty}$-relations hold: these are simpler since there
are no contributions from boundary degenerations.
\begin{lem}
\label{lem:d2=00003D0}Let us consider a weakly admissible Heegaard
datum with attaching curves with local systems $\boldsymbol{\alpha}^{E_{\boldsymbol{\alpha}}},\boldsymbol{\beta}^{E_{\boldsymbol{\beta}}}$.
The map $\partial$ on $\boldsymbol{CF}^{\infty}(\boldsymbol{\alpha}^{E_{\boldsymbol{\alpha}}},\boldsymbol{\beta}^{E_{\boldsymbol{\beta}}})$
satisfies $\partial^{2}=0$.
\end{lem}

\begin{proof}
The proof follows the usual proof of $\partial^{2}=0$, which studies
the ends of the moduli spaces ${\cal M}({\cal D})$ for domains in
${\cal D}({\bf x},{\bf y})$ with Maslov index $2$. The ends consist
of broken flowlines and boundary degenerations, and so $\partial^{2}$
equals the contribution from the boundary degenerations. See \cite[Theorem 5.5]{MR2443092},
\cite[Lemma 2.1]{MR3709653}, \cite[Proposition 7.9]{2011.00113}.
(Our case is algebraically slightly more complicated as we have nontrivial
local systems.)

Let the Heegaard surface be $\Sigma$, and let the coefficient ring
be $R^{\infty}$ where $R=\mathbb{F}\llbracket U,X_{1},\cdots,X_{k}\rrbracket$
(as in Subsection \ref{subsec:unoriented-local-system}, in particular
$X_{i}=U_{i}$ or $U_{i}^{1/2}$), where $U$ is the distinguished
variable. Let ${\bf x}\in\boldsymbol{\alpha}\cap\boldsymbol{\beta}$,
and let us consider domains ${\cal D}\in D({\bf x},{\bf x})$. We
claim that 
\[
\sum_{{\cal D}}w({\cal D})\rho({\cal D})=\left(U+\sum_{i}U_{i}\right){\rm Id}_{{\rm Hom}(E_{\boldsymbol{\alpha}},E_{\boldsymbol{\beta}})},
\]
where the two-chain of ${\cal D}$ ranges through the connected components
of $\Sigma\backslash\boldsymbol{\alpha}$, or the connected components
of $\Sigma\backslash\boldsymbol{\beta}$. This is almost exactly like
the case where all the local systems are trivial, but with the added
complication that $\partial{\cal D}$ might intersect $G$. However,
this is not possible as the two endpoints of $G$ lie in the same
connected component of $\Sigma\backslash\boldsymbol{\alpha}$ (and
$\Sigma\backslash\boldsymbol{\beta}$).

Let $\#\widetilde{N}({\cal D},{\bf x})$ be the number of boundary
degenerations with two-chain ${\cal D}$ and vertex ${\bf x}$. We
have 
\[
\partial^{2}(e{\bf x})=\sum_{{\cal D}}\#\widetilde{N}({\cal D},{\bf x})w({\cal D})\rho({\cal D})(e){\bf x},
\]
and this is zero since $\#\widetilde{N}({\cal D},{\bf x})$ does not
depend on ${\cal D}$ (and ${\bf x}$).
\end{proof}

\subsubsection{\label{subsec:Positivity}Positivity}

We define the space of \emph{filtered maps}.
\begin{defn}
Define the space of \emph{filtered maps}, $\boldsymbol{CF}_{fil}^{-}(\boldsymbol{\alpha}^{E_{\boldsymbol{\alpha}}},\boldsymbol{\beta}^{E_{\boldsymbol{\beta}}})\le\boldsymbol{CF}^{-}(\boldsymbol{\alpha}^{E_{\boldsymbol{\alpha}}},\boldsymbol{\beta}^{E_{\boldsymbol{\beta}}})$,
as follows.
\begin{itemize}
\item If either $\boldsymbol{\alpha}$ or $\boldsymbol{\beta}$ has the
trivial local system, then every element of $\boldsymbol{CF}^{-}(\boldsymbol{\alpha}^{E_{\boldsymbol{\alpha}}},\boldsymbol{\beta}^{E_{\boldsymbol{\beta}}})$
is \emph{filtered}.
\item If both $\boldsymbol{\alpha}$ and $\boldsymbol{\beta}$ have nontrivial
local systems, then $\boldsymbol{y}\in\boldsymbol{CF}^{-}(\boldsymbol{\alpha}^{E_{\boldsymbol{\alpha}}},\boldsymbol{\beta}^{E_{\boldsymbol{\beta}}})$
is \emph{filtered} if for each summand $f{\bf x}$ of $\boldsymbol{y}$
such that $\theta^{+}\in{\bf x}$, we have that ${\rm gr}_{U}f\ge0$.
\end{itemize}
In other words, we only rule out $e_{0}e_{1}^{\ast}{\bf x}$ (but
allow $Ue_{0}e_{1}^{\ast}{\bf x}$) for ${\bf x}$ such that $\theta^{+}\in{\bf x}$.
\end{defn}

The following is the main lemma that ensures positivity. (Compare
\cite[Lemmas A.3, A.5]{2308.15658})
\begin{lem}
\label{lem:positivity}Let $\boldsymbol{\alpha}_{0}^{E_{0}},\cdots,\boldsymbol{\alpha}_{d}^{E_{d}}$
be attaching curves with local systems. Assume that at most two of
them have nontrivial local systems. Consider generators $f_{i}{\bf x}_{i}\in\boldsymbol{CF}^{-}(\boldsymbol{\alpha}_{i-1}^{E_{i-1}},\boldsymbol{\alpha}_{i}^{E_{i}})$
for $i=1,\cdots,d$ and ${\bf y}\in\boldsymbol{\alpha}_{0}\cap\boldsymbol{\alpha}_{d}$.
If there exists a nonnegative domain ${\cal D}\in D({\bf x}_{1},\cdots,{\bf x}_{d},{\bf y})$,
let 
\[
g=w({\cal D})\rho({\cal D})(f_{1}\otimes\cdots\otimes f_{d}).
\]
If $f_{1}{\bf x}_{1},\cdots,f_{d}{\bf x}_{d}$ are filtered, then
$g{\bf y}$ is filtered.
\end{lem}

\begin{proof}
There is nothing to show if all the local systems are trivial. Let
us assume that $\boldsymbol{\alpha}_{a}$ and $\boldsymbol{\alpha}_{b}$
($a\le b$) have nontrivial local systems, and that the rest have
trivial local systems. Also, assume that $g\neq0$. Let $w=\partial_{+}G$,
$z=\partial_{-}G$ as in Figure \ref{fig:local-system-lemma}. We
will use the following repeatedly: if $n_{w}({\cal D})=n_{z}({\cal D})=0$
and the boundary $\partial_{\boldsymbol{\alpha}_{a}}{\cal D}$ or
$\partial_{\boldsymbol{\alpha}_{b}}{\cal D}$ intersects $G$, then
$n_{A}({\cal D})=n_{C}({\cal D})=0$ and $n_{B}({\cal D})=1$. Hence,
$a=0$, $b=d=1$, $\theta^{+}\in{\bf x}_{1}$, and $\theta^{-}\in{\bf y}$.
Note that in this case, ${\rm gr}_{U}g={\rm gr}_{U}f_{1}\ge0$, and
so $g{\bf y}$ is filtered.

We divide into a few cases. If $b\neq d$, then we we are done if
${\rm gr}_{U}g\ge-1/2$. If ${\rm gr}_{U}g<-1/2$, then it must be
the case that ${\rm gr}_{U}f_{a+1}={\rm gr}_{U}f_{b+1}=-1/2$, ${\rm gr}_{U}f_{i}=0$
for all $i\neq a+1,b+1$, and $n_{w}({\cal D})=n_{z}({\cal D})=0$.
Hence, $f_{b}=e_{0}$ or $e_{0}e_{1}^{\ast}$, and $f_{b+1}=e_{1}^{\ast}$.
Since $g{\bf y}\neq0$, the $\boldsymbol{\alpha}_{b}$-boundary $\partial_{\boldsymbol{\alpha}_{b}}{\cal D}$
intersects $G$.

If $b=d$, then we are done if ${\rm gr}_{U}g\ge0$. Hence, assume
that ${\rm gr}_{U}g<0$. Then it must be the case that ${\rm gr}_{U}f_{a+1}=-1/2$
(and so $f_{a+1}=e_{1}^{\ast}$ or $e_{0}e_{1}^{\ast}$), and ${\rm gr}_{U}f_{i}=0$
for all $i\neq a+1$, and $n_{w}({\cal D})=n_{z}({\cal D})=0$. Hence,
we can assume that $\partial_{\boldsymbol{\alpha}_{a}}{\cal D}$ and
$\partial_{\boldsymbol{\alpha}_{b}}{\cal D}$ do not intersect $G$.
\begin{itemize}
\item If $a\neq0$, then $f_{a}=e_{0}$, and so $\partial_{\boldsymbol{\alpha}_{a}}{\cal D}$
intersects $G$.
\item If $a=0$ and $\theta^{+}\in{\bf y}$, then $n_{z}({\cal D})+n_{w}({\cal D})=n_{B}({\cal D})+n_{C}({\cal D})+1\ge1$.
\item If $a=0$ and $\theta^{-}\in{\bf y}$, then since $\partial_{\boldsymbol{\alpha}_{a}}{\cal D}$
and $\partial_{\boldsymbol{\alpha}_{b}}{\cal D}$ do not intersect
$G$, $g=e_{0}e_{1}^{\ast}$. Hence $g{\bf y}$ is filtered.
\end{itemize}
\end{proof}
As everything is well-defined, we can now show that $\boldsymbol{CF}_{fil}^{-}(\boldsymbol{\alpha}^{E_{\boldsymbol{\alpha}}},\boldsymbol{\beta}^{E_{\boldsymbol{\beta}}})$
is a chain complex, and similarly that the higher $A_{\infty}$-relations
hold.
\begin{cor}
\label{cor:filtered-is-chain-complex}Let us consider a weakly admissible
Heegaard datum with attaching curves with local systems $\boldsymbol{\alpha}^{E_{\boldsymbol{\alpha}}},\boldsymbol{\beta}^{E_{\boldsymbol{\beta}}}$.
Then $\boldsymbol{CF}_{fil}^{-}(\boldsymbol{\alpha}^{E_{\boldsymbol{\alpha}}},\boldsymbol{\beta}^{E_{\boldsymbol{\beta}}})$
together with $\partial$ is a chain complex.
\end{cor}

\begin{proof}
Lemma \ref{lem:d2=00003D0} shows that $(\boldsymbol{CF}^{\infty}(\boldsymbol{\alpha}^{E_{\boldsymbol{\alpha}}},\boldsymbol{\beta}^{E_{\boldsymbol{\beta}}}),\partial)$
is a chain complex, and Lemma \ref{lem:positivity} shows that $\boldsymbol{CF}_{fil}^{-}(\boldsymbol{\alpha}^{E_{\boldsymbol{\alpha}}},\boldsymbol{\beta}^{E_{\boldsymbol{\beta}}})$
is a sub chain complex of $\boldsymbol{CF}^{\infty}(\boldsymbol{\alpha}^{E_{\boldsymbol{\alpha}}},\boldsymbol{\beta}^{E_{\boldsymbol{\beta}}})$.
\end{proof}
By abuse of notation, we mean $\boldsymbol{CF}_{fil}^{-}$ whenever
we write $\boldsymbol{CF}^{-}$ from now on.

\subsection{\label{subsec:Twisted-complexes}Twisted complexes}

We use the language of twisted complexes in order to nicely package
the necessary algebra. We mostly follow \cite[Subsection 2.3]{2308.15658}
and \cite[Sections 1b, 3k, 3l, and 3m]{MR2441780}. Also compare \cite{MR3509974}.

\subsubsection{$A_{\infty}$-categories and twisted complexes}

We work with non-unital $A_{\infty}$-categories over $\mathbb{F}$.
Our conventions are slightly different from \cite{MR2441780}, due
to the nature of Heegaard Floer homology: we use the homological convention
instead of the cohomological convention, and $A_{\infty}$-categories
are not necessarily (homologically) graded. However, the $A_{\infty}$-categories
we consider will be graded in some special cases: we say that an $A_{\infty}$-category
is \emph{graded} if the Hom spaces ${\rm Hom}(\alpha,\beta)$ are
$\mathbb{Z}$-graded and the composition maps $\mu_{d}$ have degree
$d-2$. When everything is graded, we use the convention that for
a homogeneous $f\in{\rm Hom}(\alpha,\beta)$, the corresponding element
in ${\rm Hom}(\alpha[k],\beta[l])$ has degree $\deg f+k-l$. In other
words, $[k]$ shifts the grading down by $k$.
\begin{rem}
\label{rem:ainf-pedantry}Given a weakly admissible Heegaard datum
as in Subsection \ref{subsec:unoriented-local-system}, define the
associated $A_{\infty}$-category as follows: let $\boldsymbol{\alpha}_{0}^{E_{0}},\cdots,\boldsymbol{\alpha}_{m}^{E_{m}}$
be the attaching curves together with the local systems, and assume
that we have made the analytic choices necessary to define the composition
maps in Heegaard Floer theory. Then, the $A_{\infty}$-category we
work with has objects $\boldsymbol{\alpha}_{0}^{E_{0}},\cdots,\boldsymbol{\alpha}_{m}^{E_{m}}$,
and the Hom spaces and composition maps $\mu_{d}$ are nontrivial
only if the attaching curves are in the correct order, i.e.
\[
{\rm Hom}(\boldsymbol{\alpha}_{i}^{E_{i}},\boldsymbol{\alpha}_{j}^{E_{j}})=\begin{cases}
\boldsymbol{CF}_{fil}^{-}(\boldsymbol{\alpha}_{i}^{E_{i}},\boldsymbol{\alpha}_{j}^{E_{j}}) & {\rm if}\ i<j\\
0 & {\rm if}\ i\ge j
\end{cases}.
\]
When we write $\boldsymbol{CF}^{-}(\boldsymbol{\alpha},\boldsymbol{\beta})$,
it will be implicit that $\boldsymbol{\alpha}$ comes before $\boldsymbol{\beta}$.
\end{rem}

Let ${\cal A}$ be an $A_{\infty}$-category. Then, the \emph{additive
enlargement $\Sigma{\cal A}$} is defined as follows. Objects are
formal direct sums 
\[
\bigoplus_{i\in I}V_{i}\otimes\alpha_{i}
\]
where $I$ is a finite set, $\{\alpha_{i}\}$ is a family of objects
of ${\cal A}$, and $\{V_{i}\}$ is a family of finite-dimensional
vector spaces. Morphisms are defined as a combination of morphisms
between the vector spaces and morphisms in ${\cal A}$:
\[
{\rm Hom}_{\Sigma{\cal A}}\left(\bigoplus_{i\in I}V_{i}\otimes\alpha_{i},\bigoplus_{j\in J}W_{j}\otimes\beta_{j}\right)=\bigoplus_{i,j}{\rm Hom}_{\mathbb{F}}(V_{i},W_{j})\otimes{\rm Hom}_{{\cal A}}(\alpha_{i},\beta_{j}).
\]
Similarly, compositions $\mu_{d}$ are defined by combining the ordinary
composition of maps between vector spaces and morphisms in ${\cal A}$.
The additive enlargement $\Sigma{\cal A}$ forms an $A_{\infty}$-category.

A \emph{twisted complex in ${\cal A}$ }consists of an object $\underline{\alpha}\in{\rm Ob}(\Sigma{\cal A})$
together with a differential $\delta_{\underline{\alpha}}\in{\rm Hom}_{\Sigma{\cal A}}\left(\underline{\alpha},\underline{\alpha}\right)$,
such that 
\[
\sum_{n\ge1}\mu_{n}^{\Sigma{\cal A}}(\delta_{\underline{\alpha}},\cdots,\delta_{\underline{\alpha}})=0.
\]
Seidel assumes that the differential $\delta_{\underline{\alpha}}$
is strictly lower triangular with respect to a filtration on $\underline{\alpha}$
to ensure that the above sum is finite.

Twisted complexes in ${\cal A}$ also form an $A_{\infty}$-category
${\rm Tw}{\cal A}$. Morphisms between twisted complexes are the same
as before, 
\[
{\rm Hom}_{{\rm Tw}{\cal A}}\left((\underline{\alpha},\delta_{\underline{\alpha}}),(\underline{\beta},\delta_{\underline{\beta}})\right)={\rm Hom}_{\Sigma{\cal A}}(\underline{\alpha},\underline{\beta}),
\]
but the compositions are different:
\[
\mu_{d}^{{\rm Tw}{\cal A}}(\underline{\varphi_{1}},\cdots,\underline{\varphi_{d}})=\sum_{i_{0},\dots,i_{d}\ge0}\mu_{d+i_{0}+\cdots+i_{d}}^{\Sigma{\cal A}}(\overbrace{\delta_{\underline{\alpha_{0}}},\cdots,\delta_{\underline{\alpha_{0}}}}^{i_{0}},\underline{\varphi_{1}},\overbrace{\delta_{\underline{\alpha_{1}}},\cdots,\delta_{\underline{\alpha_{1}}}}^{i_{1}},\underline{\varphi_{2}},\cdots,\underline{\varphi_{d}},\overbrace{\delta_{\underline{\alpha_{d}}},\cdots,\delta_{\underline{\alpha_{d}}}}^{i_{d}}).
\]

\begin{rem}
If ${\cal A}$ is a graded $A_{\infty}$-category, then its \emph{graded
additive enlargement} $\Sigma{\cal A}$ is defined exactly the same,
except that the finite-dimensional vector spaces $V_{i}$ are also
$\mathbb{Z}$-graded. It forms a graded $A_{\infty}$-category. Similarly,
\emph{graded twisted complexes} are defined exactly the same, except
that we require that the differential has degree $-1$. Morphisms
between twisted complexes and their compositions are defined exactly
the same. The $A_{\infty}$-category of graded twisted complexes,
${\rm Tw}{\cal A}$, is graded.
\end{rem}

\begin{example}
\label{exa:twisted-complexes}
\begin{itemize}
\item Given an object $\alpha$ of ${\cal A}$, $\alpha$ together with
the zero differential $\delta_{\alpha}=0$ form a twisted complex
in ${\cal A}$.
\item Given objects $\beta_{1},\beta_{2}$ of ${\cal A}$ and a morphism
$\psi:\beta_{1}\to\beta_{2}$ such that $\mu_{1}^{{\cal A}}(\psi)=0$,
the \emph{mapping cone of $\psi$} is the twisted complex $(\underline{\beta},\delta_{\underline{\beta}})$
given by $\underline{\beta}=\beta_{1}\oplus\beta_{2}$ and $\delta_{\underline{\beta}}=\psi$.
We write this as 
\[
\beta_{1}\xrightarrow{\psi}\beta_{2}.
\]
\end{itemize}
\end{example}

\begin{notation}
As above, we will usually underline twisted complexes and maps between
them. An exception is when the twisted complex comes from a single
object $\alpha$, as in the first example of Example \ref{exa:twisted-complexes}.
Also, if $\alpha_{0}$ and $\beta_{0}$ are summands of the twisted
complexes $\underline{\alpha}$ and $\underline{\beta}$, respectively,
and $\theta:\alpha_{0}\to\beta_{0}$ is a map, then we denote the
map $\underline{\alpha}\to\underline{\beta}$ whose $\alpha_{0}\to\beta_{0}$
component is $\theta$ and all the other components are $0$ as $\underline{\theta}:\underline{\alpha}\to\underline{\beta}$.
\end{notation}

\subsubsection{$A_{\infty}$-functors and twisted complexes}

We will define twisted complexes of attaching curves in one almost
complex structure, and deform the almost complex structure. To do
this, we have to understand how the $A_{\infty}$-functor induced
by deforming the almost complex structure extends to the category
of twisted complexes.

Recall that an $A_{\infty}$-functor ${\cal F}:{\cal A}\to{\cal B}$
between two $A_{\infty}$-categories is a collection of a map ${\rm Ob}{\cal F}:{\rm Ob}{\cal A}\to{\rm Ob}{\cal B}$
and multilinear maps ${\cal F}_{d}:{\rm Hom}_{{\cal A}}(\alpha_{0},\alpha_{1})\otimes\cdots\otimes{\rm Hom}_{{\cal A}}(\alpha_{d-1},\alpha_{d})\to{\rm Hom}_{{\cal B}}({\cal F}(\alpha_{0}),{\cal F}(\alpha_{d}))$
for $d\ge1$ that satisfy the appropriate $A_{\infty}$-relations.
If the $A_{\infty}$-categories ${\cal A}$ and ${\cal B}$ are graded,
then ${\cal F}$ is \emph{graded} if ${\cal F}_{d}$ has degree $d-1$
for all $d\ge1$.

Given an $A_{\infty}$-functor ${\cal F}:{\cal A}\to{\cal B}$ between
$A_{\infty}$-categories, the induced $A_{\infty}$-functor $\Sigma{\cal F}:\Sigma{\cal A}\to\Sigma{\cal B}$
is defined as follows: on the objects, 
\[
\left(\Sigma{\cal F}\right)\left(\bigoplus_{i\in I}V_{i}\otimes\alpha_{i}\right)=\bigoplus_{i\in I}V_{i}\otimes{\cal F}(\alpha_{i}).
\]
On the morphisms, $\left(\Sigma{\cal F}\right)_{d}(\underline{\varphi_{1}},\cdots,\underline{\varphi_{d}})$
is defined analogously to $\mu_{d}^{\Sigma{\cal A}}$, where we combine
ordinary composition of maps between vector spaces and ${\cal F}_{d}$.

The induced $A_{\infty}$-functor ${\rm Tw}{\cal F}:{\rm Tw}{\cal A}\to{\rm Tw}{\cal B}$
is defined as follows:

\[
\begin{aligned}\left({\rm Tw}{\cal F}\right)\left(\underline{\alpha},\delta_{\underline{\alpha}}\right) & =\left(\left(\Sigma{\cal F}\right)\left(\underline{\alpha}\right),\sum_{n\ge1}\left(\Sigma{\cal F}\right)_{n}\left(\delta_{\underline{\alpha}},\cdots,\delta_{\underline{\alpha}}\right)\right),\\
\left({\rm Tw}F\right)_{d}(\underline{\varphi_{1}},\cdots,\underline{\varphi_{d}}) & =\sum_{i_{0},\dots,i_{d}\ge0}\left(\Sigma{\cal F}\right)_{d+i_{0}+\cdots+i_{d}}(\overbrace{\delta_{\underline{\alpha_{0}}},\cdots,\delta_{\underline{\alpha_{0}}}}^{i_{0}},\underline{\varphi_{1}},\\
 & \ \ \ \ \ \ \ \ \ \ \ \ \ \ \ \ \ \ \ \ \overbrace{\delta_{\underline{\alpha_{1}}},\cdots,\delta_{\underline{\alpha_{1}}}}^{i_{1}},\underline{\varphi_{2}},\cdots,\underline{\varphi_{d}},\overbrace{\delta_{\underline{\alpha_{d}}},\cdots,\delta_{\underline{\alpha_{d}}}}^{i_{d}}).
\end{aligned}
\]

If everything is graded, then ${\rm Tw}{\cal F}$ is graded as well.
\begin{rem}
\label{rem:pedantry-twisted}As noted above, we will define twisted
complexes of attaching curves and maps between them in one almost
complex structure, and will continue talking about them in a different
almost complex structure. The latter should be interpreted as the
image under ${\rm Tw}{\cal F}$, where ${\cal F}$ is the $A_{\infty}$-functor
given by deforming the almost complex structure. Note that ${\cal F}$
depends on the choice of the deformation. Fortunately, it will turn
out that all the twisted complexes $\underline{\boldsymbol{\beta}}$
that we consider in this paper are ``invariant'' under ${\rm Tw}{\cal F}$.
Furthermore, although the maps between the twisted complexes that
we consider might depend on the deformation, it will not matter since
we will not care about exactly what these maps are.
\end{rem}

\subsubsection{Exact triangles and twisted complexes}

If $C_{1},C_{2},C_{3}$ are chain complexes, $f:C_{1}\to C_{2}$ is
a chain map, and the mapping cone $C_{1}\xrightarrow{f}C_{2}$ is
quasi-isomorphic to $C_{3}$, then we have an exact triangle of their
homology groups 
\[
\cdots\to H(C_{1})\xrightarrow{f_{\ast}}H(C_{2})\to H(C_{3})\to H(C_{1})\to\cdots.
\]
In this paper, we care about what the maps in this exact triangle
are. Compare with the triangle detection lemma \cite[Lemma 4.2]{MR2141852}.
\begin{lem}
\label{lem:exact-triangle}Let the following be a twisted complex% https://q.uiver.app/#q=WzAsNCxbMSwwLCJcXGJldGFfMSJdLFsyLDAsIlxcYmV0YV8yIl0sWzMsMCwiXFxiZXRhXzMiXSxbMCwwLCJcXHVuZGVybGluZXtcXGJldGF9Il0sWzAsMSwiZiJdLFsxLDIsImciXSxbMCwyLCIiLDIseyJjdXJ2ZSI6Mn1dLFszLDAsIjo9IiwxLHsic3R5bGUiOnsiYm9keSI6eyJuYW1lIjoibm9uZSJ9LCJoZWFkIjp7Im5hbWUiOiJub25lIn19fV1d
\[\begin{tikzcd}
	{\underline{\beta}} & {\beta_1} & {\beta_2} & {\beta_3}
	\arrow["{:=}"{description}, draw=none, from=1-1, to=1-2]
	\arrow["f", from=1-2, to=1-3]
	\arrow[curve={height=12pt}, from=1-2, to=1-4]
	\arrow["g", from=1-3, to=1-4]
\end{tikzcd}\]If the homology of ${\rm Hom}(\alpha,\underline{\beta})$ is $0$,
then the following sequence of homology groups is exact.
\[
H({\rm Hom}(\alpha,\beta_{1}))\xrightarrow{\mu_{2}(-,f)}H({\rm Hom}(\alpha,\beta_{2}))\xrightarrow{\mu_{2}(-,g)}H({\rm Hom}(\alpha,\beta_{3}))
\]
\end{lem}

\begin{proof}
One can show this using a diagram chasing argument and the definition
of ${\rm Hom}(\alpha,\underline{\beta})$.
\end{proof}
\begin{rem}
That $\underline{\beta}$ is a twisted complex is equivalent to that
$\underline{\beta_{23}}:=\beta_{2}\xrightarrow{g}\beta_{3}$ is a
twisted complex and that the induced map $\beta_{1}\to\underline{\beta_{23}}$
is a cycle. That the homology of ${\rm Hom}(\alpha,\underline{\beta})$
is $0$ is equivalent to that the induced map ${\rm Hom}(\alpha,\beta_{1})\to{\rm Hom}(\alpha,\underline{\beta_{23}})$
is a quasi-isomorphism.

Analogous statements hold if we consider $\underline{\beta_{12}}:=\beta_{1}\xrightarrow{f}\beta_{2}$
and $\underline{\beta_{12}}\to\beta_{3}$ instead.
\end{rem}

We record a lemma that we use to deduce the exactness of the triangles
for the hat and infinity versions from a statement for the minus version.
\begin{lem}
\label{lem:hat-infty-alg}Let $C$ be a free chain complex over $\mathbb{F}\llbracket U\rrbracket$
such that its homology $H(C)=0$. Then, 
\[
H(C\otimes_{\mathbb{F}\llbracket U\rrbracket}\mathbb{F}\llbracket U\rrbracket/U)=0,\ H(C\otimes_{\mathbb{F}\llbracket U\rrbracket}\mathbb{F}\llbracket U\rrbracket[U^{-1}])=0.
\]
\end{lem}

\begin{proof}
The first statement follows from the short exact sequence of chain
complexes 
\[
0\to C\xrightarrow{\cdot U}C\rightarrow C\otimes_{\mathbb{F}\llbracket U\rrbracket}\mathbb{F}\llbracket U\rrbracket/U\to0.
\]
For the second statement, it is easy to directly check that any cycle
is a boundary.
\end{proof}

\subsection{\label{subsec:standard-translates}Standard translates}

Our main ``local computation,'' Theorem \ref{thm:sym2-t2}, morally
says that two specific maps $\underline{f}:\underline{\boldsymbol{\beta}_{1}}\to\underline{\boldsymbol{\beta}_{2}}$
and $\underline{g}:\underline{\boldsymbol{\beta}_{2}}\to\underline{\boldsymbol{\beta}_{1}}$
are homotopy inverses of each other. However, the $A_{\infty}$-categories
we work in are not unital: we do not even define $\boldsymbol{CF}^{-}(\boldsymbol{\alpha},\boldsymbol{\alpha})$.
Instead of trying to define units and homotopy inverses, we use standard
translates and use some ad-hoc arguments. The definitions in this
subsection should be treated as useful shorthands rather than actual
definitions.

\begin{figure}[h]
\begin{centering}
\includegraphics{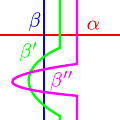}
\par\end{centering}
\caption{\label{fig:standard-translate}Standard translates of $\beta$}
\end{figure}

\begin{defn}
\label{def:standard-translate}Let $\boldsymbol{\beta}=\{\beta^{1},\cdots,\beta^{n}\}$
be an attaching curve. A \emph{standard translate $\boldsymbol{\beta}'=\{\beta'{}^{1},\cdots,\beta'{}^{n}\}$}
of $\boldsymbol{\beta}$ is given by slightly translating each circle
$\beta^{i}$, such that $\#(\beta^{i}\cap\beta'{}^{j})=2\delta_{ij}$.
If we consider other attaching curves $\boldsymbol{\alpha}$, then
we assume that the translate is sufficiently small in the sense that
$\boldsymbol{\alpha}$ meets $\boldsymbol{\beta}$ and $\boldsymbol{\beta}'$
in the same way (see Figure \ref{fig:standard-translate}). Denote
the top homological grading generator of $\boldsymbol{CF}^{-}(\boldsymbol{\beta},\boldsymbol{\beta}')$
as $\Theta_{\boldsymbol{\beta}}^{+}$, or simply $\Theta^{+}$. If
$\boldsymbol{\beta}$ intersects the oriented arc $G$ (i.e. has a
nontrivial local system), then the translate is small enough so that
$G$ also meets $\boldsymbol{\beta}$ and $\boldsymbol{\beta}'$ in
the same way. We will sometimes denote the generator ${\rm Id}_{E}\Theta_{\boldsymbol{\beta}}^{+}$
of $\boldsymbol{CF}^{-}(\boldsymbol{\beta}^{E},\boldsymbol{\beta}'{}^{E})$
as $\Theta_{\boldsymbol{\beta}}^{+}$ or $\Theta^{+}$.

Given a map $f:\boldsymbol{\alpha}^{E_{\boldsymbol{\alpha}}}\to\boldsymbol{\beta}^{E_{\boldsymbol{\beta}}}$,
define the corresponding maps $f':a\to b$ for $a\in\{\boldsymbol{\alpha}^{E_{\boldsymbol{\alpha}}},\boldsymbol{\alpha}'{}^{E_{\boldsymbol{\alpha}'}}\}$
and $b\in\{\boldsymbol{\beta}^{E_{\boldsymbol{\beta}}},\boldsymbol{\beta}'{}^{E_{\boldsymbol{\beta}'}}\}$
by taking the nearest intersection points. Using these, define standard
translates $\underline{\boldsymbol{\alpha}'}$ of twisted complexes
$\underline{\boldsymbol{\alpha}}$ of attaching curves (see Remark
\ref{rem:near-caveat} for a caveat), and given a map $\underline{f}:\underline{\boldsymbol{\alpha}}\to\underline{\boldsymbol{\beta}}$,
define the corresponding maps $\underline{f'}:a\to b$ for $a\in\{\underline{\boldsymbol{\alpha}},\underline{\boldsymbol{\alpha}'}\}$
and $b\in\{\underline{\boldsymbol{\beta}},\underline{\boldsymbol{\beta}'}\}$
as well.

Similarly, if $\underline{\boldsymbol{\beta}}$ is a twisted complex,
then denote the map $\underline{\boldsymbol{\beta}}\to\underline{\boldsymbol{\beta}'}$
given by $\Theta_{\boldsymbol{\beta}}^{+}:\boldsymbol{\beta}\to\boldsymbol{\beta}'$
for each summand $\boldsymbol{\beta}$ as $\Theta_{\underline{\boldsymbol{\beta}}}^{+}$,
or simply $\Theta^{+}\in\boldsymbol{CF}^{-}(\underline{\boldsymbol{\beta}},\underline{\boldsymbol{\beta}'})$.
\end{defn}

\begin{rem}
\label{rem:near-caveat}We do not claim that the induced map $f'$,
resp., $\underline{f'}$ is always a cycle if $f$, resp., $\underline{f}$
is a cycle, that $\underline{\boldsymbol{\alpha}'}$ is always a twisted
complex, and that $\Theta_{\underline{\boldsymbol{\beta}}}^{+}$ is
a cycle. However, they will be for all the cases that we consider
in this paper.
\end{rem}

Instead of working with homotopy equivalences (and homotopy inverses),
we use the following and work with quasi-isomorphisms in the sense
of Definition \ref{def:moral-def}\footnote{By Yoneda, they should be the same; but our $A_{\infty}$-categories
do not even have ${\rm Hom}(\alpha,\alpha)$.}.

\begin{lem}
\label{lem:theta-quasi-iso}Let us consider a weakly admissible Heegaard
datum. Consider a twisted complex $\underline{\boldsymbol{\beta}}$
of attaching curves with local systems and its small translate $\underline{\boldsymbol{\beta}'}$,
with the same filtration. If $\underline{\boldsymbol{\beta}'}$ is
a twisted complex, and if $\underline{\Theta}:\underline{\boldsymbol{\beta}}\to\underline{\boldsymbol{\beta}'}$
is a cycle that respects the filtration (i.e. lower triangular), and
for each summand $\boldsymbol{\beta}$ of $\underline{\boldsymbol{\beta}}$,
the component of $\underline{\Theta}$ that is from $\boldsymbol{\beta}$
to $\boldsymbol{\beta}'$ is homotopic to $\Theta_{\boldsymbol{\beta}}^{+}:\boldsymbol{\beta}\to\boldsymbol{\beta}'$,
then for any other attaching curve $\boldsymbol{\alpha}$, the induced
map 
\[
\mu_{2}(-,\underline{\Theta}):\boldsymbol{CF}^{-}(\boldsymbol{\alpha},\underline{\boldsymbol{\beta}})\to\boldsymbol{CF}^{-}(\boldsymbol{\alpha},\underline{\boldsymbol{\beta}'})
\]
is a quasi-isomorphism.
\end{lem}

\begin{proof}
The filtrations on $\underline{\boldsymbol{\beta}},\underline{\boldsymbol{\beta}'}$
induce a filtration on the chain complexes $\boldsymbol{CF}^{-}(\boldsymbol{\alpha},\underline{\boldsymbol{\beta}})$
and $\boldsymbol{CF}^{-}(\boldsymbol{\alpha},\underline{\boldsymbol{\beta}'})$.
With respect to these filtrations, $\mu_{2}(-,\underline{\Theta})$
is lower triangular, and the induced maps on the successive quotients
are quasi-isomorphisms.
\end{proof}
\begin{rem}
\label{rem:theta-quasi-iso}Note that if $\Theta_{\underline{\boldsymbol{\beta}}}^{+}:\underline{\boldsymbol{\beta}}\to\underline{\boldsymbol{\beta}'}$
is a cycle, then it satisfies the condition of Lemma \ref{lem:theta-quasi-iso}.
Also, if the $A_{\infty}$-functor induced by changing the almost
complex structure preserves $\underline{\boldsymbol{\beta}}$ and
$\underline{\boldsymbol{\beta}'}$, then the image of $\Theta_{\underline{\boldsymbol{\beta}}}^{+}$
satisfies the condition of Lemma \ref{lem:theta-quasi-iso}.
\end{rem}

\begin{lem}[Derived Nakayama]
\label{lem:derived-nakayama}Let $R$ be a power series ring $\mathbb{F}\llbracket X_{1},\cdots,X_{n}\rrbracket$
over a field $\mathbb{F}$, and let $C,C'$ be two finite, free chain
complexes over $R$. Then, a $R$-linear chain map $f:C\to C'$ is
a quasi-isomorphism if 
\[
f\otimes_{R}\mathbb{F}:C\otimes_{R}\mathbb{F}\to C'\otimes_{R}\mathbb{F}
\]
is a quasi-isomoprhism.
\end{lem}

\begin{proof}
Let $R_{k}=R/(X_{1},\cdots,X_{k})$, and consider the mapping cones
$C_{k}$ of $f\otimes_{R}R_{k}$. We show by induction that the homology
$H(C_{0})$ is trivial if $H(C_{n})$ is trivial. 

Assume that $H(C_{k+1})$ is trivial. Consider the long exact sequence
induced by the short exact sequence 
\[
0\to C_{k}\xrightarrow{\cdot X_{k}}C_{k}\to C_{k+1}\to0.
\]
Since $H(C_{k+1})$ is trivial, the map $\cdot X_{k}:H(C_{k})\to H(C_{k})$
is an isomorphism, and so $H(C_{k})$ is trivial by Nakayama's Lemma.
\end{proof}
\begin{defn}
\label{def:moral-def}Two cycles $\underline{f}:\underline{\boldsymbol{\beta}_{1}}\to\underline{\boldsymbol{\beta}_{2}}$
and $\underline{g}:\underline{\boldsymbol{\beta}_{2}}\to\underline{\boldsymbol{\beta}_{1}}$
are \emph{morally quasi-inverses} if the compositions $\mu_{2}(\underline{f},\underline{g'}):\underline{\boldsymbol{\beta}_{1}}\to\underline{\boldsymbol{\beta}_{1}'}$
and $\mu_{2}(\underline{g'},\underline{f}):\underline{\boldsymbol{\beta}_{2}'}\to\underline{\boldsymbol{\beta}_{2}}$
satisfy the conditions of Lemma \ref{lem:theta-quasi-iso} modulo
$(X_{1},\cdots,X_{n})$, where the coefficient ring is of the form
$R=\mathbb{F}\llbracket X_{1},\cdots,X_{n}\rrbracket$. If such $\underline{f},\underline{g}$
exist, then we say that $\underline{\boldsymbol{\beta}_{1}}$ and
$\underline{\boldsymbol{\beta}_{2}}$ are \emph{morally quasi-isomorphic}.

Similarly, cycles $\underline{f}:\underline{\boldsymbol{\beta}_{1}}\to\underline{\boldsymbol{\beta}_{2}}$
and $\underline{g}:\underline{\boldsymbol{\beta}_{2}}\to\underline{\boldsymbol{\beta}_{1}}$
are \emph{morally homotopy inverses} if the compositions $\mu_{2}(\underline{f},\underline{g'})$
and $\mu_{2}(\underline{g'},\underline{f})$ are $\Theta^{+}$. If
such $\underline{f},\underline{g}$ exist, then we say that $\underline{\boldsymbol{\beta}_{1}}$
and $\underline{\boldsymbol{\beta}_{2}}$ are \emph{morally homotopy
equivalent}.
\end{defn}

Note that if two maps are morally homotopy inverses, then they are
morally quasi-inverses.
\begin{rem}
\label{rem:moral-caveat}Note that we do not claim that being morally
quasi-isomorphic or morally homotopy equivalent is an equivalence
relation, nor that it does not depend on the choice of the standard
translate. In our case, we will have twisted complexes $\underline{\boldsymbol{\beta}_{i}}$
for $i=1,2,3,4$, and we will show that $\underline{\boldsymbol{\beta}_{i}}$
and $\underline{\boldsymbol{\beta}_{i+1}}$ are morally quasi-isomorphic.
We will deduce that $\underline{\boldsymbol{\beta}_{1}}$ and $\underline{\boldsymbol{\beta}_{4}}$
are morally homotopy equivalent by using Lemma \ref{lem:theta-quasi-iso}
and an ad-hoc argument. (There are other ways to finish off.)
\end{rem}

\subsection{\label{subsec:Stabilizations}Stabilizations}

If the multi-Heegaard diagram is the connected sum of two Heegaard
diagrams where the curves are sufficiently simple on one side, we
can often reduce the computations to a computation on the other side,
by stretching the neck, as in \cite[Theorem 9.4]{MR2113020}. A subtlety
is that we have infinitely many domains that might contribute. One
could quotient out the coefficient ring by $(U_{1}^{N},\cdots,U_{l}^{N})$
for large $N$'s, in which case only finitely many domains contribute;
another way is to degenerate the neck and work in a pinched Heegaard
surface as in \cite[Sections 8, 9, and 10]{2011.00113} and \cite[Section 6]{2201.12906}.
Also compare \cite[Subsection 6.1]{2308.15658}. In this subsection,
we consider the latter and work in a pinched Heegaard surface, but
if there are only finitely many domains that could contribute, then
all the statements hold for a sufficiently long neck.

Consider a Heegaard datum with Heegaard surface $\Sigma$. Let us
first discuss genus stabilization, free stabilization, and free link
stabilization.

\begin{figure}[h]
\begin{centering}
\includegraphics{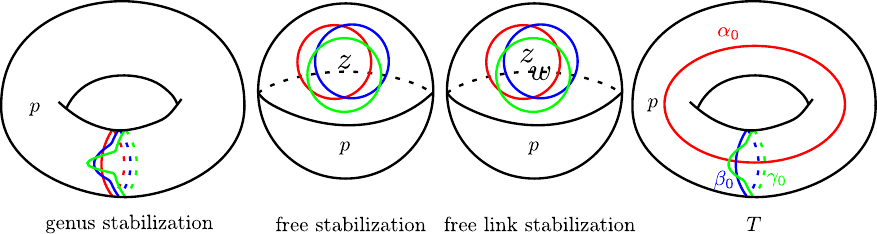}
\par\end{centering}
\caption{\label{fig:stabilization}Different kinds of stabilizations that we
consider. $p$ is the point we connected sum along, and $z$ is a
basepoint.}
\end{figure}

\emph{Genus stabilization} is where the Heegaard surface $S(\Sigma)$
is given by connected summing $\Sigma$ with $\mathbb{T}^{2}$, and
the attaching curves are $S(\boldsymbol{\beta})$, which are given
by the attaching curve $\boldsymbol{\beta}$ on $\Sigma$ together
with some fixed, essential curve, in a standard way (i.e. the fixed
curves are pairwise standard translates of each other) as in Figure
\ref{fig:stabilization}. The coefficient ring and weight functions
are unchanged. For a generator $f{\bf x}\in\boldsymbol{CF}^{-}(\boldsymbol{\alpha}^{E_{\boldsymbol{\alpha}}},\boldsymbol{\beta}^{E_{\boldsymbol{\beta}}})$,
we define 
\[
S(f{\bf x}):=f{\bf x}\times\theta^{+}\in\boldsymbol{CF}^{-}(S(\boldsymbol{\alpha}^{E_{\boldsymbol{\alpha}}}),S(\boldsymbol{\beta}^{E_{\boldsymbol{\beta}}}))
\]
where $\theta^{+}$ is the top homological grading intersection point
in the connected summed $\mathbb{T}^{2}$.

\emph{Free stabilization} is similar, but we connected sum with a
pointed $S^{2}$ as in Figure \ref{fig:stabilization}. The coefficient
ring is $R\llbracket V\rrbracket$, where $V$ is a new variable,
and the weight function is the same except that we assign the weight
$V$ to the new basepoint. We called this connected summed part the
\emph{free stabilization region}.

\emph{Free link stabilization} is a slight variant of free stabilization.
We connected sum with a doubly pointed $S^{2}$, the coefficient ring
is $R\llbracket V^{1/2}\rrbracket$, and we assign the weight $V^{1/2}$
to both the two new basepoints.

A Heegaard datum ${\cal H}^{stab}$ is a \emph{stabilization of a
Heegaard datum} ${\cal H}$ if ${\cal H}^{stab}$ is obtained from
${\cal H}$ by a sequence of genus stabilizations, free stabilizations,
and free link stabilizations. To apply \cite[Proposition 6.5]{2201.12906}
as stated, we will assume that subsequent genus, free, and free link
stabilizations are connected summed at the stabilized region, i.e.
a stabilization of a Heegaard datum with Heegaard surface $\Sigma$
is given by connected summing with a pointed surface $\Sigma'$ at
one point\footnote{This is not a serious constraint, as one can always assume this after
some handleslides.}\footnote{See \cite[Proposition 6.1]{MR4503956} for generalized connected sums,
where we ``connected sum'' at more than one point.}. We denote the composition of the $S$ maps as $S$.
\begin{prop}[{\cite[Proposition 6.5]{2201.12906}}]
\label{prop:stabilization}If the connected summing region is pinched,
then 
\[
\mu_{n}(S(\boldsymbol{y}_{1}),\cdots,S(\boldsymbol{y}_{n}))=S(\mu_{n}(\boldsymbol{y}_{1},\cdots,\boldsymbol{y}_{n})).
\]
\end{prop}

Another case we need is when we have a triple Heegaard diagram $(\Sigma,\boldsymbol{\alpha},\boldsymbol{\beta},\boldsymbol{\gamma},\boldsymbol{p})$
and we connect sum with $(\mathbb{T}^{2},\alpha_{0},\beta_{0},\gamma_{0})$
where $\alpha_{0}$ is the meridian and $\beta_{0}$ and $\gamma_{0}$
are (standard translates of) the longitude, and get $(\Sigma\#\mathbb{T}^{2},T(\boldsymbol{\alpha}),T(\boldsymbol{\beta}),T(\boldsymbol{\gamma}),\boldsymbol{p})$.
Note that if we only consider the $\beta$ and $\gamma$ attaching
curves, then this is a genus stabilization. For a generator $f{\bf x}\in\boldsymbol{CF}^{-}(\boldsymbol{\alpha}^{E_{\boldsymbol{\alpha}}},\boldsymbol{\beta}^{E_{\boldsymbol{\beta}}})$
or $\boldsymbol{CF}^{-}(\boldsymbol{\alpha}^{E_{\boldsymbol{\alpha}}},\boldsymbol{\gamma}^{E_{\boldsymbol{\gamma}}})$,
let $T(f{\bf x}):=f{\bf x}\times\theta$ where $\theta$ is the unique
intersection point in the connected summed $\mathbb{T}^{2}$. For
$\boldsymbol{y}\in\boldsymbol{CF}^{-}(\boldsymbol{\beta}^{E_{\boldsymbol{\beta}}},\boldsymbol{\gamma}^{E_{\boldsymbol{\gamma}}})$,
let $T(\boldsymbol{y}):=S(\boldsymbol{y})$. Then, \cite[Proposition 6.5]{2201.12906}
also implies the following proposition.
\begin{prop}[{\cite[Proposition 6.5]{2201.12906}}]
\label{prop:stabilization-1}If the connected summing region is pinched,
then 
\[
\mu_{2}(T(\boldsymbol{y}_{1}),T(\boldsymbol{y}_{2}))=T(\mu_{2}(\boldsymbol{y}_{1},\boldsymbol{y}_{2})).
\]
\end{prop}

We understand the chain complexes of stabilized diagrams. We will
need a statement for free (link) stabilizations. Note that for free
link stabilizations, since $z$ and $w$ are in the same elementary
two-chain, if $\boldsymbol{CF}^{-}({\cal H}^{freelink})$ and $\boldsymbol{CF}^{-}({\cal H}^{free})$
are the chain complexes of free link stabilization and free stabilization,
respectively, then we simply have $\boldsymbol{CF}^{-}({\cal H}^{freelink})=\boldsymbol{CF}^{-}({\cal H}^{free})\oplus V^{1/2}\boldsymbol{CF}^{-}({\cal H}^{free})$
as chain complexes.
\begin{prop}[{\cite[Proposition 6.5]{MR2443092}, \cite[Proposition 6.5, Lemma 14.4]{1512.01184}}]
\label{prop:free-stabilization}Let us consider a weakly admissible
Heegaard datum ${\cal H}$ with Heegaard diagram $(\Sigma,\boldsymbol{\alpha},\boldsymbol{\beta},\boldsymbol{p})$
and coefficient ring $R$. Let ${\cal H}^{stab}$ be a free stabilization
of ${\cal H}$ at the point $p\in\Sigma$. Then, if the connected
summing region is pinched, then the differential for ${\cal H}^{stab}$
is of the form 
\[
\begin{pmatrix}\partial_{{\cal H}}\otimes_{R}{\rm Id}_{R\llbracket V\rrbracket} & F\\
0 & \partial_{{\cal H}}\otimes_{R}{\rm Id}_{R\llbracket V\rrbracket}
\end{pmatrix}
\]
with respect to the identification $\boldsymbol{CF}^{-}({\cal H}^{stab})\simeq\boldsymbol{CF}^{-}({\cal H})\llbracket V\rrbracket\oplus\boldsymbol{CF}^{-}({\cal H})\llbracket V\rrbracket$
where the two summands consist of the generators that have $\theta^{+}$,
resp., $\theta^{-}$.

Furthermore, if $W$ is the weight of the connected component of $\Sigma\backslash\boldsymbol{\alpha}$
that contains the point $p$, then the map
\[
F:\boldsymbol{CF}^{-}({\cal H})\{\theta^{-}\}\llbracket V\rrbracket\to\boldsymbol{CF}^{-}({\cal H})\{\theta^{+}\}\llbracket V\rrbracket
\]
is homotopic to $V-W$ over $R\llbracket V\rrbracket$.
\end{prop}

In particular, if we assume the condition described in Subsection
\ref{subsec:unoriented-local-system}, then $W=U_{i}$ for some $i$,
and so the free-stabilization map $S:\boldsymbol{CF}^{-}({\cal H})\to\boldsymbol{CF}^{-}({\cal H}^{stab})$
is a quasi-isomorphism.

\subsection{\label{subsec:link-heegaard-diagram}Links and Heegaard diagrams}

We assume that the reader is familiar with the definition of link
Floer homology \cite{MR2443092}. There are various additional data
that we can equip to a link; we clarify our conventions.
\begin{defn}
\label{def:suture-data}A \emph{minimally pointed link $(L,\boldsymbol{u})$}
is a link $L\subset Y$ together with a set of \emph{link basepoints}
$\boldsymbol{u}\subset L$ such that for each connected component
$L_{i}$ of $L$, we have $|L_{i}\cap\boldsymbol{u}|=2$.

A \emph{suture datum} on a minimally pointed link $L\subset Y$ is
a choice of one of the two connected components of $L_{i}\backslash\boldsymbol{u}$
(which corresponds to the component inside the $\alpha$-handlebody)
for each connected component $L_{i}$. Denote the set of the chosen
components as $\alpha_{L}$. 

A \emph{$w|z$-partition} of a minimally pointed link $(L,\boldsymbol{u})$
is a partition $\boldsymbol{u}=\boldsymbol{w}\sqcup\boldsymbol{z}$
of the link basepoints into $w$-basepoints $\boldsymbol{w}$ and
$z$-basepoints $\boldsymbol{z}$ such that each link component has
exactly one $w$-basepoint and one $z$-basepoint.
\end{defn}

Let us consider the Heegaard diagram $(\Sigma,\boldsymbol{\alpha},\boldsymbol{\beta},\boldsymbol{p})$
of a Heegaard datum that we considered in Subsection \ref{subsec:unoriented-local-system}.
In this subsection, we assume that there is no oriented arc $G$,
for simplicity.

This Heegaard diagram specifies a sutured manifold \cite{MR2253454,MR4337438}
whose boundary components are either a sphere with one suture or a
torus with two sutures. Equivalently\footnote{Different sutures on the sphere components give rise to naturally
chain homotopy equivalent Heegaard Floer chain complexes, so we only
focus on the sutures on the torus components.}, this specifies a minimally pointed link $(L,\boldsymbol{u})$ equipped
with a suture datum, inside a pointed three-manifold $(Y,\boldsymbol{v})$.
(We have $\boldsymbol{p}=\boldsymbol{v}\sqcup\boldsymbol{u}$).
\begin{rem}
\label{rem:orientation2}We use the convention that the orientation
of the Heegaard surface is the \emph{same} as the boundary of the
$\alpha$-handlebody and is \emph{different }from the boundary of
the $\beta$-handlebody. Also see Remark \ref{rem:orientations}.
Note that \cite{MR2350128} uses a different convention (see \cite[Section 4.1]{MR3604486});
one should take the mirror of the links to translate to our convention.
\end{rem}

The sutured manifold corresponding to a suture datum is given as follows.
The underlying three-manifold is $Y\backslash N(L\cup\boldsymbol{v})$
where $N$ means tubular neighborhood. The boundary component $\partial N(L_{i})$
has two meridional sutures on $\partial N(L_{i})$, one for each basepoint
in $L_{i}\cap\boldsymbol{u}$; they are oriented such that $R_{-}$\footnote{Our convention is that $R_{-}$ corresponds to the $\alpha$-handlebody
part.} is the part that lies over the chosen component of $L_{i}\backslash\boldsymbol{u}$.
Also, each connected component of $\partial N(\boldsymbol{v})$ has
one suture.

Given a minimally pointed link $(L,\boldsymbol{u})$, there are three
different kinds of choices we can make: a $w|z$-partition, an orientation
on $L$, and a suture datum. Two of these determine the third: we
use the convention that the link goes from $z$-basepoints to $w$-basepoints
in the chosen components of $L\backslash\boldsymbol{u}$ (i.e. the
$\alpha$-handlebody). In particular, given a suture datum, a $w|z$-partition
is equivalent to an orientation on $L$. We say that the Heegaard
diagram $(\Sigma,\boldsymbol{\alpha},\boldsymbol{\beta},\boldsymbol{p})$
together with such choice $\boldsymbol{z},\boldsymbol{w}\subset\boldsymbol{p}$
\emph{represents the oriented link $\overrightarrow{L}\subset Y$}
(equipped with the $w|z$-partition, or a sutured datum).

\subsection{\label{subsec:spinc-link}${\rm Spin}^{c}$-structures for unoriented
links}

In this subsection, we discuss ${\rm Spin}^{c}$-structures for unoriented
links equipped with a suture datum in a three-manifold, and discuss
relative homological gradings of the corresponding Heegaard Floer
chain complex. 

A Heegaard diagram with more than two attaching curves defines a four-manifold
$X$ with boundary, and in our case, the link basepoints define an
embedded surface $S$. It should be possible to define the composition
map $\mu_{n}^{\mathfrak{s}}$ corresponding to a relative ${\rm Spin}^{c}$-structure
$\mathfrak{s}$ on $(X,S)$, such that $\mu_{n}=\sum_{\mathfrak{s}}\mu_{n}^{\mathfrak{s}}$.
We will not discuss this in this paper in full generality. However,
we will discuss two crude versions of this idea: one is in the form
of an Alexander $\mathbb{Z}/2$-splitting, which we introduce in Subsubsection
\ref{subsec:z/2-structure}; the other case is when the Heegaard diagram
is sufficiently simple, which we discuss in Subsection \ref{subsec:Homologous-attaching-curves}. 

\subsubsection{\label{subsec:G=00003Demptyset}The case $G=\emptyset$}

Let $L^{sut}=(L,\boldsymbol{u},\alpha_{L})$ be a minimally pointed
(unoriented) link equipped with a suture datum (Definition \ref{def:suture-data})
in a pointed three-manifold $(Y,\boldsymbol{v})$. Let us consider
a Heegaard datum whose Heegaard diagram $(\Sigma,\boldsymbol{\alpha},\boldsymbol{\beta},\boldsymbol{p})$
represents $L^{sut}$ in $(Y,\boldsymbol{v})$. In particular, assume
that there is no oriented arc $G$. Let $Y_{L^{sut}}$ be the sutured
manifold given by $L^{sut}$.

There are two different spaces of relative ${\rm Spin}^{c}$-structures
that we can consider. We can consider ${\rm Spin}^{c}(Y_{L^{sut}})$,
the space of \emph{${\rm Spin}^{c}$-structures on the sutured manifold
$Y_{L^{sut}}$,} as in \cite[Section 4]{MR2253454}, which comes with
a map $\mathfrak{s}:\boldsymbol{\alpha}\cap\boldsymbol{\beta}\to{\rm Spin}^{c}(Y_{L^{sut}})$
that specifies the ${\rm Spin}^{c}$-structure of intersection points.
We can also consider $\underline{{\rm Spin}^{c}}(Y,L)$, the space
of \emph{relative ${\rm Spin}^{c}$-structures on $Y\backslash N(L)$},
as in \cite[Section 3]{MR2443092}, which comes with a map $c_{1}:\underline{{\rm Spin}^{c}}(Y,L)\to H^{2}(Y,L;\mathbb{Z})$.
Both spaces are $H^{2}(Y,L;\mathbb{Z})$-torsors. 

Given an orientation $\mathfrak{o}$ on $L$, \cite[Section 3.6]{MR2443092}
gives an $H^{2}(Y,L;\mathbb{Z})$-torsor isomorphism $F_{\mathfrak{o}}:{\rm Spin}^{c}(Y_{L^{sut}})\to\underline{{\rm Spin}^{c}}(Y,L)$.
Hence, we can also define a map $\mathfrak{s}_{\mathfrak{o}}:=F_{\mathfrak{o}}\circ\mathfrak{s}:\boldsymbol{\alpha}\cap\boldsymbol{\beta}\to\underline{{\rm Spin}^{c}}(Y,L)$\footnote{In \cite{MR2443092}, this is denoted as $\mathfrak{s}_{\boldsymbol{w},\boldsymbol{z}}$
where $\boldsymbol{w},\boldsymbol{z}$ are the $w$-, $z$-basepoints,
respectively, which are determined by $\mathfrak{o}$ and the Heegaard
diagram.}. If $\mathfrak{o}'$ is another orientation on $L$, then $\mathfrak{s}_{\mathfrak{o}'}({\bf x})=\mathfrak{s}_{\mathfrak{o}}({\bf x})+PD(\mu)$
\cite[Lemma 3.12]{MR2443092} for some $\mu\in M$\footnote{The homology class $\mu$ is the sum of the meridians of components
of $L$ on which $\mathfrak{o}$ and $\mathfrak{o}'$ are different;
we do not need this fact.}, where $M$ is the $\mathbb{Z}$-linear subspace of $H_{1}(Y\backslash L;\mathbb{Z})$
spanned by the meridians of components of $L$. Hence, given $\mathfrak{s}\in{\rm Spin}^{c}(Y_{L^{sut}})$,
we can define its \emph{Chern class} $c_{1}(\mathfrak{s})\in H^{2}(Y;\mathbb{Z})$
as the image of $c_{1}(F_{\mathfrak{o}}(\mathfrak{s}))\in H^{2}(Y,L;\mathbb{Z})$
under the map $H^{2}(Y,L;\mathbb{Z})\to H^{2}(Y;\mathbb{Z})$, which
does not depend on the orientation $\mathfrak{o}$.

We do not need the entirety of ${\rm Spin}^{c}(Y_{L^{sut}})$: we
consider a quotient. Let us identify $H^{2}(Y,L;\mathbb{Z})\simeq H_{1}(Y\backslash L;\mathbb{Z})$.
\begin{defn}
\label{def:orbifold-spinc}Define the space of ${\rm Spin}^{c}$-structures
as ${\rm Spin}^{c}(Y(L^{sut})):={\rm Spin}^{c}(Y_{L^{sut}})/2M$,
which is an $H_{1}^{orb}(Y(L);\mathbb{Z}):=H_{1}(Y\backslash L;\mathbb{Z})/2M$-torsor.
The above map $\mathfrak{s}:\boldsymbol{\alpha}\cap\boldsymbol{\beta}\to{\rm Spin}^{c}(Y_{L^{sut}})$
descends to a map $\mathfrak{s}:\boldsymbol{\alpha}\cap\boldsymbol{\beta}\to{\rm Spin}^{c}(Y(L^{sut}))$.

For $\mathfrak{s}\in{\rm Spin}^{c}(Y(L^{sut}))$, define its \emph{Chern
class} $c_{1}(\mathfrak{s})\in H^{2}(Y;\mathbb{Z})$ as above. A ${\rm Spin}^{c}$-structure
$\mathfrak{s}$ is \emph{torsion} if $c_{1}(\mathfrak{s})$ is torsion.
\end{defn}

Recall that the coefficient ring $R$ is of the form $\mathbb{F}\llbracket U_{1}^{k_{1}},\cdots,U_{m}^{k_{m}}\rrbracket$
for $k_{i}\in\{1,1/2\}$. Recall that $k_{i}=1/2$ if and only if
it corresponds to a link component $L_{i}$ of $L$; let $\mu_{i}\in H_{1}^{orb}(Y(L);\mathbb{Z})$
be the meridian of $L_{i}$ in this case. If $k_{i}=1$, let $\mu_{i}=0$.
Let 
\[
\mathfrak{s}\left(U_{1}^{n_{1}}\cdots U_{m}^{n_{m}}{\bf x}\right):=\mathfrak{s}({\bf x})+\sum_{i}2n_{i}\mu_{i}\in{\rm Spin}^{c}(Y(L^{sut})).
\]
To talk about ${\rm Spin}^{c}(Y(L^{sut}))$-summands, we should view
$\boldsymbol{CF}^{-}(\boldsymbol{\alpha},\boldsymbol{\beta})$ as
an $R':=\mathbb{F}\llbracket U_{1},\cdots,U_{m}\rrbracket$-module.
\begin{defn}
\label{def:rprime-gen-gempty}An \emph{$R'$-generator }of $\boldsymbol{CF}^{-}(\boldsymbol{\alpha},\boldsymbol{\beta})$
is an element of the form $U_{1}^{n_{1}}\cdots U_{m}^{n_{m}}{\bf x}$
where $n_{i}\in\{0,1/2\}$, and ${\bf x}\in\boldsymbol{\alpha}\cap\boldsymbol{\beta}$.
Note that these form an $R'$-basis of $\boldsymbol{CF}^{-}(\boldsymbol{\alpha},\boldsymbol{\beta})$.

For $\mathfrak{s}\in{\rm Spin}^{c}(Y(L^{sut}))$, define $\boldsymbol{CF}^{-}(\boldsymbol{\alpha},\boldsymbol{\beta};\mathfrak{s})$
as the $R'$-sub chain complex\footnote{Indeed, this is a sub chain complex by \cite[Lemma 4.7]{MR2253454}.}
of $\boldsymbol{CF}^{-}(\boldsymbol{\alpha},\boldsymbol{\beta})$
generated by the $R'$-generators $\boldsymbol{y}$ such that $\mathfrak{s}(\boldsymbol{y})=\mathfrak{s}$.
\end{defn}

Recall from \cite[Section 2.6]{MR2113019}, \cite[Section 4.2]{1512.01184}
that we can define $\mathfrak{s}_{\boldsymbol{r}}:\boldsymbol{\alpha}\cap\boldsymbol{\beta}\to{\rm Spin}^{c}(Y)$
if $\boldsymbol{r}\subset\Sigma$ is a set of points such that each
connected component of $\Sigma\backslash\boldsymbol{\alpha}$ and
$\Sigma\backslash\boldsymbol{\beta}$ contains exactly one point in
$\boldsymbol{r}$. The Chern class of $\mathfrak{s}({\bf x})$ can
be computed in terms of the Chern class of $\mathfrak{s}_{\boldsymbol{r}}({\bf x})$:
indeed, choose an orientation on $L$ and let $\overrightarrow{L}$
be the corresponding oriented link, and let $\boldsymbol{w},\boldsymbol{z}$
be the corresponding $w|z$-partition. Let $\boldsymbol{v}$ be the
free basepoints. Then, \cite[Lemma 3.13]{MR2443092} (also see \cite[Lemma 3.3]{MR3905679})
implies
\begin{equation}
c_{1}(\mathfrak{s}({\bf x}))=c_{1}(\mathfrak{s}_{\boldsymbol{v}\sqcup\boldsymbol{w}}({\bf x}))-PD(\overrightarrow{L})=c_{1}(\mathfrak{s}_{\boldsymbol{v}\sqcup\boldsymbol{z}}({\bf x}))+PD(\overrightarrow{L}).\label{eq:c1-average}
\end{equation}

\begin{lem}
\label{lem:cornerless-maslov}Let $(\Sigma,\boldsymbol{\alpha},\boldsymbol{\beta},\boldsymbol{p})$
be a Heegaard diagram, and let ${\cal D}\in D({\bf x},{\bf x})$ be
a cornerless domain. Then, 
\[
\mu({\cal D})=\left\langle c_{1}(\mathfrak{s}({\bf x})),H({\cal D})\right\rangle +P({\cal D}),
\]
where $P({\cal D})$ is the total multiplicity of ${\cal D}$ (Definition
\ref{def:total-multiplicity}).
\end{lem}

\begin{proof}
This is a direct consequence of the Maslov index formula for cornerless
domains in Heegaard diagrams for three-manifolds (\cite[Proposition 7.5]{MR2113020},
\cite[Corollary 4.12]{MR2240908}, \cite[Equation (4.9)]{1512.01184})
and Equation (\ref{eq:c1-average}). (Compare \cite[Section 5.2]{1711.07110}.)
\end{proof}
\begin{cor}
Given a weakly admissible Heegaard datum with Heegaard diagram $(\Sigma,\boldsymbol{\alpha},\boldsymbol{\beta},\boldsymbol{p})$,
the chain complex $\boldsymbol{CF}^{-}(\boldsymbol{\alpha},\boldsymbol{\beta};\mathfrak{s})$
has a relative homological $\mathbb{Z}/\mathfrak{d}(c_{1}(\mathfrak{s}))$-grading,
where $\mathfrak{d}(c_{1}(\mathfrak{s}))$ is the divisibility of
$c_{1}(\mathfrak{s})$.
\end{cor}

\begin{proof}
Immediate from Lemma \ref{lem:cornerless-maslov}.
\end{proof}
\begin{example}
\label{exa:simple-obstruction}Consider the weakly admissible Heegaard
diagram $(\mathbb{T}^{2},\alpha,\beta,\{w,z\})$ given by the left
side of Figure \ref{fig:genus1-ab}. Work over $\mathbb{F}\llbracket U^{1/2}\rrbracket$,
and assign the weight $U^{1/2}$ to both $w$ and $z$. The group
$H_{1}^{orb}(Y(L);\mathbb{Z})\simeq\mathbb{Z}\oplus\mathbb{Z}/2$;
the meridian $\mu$ represents $(0,1)$. The ${\rm Spin}^{c}$-structures
$\mathfrak{s}(\sigma)$ and $\mathfrak{s}(\tau)$ are both torsion
(in fact $c_{1}=0$), and they differ by $\mu$. We have 
\[
\boldsymbol{CF}^{-}(\alpha,\beta;\mathfrak{s}(\sigma))=\sigma\mathbb{F}\llbracket U\rrbracket\oplus U^{1/2}\tau\mathbb{F}\llbracket U\rrbracket,\ \boldsymbol{CF}^{-}(\alpha,\beta;\mathfrak{s}(\tau))=\tau\mathbb{F}\llbracket U\rrbracket\oplus U^{1/2}\sigma\mathbb{F}\llbracket U\rrbracket.
\]
The differential is identically $0$.

Now, consider the right side of Figure \ref{fig:genus1-ab}. The group
$H_{1}^{orb}(Y(L);\mathbb{Z})\simeq H_{1}(Y;\mathbb{Z})\simeq\mathbb{Z}$,
and $\mathfrak{s}(\sigma)=\mathfrak{s}(\tau)$. Their Chern class
$c_{1}(\mathfrak{s}(\sigma))$ is a generator of $H_{1}(Y;\mathbb{Z})$;
in particular, it has divisibility $1$. Indeed, there are two domains
in $D(\tau,\sigma)$ with weights $1,U^{1/2}$, respectively.
\end{example}

\begin{figure}[h]
\begin{centering}
\includegraphics[scale=1.5]{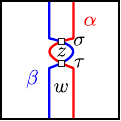}\qquad{}\includegraphics[scale=1.5]{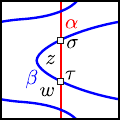}
\par\end{centering}
\caption{\label{fig:genus1-ab}Some simple doubly pointed Heegaard diagrams}
\end{figure}

\subsubsection{\label{subsec:Gneq0}The case $G\protect\neq\emptyset$}

Let us consider the case where $G$ may exist. In this case, we can
formally work with the Heegaard diagram $(\Sigma,\boldsymbol{\alpha},\boldsymbol{\beta},\boldsymbol{p}\sqcup\{\partial_{+}G\})$;
define the corresponding $L^{sut}$ and $\mathfrak{s}:\boldsymbol{\alpha}\cap\boldsymbol{\beta}\to{\rm Spin}^{c}(Y(L^{sut}))$
accordingly.

If at most one of $\boldsymbol{\alpha}$ and $\boldsymbol{\beta}$
intersects $G$, then Remark \ref{rem:same-as-unoriented} justifies
this. However, if both $\boldsymbol{\alpha}$ and $\boldsymbol{\beta}$
intersect $G$, then it turns out that the link $L$ is not the link
we want: consider the projection map $\pi:\overline{N}(\Sigma)\to\Sigma$,
and assume that $L$ is vertical in $\overline{N}(\Sigma)$, i.e.
that $L\cap\overline{N}(\Sigma)$ is the inverse image of the link
basepoints on $\Sigma$, under $\pi$. Then $\pi^{-1}(G)$ is a split
band on $L$; the link we should consider is the link obtained by
surgering $L$ along $\pi^{-1}(G)$. We will not try to fix this;
for our purposes, it is sufficient to consider Chern class summands
as follows.
\begin{defn}
\label{def:spinc-g-nontrivial}For $c\in H^{2}(Y;\mathbb{Z})$, define
$\boldsymbol{CF}^{-}(\boldsymbol{\alpha}^{E_{\boldsymbol{\alpha}}},\boldsymbol{\beta}^{E_{\boldsymbol{\beta}}};c)$
as the $R$-sub chain complex of $\boldsymbol{CF}^{-}(\boldsymbol{\alpha}^{E_{\boldsymbol{\alpha}}},\boldsymbol{\beta}^{E_{\boldsymbol{\beta}}})$
generated by the intersection points ${\bf x}$ such that $c_{1}(\mathfrak{s}({\bf x}))=c$.
\end{defn}

\subsubsection{\label{subsec:Strong-admissibility-2}Strong $\mathfrak{s}$-admissibility
and $CF^{-}$}

Finally, let us comment on how to define $CF^{-}$.
\begin{defn}
\label{def:strongly-admissible}The Heegaard diagram $(\Sigma,\boldsymbol{\alpha},\boldsymbol{\beta},\boldsymbol{p})$
is \emph{$\mathfrak{s}$-strongly admissible} if all cornerless two-chains
${\cal D}$ such that 
\[
\left\langle c_{1}(\mathfrak{s}),H({\cal D})\right\rangle +P({\cal D})=0
\]
have both positive and negative local multiplicities, where $H({\cal D})\in H_{2}(Y;\mathbb{Z})$
is the homology class that ${\cal D}$ represents.
\end{defn}

Note that a weakly admissible Heegaard diagram is $\mathfrak{s}$-strongly
admissible for all torsion $\mathfrak{s}$.

If $(\Sigma,\boldsymbol{\alpha},\boldsymbol{\beta},\boldsymbol{p})$
is $\mathfrak{s}$-strongly admissible, then by Lemma \ref{lem:cornerless-maslov},
we can work over the polynomial ring $\mathbb{F}[U_{1}^{k_{1}},\cdots,U_{m}^{k_{m}}]$
instead and define $CF^{-}(\boldsymbol{\alpha},\boldsymbol{\beta};\mathfrak{s})$.
If $G\neq\emptyset$, similarly define $CF^{-}(\boldsymbol{\alpha}^{E_{\boldsymbol{\alpha}}},\boldsymbol{\beta}^{E_{\boldsymbol{\beta}}};c)$.
\begin{rem}
Equation (\ref{eq:c1-average}) says that our strong admissibility
condition considers the average of the ${\rm Spin}^{c}$-structures
obtained by the different choices of the $z$-basepoints. This has
the advantage that the left hand side of Figure \ref{fig:genus1-ab}
is strongly admissible with respect to both torsion ${\rm Spin}^{c}$-structures.
In contrast, it is not even weakly admissible with respect to $z$
(or $w$).
\end{rem}

\subsection{\label{subsec:Gradings}Gradings}

Gradings in $A_{\infty}$-categories, if they exist, will turn out
to be very useful. For instance, we will use gradings to show that
some generators do not arise as summands of certain composition maps.

\subsubsection{A homological $\mathbb{Z}$-grading}

Sometimes, the $A_{\infty}$-categories we consider can be given a
homological $\mathbb{Z}$-grading, in the sense of Subsection \ref{subsec:Twisted-complexes}.
Compare \cite[Section 3]{MR2509750}, \cite[Section 5.2]{MR2964628}.
Note that our notions differ: their periodic domains are cornerless
two-chains in this paper.

\begin{defn}
\label{def:gradable}A Heegaard diagram is \emph{(homologically) $\mathbb{Z}$-gradable}
if whenever ${\cal D}$ and ${\cal D}'$ are two domains with the
same vertices, then 
\begin{equation}
P({\cal D})-\mu({\cal D})=P({\cal D}')-\mu({\cal D}').\label{eq:hom-gradable}
\end{equation}
\end{defn}

Lemma \ref{lemma:deg-output} implies the following proposition.
\begin{prop}
\label{prop:Z-grading}If a weakly admissible Heegaard datum is homologically
$\mathbb{Z}$-gradable, then the corresponding $A_{\infty}$-category
${\cal A}$ (Remark \ref{rem:ainf-pedantry}) can be (homologically)
$\mathbb{Z}$-graded (i.e. $\mu_{d}$ has degree $d-2$), such that:
\begin{itemize}
\item each $R$-generator $f{\bf x}\in\boldsymbol{CF}^{-}(\boldsymbol{\alpha}^{E_{\boldsymbol{\alpha}}},\boldsymbol{\beta}^{E_{\boldsymbol{\beta}}})$
is homogeneous;
\item $U_{i}^{k}$ has degree $-2k$, i.e. if $\boldsymbol{y}$ is homogeneous,
then ${\rm gr}_{{\cal A}}(U_{i}^{k}\boldsymbol{y})={\rm gr}_{{\cal A}}(\boldsymbol{y})-2k$
for $k\in\mathbb{Z}$ or $1/2\mathbb{Z}$;
\item if $f{\bf x}$ and $g{\bf x}$ are generators, then ${\rm gr}_{{\cal A}}(f{\bf x})-{\rm gr}_{{\cal A}}(g{\bf x})=-2({\rm gr}_{U}(f)-{\rm gr}_{U}(g))$.
\end{itemize}
\end{prop}

Similarly, we say that a Heegaard diagram is homologically $\mathbb{Z}/n$-gradable
if Equation (\ref{eq:hom-gradable}) holds modulo $n$. If it is homologically
$\mathbb{Z}/n$-gradable, then ${\cal A}$ can be (homologically)
$\mathbb{Z}/n$-graded.

Note that if the attaching curves are $\boldsymbol{\alpha}_{0}^{E_{0}},\cdots,\boldsymbol{\alpha}_{m}^{E_{m}}$,
then we are free to shift the homological gradings of each $\boldsymbol{CF}^{-}(\boldsymbol{\alpha}_{i}^{E_{i}},\boldsymbol{\alpha}_{i+1}^{E_{i+1}})$
by a constant $C_{i}$.
\begin{defn}
We use the convention that if $\boldsymbol{\alpha}_{j}$ is a small,
standard translate of $\boldsymbol{\alpha}_{i}$, then $\Theta^{+}\in\boldsymbol{CF}^{-}(\boldsymbol{\alpha}_{i}^{E_{i}},\boldsymbol{\alpha}_{j}^{E_{j}})$
has homological grading $0$\footnote{If the local systems are trivial, then $\Theta^{+}$ can be viewed
as an element of $\boldsymbol{CF}^{-}(\#^{N}S^{1}\times S^{2})$.
Our convention is different from the absolute $\mathbb{Q}$-grading
of $\Theta^{+}$, which is $N/2$.}.
\end{defn}

\subsubsection{\label{subsec:z/2-structure}An Alexander $\mathbb{Z}/2$-splitting}

Sometimes, the $A_{\infty}$-category can be given an \emph{Alexander
$\mathbb{Z}/2$-splitting} in the sense of Proposition \ref{prop:Z/2-splitting}.
For links in $S^{3}$, this is the relative, collapsed Alexander $\mathbb{Z}/2$-grading,
and in general, it is closely related to the ${\rm Spin}^{c}$-splitting
discussed in Subsection \ref{subsec:spinc-link}.
\begin{defn}
\label{def:z/2-gradable}A Heegaard diagram is \emph{(Alexander) $\mathbb{Z}/2$-splittable}
if whenever ${\cal D}$ and ${\cal D}'$ are two domains with the
same vertices, we have $P({\cal D})=P({\cal D}')$ modulo $2$.
\end{defn}

To talk about the corresponding splitting in the $A_{\infty}$-category,
we consider $\boldsymbol{CF}^{-}(\boldsymbol{\alpha}^{E_{\boldsymbol{\alpha}}},\boldsymbol{\beta}^{E_{\boldsymbol{\beta}}})$
not as an $R$-module, but as an $R':=\mathbb{F}\llbracket U_{1},\cdots,U_{m}\rrbracket$-module,
as in Subsection \ref{subsec:spinc-link}. If $G$ exists, then one
of the $U_{i}$'s is the distinguished variable $U$.
\begin{defn}
An \emph{$R'$-generator }of $\boldsymbol{CF}^{-}(\boldsymbol{\alpha}^{E_{\boldsymbol{\alpha}}},\boldsymbol{\beta}^{E_{\boldsymbol{\beta}}})$
is of the form $fU_{1}^{n_{1}}\cdots U_{k}^{n_{k}}{\bf x}$ where
$f$ is an $R$-basis element of ${\rm Hom}_{R}(E_{\boldsymbol{\alpha}},E_{\boldsymbol{\beta}})$,
$n_{i}\in\{0,1/2\}$, and ${\bf x}\in\boldsymbol{\alpha}\cap\boldsymbol{\beta}$.
We write $U^{{\bf n}}=U_{1}^{n_{1}}\cdots U_{k}^{n_{k}}$ for ${\bf n}=(n_{1},\cdots,n_{k})$.
Note that these form an $R'$-basis of $\boldsymbol{CF}^{-}(\boldsymbol{\alpha}^{E_{\boldsymbol{\alpha}}},\boldsymbol{\beta}^{E_{\boldsymbol{\beta}}})$.

Let $f_{i}U^{{\bf n}_{i}}{\bf x}_{i}\in\boldsymbol{CF}^{-}(\boldsymbol{\alpha}_{i-1}^{E_{i-1}},\boldsymbol{\alpha}_{i}^{E_{i}})$
for $i=1,\cdots,d$ and $\boldsymbol{y}_{d+1}=f_{d+1}U^{{\bf n}_{d+1}}{\bf x}_{d+1}\in\boldsymbol{CF}^{-}(\boldsymbol{\alpha}_{0}^{E_{0}},\boldsymbol{\alpha}_{d}^{E_{d}})$
be $R'$-generators. Let $D(f_{1}U^{{\bf n}_{1}}{\bf x}_{1},\cdots,f_{d+1}U^{{\bf n}_{d+1}}{\bf x}_{d+1})$
be the set of domains ${\cal D}\in D({\bf x}_{1},\cdots,{\bf x}_{d+1})$
such that $w({\cal D})\rho({\cal D})(f_{1}U^{{\bf n}_{1}}\otimes\cdots\otimes f_{d}U^{{\bf n}_{d}})=rf_{d+1}U^{{\bf n}_{d+1}}$
for some nonzero $r\in R'$.
\end{defn}

A similar idea as above gives the following proposition.
\begin{prop}
\label{prop:Z/2-splitting}If a weakly admissible Heegaard datum is
Alexander $\mathbb{Z}/2$-splittable, then the associated $A_{\infty}$-category
${\cal A}$ can be equipped with an \emph{(Alexander) $\mathbb{Z}/2$-splitting}
in the following sense:
\begin{itemize}
\item The chain complexes $\boldsymbol{CF}^{-}(\boldsymbol{\alpha}^{E_{\boldsymbol{\alpha}}},\boldsymbol{\beta}^{E_{\boldsymbol{\beta}}})=C_{0}\oplus C_{1}$
are $\mathbb{Z}/2$-graded, and satisfies the three conditions in
Proposition \ref{prop:Z-grading}. In particular, the summands $C_{0}$
and $C_{1}$ are $R'$-modules. We denote this grading as ${\rm gr}_{A}^{\mathbb{Z}/2}(\boldsymbol{y})$
for a homogeneous element $\boldsymbol{y}$, and call it the\emph{
Alexander $\mathbb{Z}/2$-grading}.
\item Let $\boldsymbol{y}_{i}\in\boldsymbol{CF}^{-}(\boldsymbol{\alpha}_{i-1}^{E_{i-1}},\boldsymbol{\alpha}_{i}^{E_{i}})$
for $i=1,\cdots,d$ and $\boldsymbol{y}_{d+1}\in\boldsymbol{CF}^{-}(\boldsymbol{\alpha}_{0}^{E_{0}},\boldsymbol{\alpha}_{d}^{E_{d}})$
be $R'$-generators. If there exists a domain ${\cal D}\in D(\boldsymbol{y}_{1},\cdots,\boldsymbol{y}_{d+1})$,
then
\[
\sum_{i=1}^{d+1}{\rm gr}_{A}^{\mathbb{Z}/2}(\boldsymbol{y}_{i})=0.
\]
In particular, the composition maps $\mu_{d}$ have Alexander $\mathbb{Z}/2$-degree
$0$.
\end{itemize}
\end{prop}

We use the convention that $\Theta^{+}$ has Alexander $\mathbb{Z}/2$-grading
$0$.
\begin{example}
Recall Example \ref{exa:simple-obstruction}. The left hand side of
Figure \ref{fig:genus1-ab} is Alexander $\mathbb{Z}/2$-splittable,
and the Alexander $\mathbb{Z}/2$-splitting on $\boldsymbol{CF}^{-}(\alpha,\beta)$
is $\boldsymbol{CF}^{-}(\alpha,\beta;\mathfrak{s}(\sigma))\oplus\boldsymbol{CF}^{-}(\alpha,\beta;\mathfrak{s}(\tau))$.
\end{example}

\begin{rem}
Note that Alexander $\mathbb{Z}/2$-splittings do not see the Maslov
index of domains: they only see whether domains exist. The second
condition of Proposition \ref{prop:Z-grading} implies that if the
$A_{\infty}$-category is equipped with an Alexander $\mathbb{Z}/2$-splitting,
then the $A_{\infty}$-functor ${\cal F}$ induced by changing the
almost complex structure preserves the splitting, i.e. ${\cal F}_{d}$
has Alexander $\mathbb{Z}/2$-degree $0$ for all $d$.
\end{rem}

\begin{rem}
We consider twisted complexes in this setting as well, in which case
we add the condition that the differential $\delta$ is homogeneous
of degree $0$. The composition maps $\mu_{d}^{{\rm Tw}}$ and the
functor induced by changing the almost complex structure are homogeneous
of degree $0$.
\end{rem}

\subsection{\label{subsec:Homologous-attaching-curves}Handlebody-equivalent
attaching curves}
\begin{defn}
Two attaching curves $\boldsymbol{\beta}_{1},\boldsymbol{\beta}_{2}$
on a Heegaard surface $\Sigma$ are \emph{handlebody-equivalent} if
they define the same handlebody (ignoring the basepoints on $\Sigma$).
\end{defn}

In this subsection, let us consider a Heegaard diagram $(\Sigma,\boldsymbol{\alpha}_{0},\cdots,\boldsymbol{\alpha}_{a},\boldsymbol{\beta}_{0},\cdots,\boldsymbol{\beta}_{b},\boldsymbol{p})$
(and possibly also $G$) such that the $\boldsymbol{\alpha}_{i}$'s
are pairwise handlebody-equivalent and also the $\boldsymbol{\beta}_{i}$'s
are pairwise handlebody-equivalent\footnote{If all the attaching curves are handlebody-equivalent, then we also
are implicitly making a choice of $a$; i.e. we choose which ones
are the $\boldsymbol{\alpha}_{i}$'s and which are the $\boldsymbol{\beta}_{j}$'s.}. Under this assumption, it is easier to check whether the Heegaard
diagram is homologically $\mathbb{Z}$-gradable or Alexander $\mathbb{Z}/2$-splittable.
Similar statements to the results in this subsection hold when the
attaching curves are cyclically permuted.
\begin{rem}
It is possible to relax ``handlebody-equivalent'' to ``\emph{homologous},''
where two attaching curves $\boldsymbol{\beta}_{1},\boldsymbol{\beta}_{2}$
on a Heegaard surface $\Sigma$ are \emph{homologous} if they span
the same subspace of $H_{1}(\Sigma)$.
\end{rem}

\begin{lem}
\label{lem:domain-lemma}If ${\cal D},{\cal D}'\in D({\bf x}_{0},\cdots,{\bf x}_{a+b+1})$
are domains with the same vertices, then there exist cornerless domains
${\cal P}_{j}\in D({\bf x}_{j},{\bf x}_{j})$ for $j\in J$ where
$J=\{0,\cdots,a+b+1\}\backslash\{a+b+1\}$ or $\{0,\cdots,a+b+1\}\backslash\{a\}$,
such that 
\[
{\cal D}={\cal D}'+\sum_{j\in J}{\cal P}_{j}.
\]
\end{lem}

\begin{proof}
Let us consider the case where $J=\{0,\cdots,a+b+1\}\backslash\{a+b+1\}$;
the other case follows similarly. Let ${\cal P}$ be the difference
${\cal D}-{\cal D}'$ of the underlying two-chains. Recursively define
${\cal P}_{a+j}$ for $j=b,\cdots,1$ such that the $\boldsymbol{\beta}_{i}$-boundary
of the two-chain ${\cal P}-\sum_{\ell=a+j}^{a+b}{\cal P}_{\ell}$
is empty for $i\ge j$. Similarly define ${\cal P}_{j}$ for $j=0,\cdots,a-1$
such that (all the $\boldsymbol{\beta}_{k}$-boundaries and) the $\boldsymbol{\alpha}_{i}$-boundary
of the two-chain ${\cal P}-\sum_{\ell=a+1}^{a+b}{\cal P}_{\ell}-\sum_{\ell=0}^{j}{\cal P}_{\ell}$
is empty for $i\le j$. Let ${\cal P}_{a}$ be the cornerless $\boldsymbol{\alpha}_{a}\boldsymbol{\beta}_{0}$-domain
with underlying two-chain ${\cal P}-\sum_{j\in J\backslash\{a\}}{\cal P}_{j}$.
\end{proof}
Identify the alpha-handlebodies and also identify the beta-handlebodies,
and let $Y$ be the three-manifold given by the alpha-handlebody and
the beta-handlebody. Define the homology class $H({\cal D})$ of a
cornerless two-chain ${\cal D}$ as an element in $H_{2}(Y;\mathbb{Z})$.
We get the following as a corollary.
\begin{lem}
\label{lem:cornerless-maslov-1}If ${\cal D},{\cal D}'\in D({\bf x}_{0},\cdots,{\bf x}_{a+b+1})$
are domains with the same vertices, then 
\[
\mu({\cal D})-\mu({\cal D}')=\left\langle c(\mathfrak{s}({\bf x}_{a})),H({\cal D}-{\cal D}')\right\rangle +P({\cal D})-P({\cal D}').
\]
\end{lem}

\begin{proof}
By Lemma \ref{lem:domain-lemma}, there exist cornerless domains ${\cal P}_{k}\in D({\bf x}_{k},{\bf x}_{k})$
for $k=0,\cdots,a+b$ such that ${\cal D}={\cal D}'+\sum_{k=0}^{a+b}{\cal P}_{k}$.
Since $\mu$ and $P$ are additive, we have (by Lemma \ref{lem:cornerless-maslov})
\[
(\mu({\cal D})-P({\cal D}))-(\mu({\cal D}')-P({\cal D}'))=\sum_{k=0}^{a+b}\left\langle c_{1}(\mathfrak{s}({\bf x}_{k})),H({\cal P}_{k})\right\rangle =\left\langle c(\mathfrak{s}({\bf x}_{a})),H({\cal P}_{a})\right\rangle .
\]
The lemma follows since $H({\cal D}-{\cal D}')=H({\cal P}_{a})\in H_{2}(Y;\mathbb{Z})$.
\end{proof}
\begin{lem}
\label{lem:chern-class-additive}Let ${\cal D}\in D({\bf x}_{0},\cdots,{\bf x}_{a+b+1})$
be a domain. If $c_{1}(\mathfrak{s}({\bf x}_{j}))=0$ for $j\in\{0,\cdots,a+b+1\}\backslash\{a,a+b+1\}$,
then $c_{1}(\mathfrak{s}({\bf x}_{a}))=c_{1}(\mathfrak{s}({\bf x}_{a+b+1}))$.
\end{lem}

\begin{proof}
Let $\boldsymbol{v}$ be the set of free basepoints. If $G$ exists,
for the purpose of this Lemma, consider $\partial_{+}G,\partial_{-}G$
as link basepoints and replace $\boldsymbol{p}$ with $\boldsymbol{p}\sqcup\{\partial_{+}G\}$,
and let us choose which link basepoints are $z$-basepoints; write
$\boldsymbol{p}=\boldsymbol{v}\sqcup\boldsymbol{w}\sqcup\boldsymbol{z}$. 

We can define the $\boldsymbol{v}\sqcup\boldsymbol{w}$-${\rm Spin}^{c}$
structure of ${\cal D}$ and its Chern class $c_{1}(\mathfrak{s}_{\boldsymbol{v}\sqcup\boldsymbol{w}}({\cal D}))\in H^{2}(X;\mathbb{Z})$.
The key property we use is that $c_{1}(\mathfrak{s}_{\boldsymbol{v}\sqcup\boldsymbol{w}}({\cal D}))$
restricts to $c_{1}(\mathfrak{s}_{\boldsymbol{v}\sqcup\boldsymbol{w}}({\bf x}_{i}))\in H^{2}(Y_{\boldsymbol{\delta\epsilon}};\mathbb{Z})$
where ${\bf x}_{i}\in\boldsymbol{\delta}\cap\boldsymbol{\epsilon}$.
Let $\overrightarrow{L_{\boldsymbol{\delta\epsilon}}}\subset Y_{\boldsymbol{\delta\epsilon}}$
be the corresponding oriented link; then we have $c_{1}(\mathfrak{s}_{\boldsymbol{v}\sqcup\boldsymbol{w}}({\bf x}_{i}))=c_{1}(\mathfrak{s}({\bf x}_{i}))+PD(\overrightarrow{L_{\boldsymbol{\delta\epsilon}}})$
by Equation (\ref{eq:c1-average}).

Recall the identification
\[
H^{2}(X;\mathbb{Z})\simeq{\rm coker}\left(H_{1}(\Sigma;\mathbb{Z})\to\bigoplus_{\boldsymbol{\gamma}}H_{1}(U_{\boldsymbol{\gamma}};\mathbb{Z})\right),\ H^{2}(Y_{\boldsymbol{\delta\epsilon}};\mathbb{Z})\simeq{\rm coker}\left(H_{1}(\Sigma;\mathbb{Z})\to H_{1}(U_{\boldsymbol{\delta}};\mathbb{Z})\oplus H_{1}(U_{\boldsymbol{\epsilon}};\mathbb{Z})\right),
\]
where $U_{\boldsymbol{\gamma}}$ is the $\boldsymbol{\gamma}$-handlebody.
For each pair of link basepoints, choose an oriented arc on $\Sigma$
from the $w$-basepoint to the $z$-basepoint, and let $\overrightarrow{L_{\boldsymbol{\gamma}}}\in H_{1}(U_{\boldsymbol{\gamma}};\mathbb{Z})$
be obtained by adding the above oriented arc to an oriented arc on
$\Sigma\backslash\boldsymbol{\gamma}$ from the $z$-basepoint to
the $w$-basepoint. Then the oriented link $\overrightarrow{L_{\boldsymbol{\delta\epsilon}}}\in H_{1}(Y_{\boldsymbol{\delta\epsilon}};\mathbb{Z})\simeq H^{2}(Y_{\boldsymbol{\delta\epsilon}};\mathbb{Z})$
is represented by the class $(\overrightarrow{L_{\boldsymbol{\delta}}},\overrightarrow{L_{\boldsymbol{\epsilon}}})$
under the above identification. Using the above identification together
with that the restriction map $H^{2}(X;\mathbb{Z})\to H^{2}(Y_{\boldsymbol{\delta\epsilon}};\mathbb{Z})$
is induced by the projection map $\bigoplus_{\boldsymbol{\gamma}}H_{1}(U_{\boldsymbol{\gamma}};\mathbb{Z})\to H_{1}(U_{\boldsymbol{\delta}};\mathbb{Z})\oplus H_{1}(U_{\boldsymbol{\epsilon}};\mathbb{Z})$,
one can write $c_{1}(\mathfrak{s}_{\boldsymbol{v}\sqcup\boldsymbol{w}}({\bf x}_{a+b+1}))$
in terms of $c_{1}(\mathfrak{s}_{\boldsymbol{v}\sqcup\boldsymbol{w}}({\bf x}_{a}))$
and the $\overrightarrow{L_{\boldsymbol{\gamma}}}$'s. Combining everything,
we get $c_{1}(\mathfrak{s}({\bf x}_{a}))=c_{1}(\mathfrak{s}({\bf x}_{a+b+1}))$.
\end{proof}
The above proof also shows that if the hypothesis on $c_{1}$ is removed,
then there is a relation between the $c_{1}(\mathfrak{s}({\bf x}_{j}))$'s.

\begin{prop}
\label{prop:torsion}Consider a weakly admissible Heegaard datum whose
underlying Heegaard diagram is $(\Sigma,\boldsymbol{\alpha}_{0},\cdots,\boldsymbol{\alpha}_{a},\boldsymbol{\beta}_{0},\cdots,\boldsymbol{\beta}_{b},\boldsymbol{p})$
(it might also have $G$) such that the $\boldsymbol{\alpha}_{i}$'s
are pairwise handlebody-equivalent and also the $\boldsymbol{\beta}_{i}$'s
are pairwise handlebody-equivalent. Identify the alpha-handlebodies
and also identify the beta-handlebodies, let $Y$ be the three-manifold
given by the alpha-handlebody and the beta-handlebody, and let $c\in H^{2}(Y;\mathbb{Z})$.
Consider the $A_{\infty}$-category ${\cal A}$ whose objects are
the attaching curves (equipped with local systems), and 
\begin{gather*}
{\rm Hom}_{{\cal A}}(\boldsymbol{\alpha}_{k}^{E_{\boldsymbol{\alpha}_{k}}},\boldsymbol{\beta}_{\ell}^{E_{\boldsymbol{\beta}_{\ell}}})=\boldsymbol{CF}^{-}(\boldsymbol{\alpha}_{k}^{E_{\boldsymbol{\alpha}_{k}}},\boldsymbol{\beta}_{\ell}^{E_{\boldsymbol{\beta}_{\ell}}};c)\\
{\rm Hom}_{{\cal A}}(\boldsymbol{\alpha}_{i}^{E_{\boldsymbol{\alpha}_{i}}},\boldsymbol{\alpha}_{j}^{E_{\boldsymbol{\alpha}_{j}}})=\boldsymbol{CF}^{-}(\boldsymbol{\alpha}_{i}^{E_{\boldsymbol{\alpha}_{i}}},\boldsymbol{\alpha}_{j}^{E_{\boldsymbol{\alpha}_{j}}};0),\ {\rm Hom}_{{\cal A}}(\boldsymbol{\beta}_{i}^{E_{\boldsymbol{\beta}_{i}}},\boldsymbol{\beta}_{j}^{E_{\boldsymbol{\beta}_{j}}})=\boldsymbol{CF}^{-}(\boldsymbol{\beta}_{i}^{E_{\boldsymbol{\beta}_{i}}},\boldsymbol{\beta}_{j}^{E_{\boldsymbol{\beta}_{j}}};0),
\end{gather*}
for all $k,\ell$ and $i<j$; if $i\ge j$, then define ${\rm Hom}_{{\cal A}}(\boldsymbol{\alpha}_{i}^{E_{\boldsymbol{\alpha}_{i}}},\boldsymbol{\alpha}_{j}^{E_{\boldsymbol{\alpha}_{j}}})={\rm Hom}_{{\cal A}}(\boldsymbol{\beta}_{i}^{E_{\boldsymbol{\beta}_{i}}},\boldsymbol{\beta}_{j}^{E_{\boldsymbol{\beta}_{j}}})=0$.

Then, ${\cal A}$ is homologically $\mathbb{Z}/\mathfrak{d}(c)$-gradable.
\end{prop}

\begin{proof}
The composition maps are well-defined and ${\cal A}$ is an $A_{\infty}$-category
by Lemma \ref{lem:chern-class-additive}. It is homologically $\mathbb{Z}/\mathfrak{d}(c)$-gradable
by Lemma \ref{lem:cornerless-maslov-1}.
\end{proof}

\subsubsection{\label{subsec:Strong-admissibility-c}Strong $c$-admissibility}

We can define strong $c$-admissibility in the context of this subsection.
Compare \cite[Definition 8.8 and Section 8.4.2]{MR2113019} and \cite[Section 4.8]{1512.01184}.
\begin{defn}
\label{def:strongly-admissible-multi}Let the Heegaard diagram and
the three-manifold $Y$ be as in Proposition \ref{prop:torsion}.
The Heegaard diagram is \emph{$c$-strongly admissible} for $c\in H^{2}(Y;\mathbb{Z})$
if every cornerless two-chain ${\cal D}$ such that 
\[
\left\langle c,H({\cal D})\right\rangle +P({\cal D})=0
\]
has both positive and negative local multiplicities.
\end{defn}

If a Heegaard diagram is $c$-strongly admissible, recall from Subsubsection
\ref{subsec:Strong-admissibility-2} that we can work over polynomial
rings and define $CF^{-}$. In this setting, $\mu_{d}$ is well-defined
by Lemmas \ref{lem:cornerless-maslov-1} and \ref{lem:chern-class-additive},
the $A_{\infty}$-relations hold, and Proposition \ref{prop:torsion}
works for $CF^{-}$.

\section{\label{sec:band-maps}Band maps in unoriented link Floer homology}

The goal of this section is to precisely define the objects we consider,
to define the band maps, and to show that they are well-defined. Compare
\cite[Section 6]{MR3905679} and \cite{1711.07110}. Note that the
objects we consider are different, and hence the band maps are different
as well. The band maps we consider are different from \cite{1711.07110}
even for non-orientable bands: Fan chooses ``the other generator,''
i.e. the generator that corresponds to $\sigma$ in the left hand
side of Figure \ref{fig:local-band}.

We first define the unoriented link Floer chain complex of a link
in a three-manifold, together with various decorations.
\begin{defn}
\label{def:unoriented-suture-datum}Let $(L,\boldsymbol{u},\alpha_{L})$
be a minimally pointed link equipped with a suture datum, inside a
pointed three-manifold $(Y,\boldsymbol{v})$. Define an equivalence
relation on $\boldsymbol{u}\sqcup\boldsymbol{v}$ by declaring $x\sim y$
if and only if $x,y\in\boldsymbol{u}$ and $x,y$ belong to the same
component of $L$ (compare Subsection \ref{subsec:unoriented-local-system}).
Assume that the equivalence classes are totally ordered. A Heegaard
datum \emph{represents this data} (or represents $(Y,\boldsymbol{v}),(L,\boldsymbol{u},\alpha_{L})$,
for simplicity, when the total order is understood) if:
\begin{itemize}
\item the underlying Heegaard diagram $(\Sigma,\boldsymbol{\alpha},\boldsymbol{\beta},\boldsymbol{u}\sqcup\boldsymbol{v})$
represents $(Y,\boldsymbol{v}),(L,\boldsymbol{u},\alpha_{L})$, and
\item the \emph{coefficient ring} is $R=\mathbb{F}\llbracket U_{1}^{k_{1}},\cdots,U_{n}^{k_{n}}\rrbracket$,
where the $k_{\ell}$'s and the weight function are as follows: if
the $\ell$th equivalence class has size two, then $k_{\ell}=1/2$
and assign the weight $U_{\ell}^{1/2}$ to both basepoints in the
equivalence class. If the $\ell$th equivalence class has size one,
then $k_{\ell}=1$ and assign the weight $U_{\ell}$ to the basepoint
in the equivalence class.
\end{itemize}
\emph{The unoriented link Floer chain complex of $(Y,\boldsymbol{v}),(L,\boldsymbol{u},\alpha_{L})$},
denoted
\[
\boldsymbol{CFL'}^{-}((Y,\boldsymbol{v}),(L,\boldsymbol{u},\alpha_{L})),
\]
is defined as the chain complex of a weakly admissible Heegaard datum
that represents this data.
\end{defn}

\subsection{Bands and balled links}

In this subsection, we introduce the notion of a balled link and define
its unoriented link Floer homology. We first explain the motivation
behind considering balled links. Recall (from Definition \ref{def:unoriented-intro})
that we wanted each link component to have exactly two link basepoints.
Hence, the merge and split band maps involve links with different
number of basepoints. As in the discussion of Subsection \ref{subsec:The-proof},
to go around this issue, we free-stabilize our ambient manifold when
we consider the link with fewer basepoints. 

Figure \ref{fig:split-band} is a diagram of a kind of split band
that we consider. The link $L_{a}$ has two link basepoints $z$ and
$w$, but there is also a free basepoint $v$ in the underlying three-manifold.
Band surgery along the specified band creates a new link component
(the left hand side part of $L_{b}$). We added the free basepoint
$v$ so that after we do this band surgery, we can make the basepoint
$v$ into two link basepoints $v,v'$ of the left hand side component
(compare the Heegaard diagram). However, to do this, we have to isotope
the left hand side component so that it goes through $v$ (and $v'$).
Hence, the link $L_{a}$, the basepoints $v,w,z$, and the band do
not uniquely determine $L_{b}$, let alone the isotopy.

We introduce \emph{baseballs} and consider \emph{balled links} to
specify $L_{b}$ and the isotopy up to choices that do not affect
the unoriented link Floer homology and maps between them: we specify
a baseball $B_{v}\ni v$ that intersects $L_{a}$ in an interval,
and isotope the link inside the baseball $B_{v}$. (We will also let
$v'\in B_{v}$.)

Considering baseballs and balled links also has the advantage that
the following can be encoded in a nice way: in this paper, we only
consider bands that are supported in the $\beta$-handlebody parts
($\beta$-bands in the sense of \cite[Section 6.1]{MR3905679}). This
can be thought of as that the two link basepoints in each link component
are ``very close to each other'' with respect to the bands (the
$\alpha$-handlebody part is ``short''), and that we treat the two
basepoints as one basepoint. This is encoded as that each link component
has one \emph{link baseball} (one should think of baseballs as being
very small and very round), which contains these two basepoints. The
link baseballs mark the $\alpha$-handlebody part of the link. 
\begin{figure}[h]
\begin{centering}
\includegraphics[scale=2]{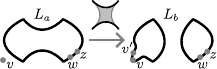}\includegraphics[scale=0.75]{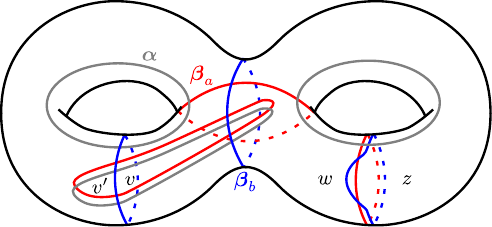}
\par\end{centering}
\caption{\label{fig:split-band}A split band and a Heegaard diagram for it}
\end{figure}

\begin{defn}
A \emph{balled link $L$ in a three-manifold $Y$} consists of:
\begin{itemize}
\item a (potentially empty) link $L^{link}$ in $Y$, referred to as the
\emph{underlying link} (also denoted $L$),
\item a finite, nonempty, ordered sequence of pairwise disjoint, embedded
(closed) three-balls $B_{1},\cdots,B_{n}$ in $Y$ ($n\ge1$), called
\emph{baseballs}, such that each link component of $L^{link}$ intersects
some baseball, and if a baseball $B_{b}$ intersects $L^{link}$,
then there exists an open neighborhood $U\supset B_{b}$ such that
$(U,L^{link}\cap U,B_{b})$ is diffeomorphic to $(\mathbb{R}^{3},\mathbb{R}\times\{(0,0)\},\{{\bf x}:\left|{\bf x}\right|\le1\})$,
\item one of two \emph{types} (\emph{link} or \emph{free}) associated to
each baseball: for each link component of $L^{link}$, one baseball
intersecting the component is said to be a \emph{link baseball}, and
all baseballs not said to be link baseballs are said to be \emph{free
baseballs},
\item a \emph{marking} in $\{1/2,1,\infty\}$ associated to each link baseball
and a \emph{marking} in $\{1,\infty\}$ associated to each free baseball.
\end{itemize}
\end{defn}

To define the unoriented link Floer chain complex of a balled link
$L\subset Y$, we choose an \emph{auxiliary datum} for $L\subset Y$.
\begin{defn}
\label{def:auxiliary-datum}Let $L\subset Y$ be a balled link with
baseballs $B_{1},\cdots,B_{n}$. An \emph{auxiliary datum} for $L\subset Y$
is a link $\widetilde{L}$ together with a choice of \emph{link basepoints}
$\boldsymbol{u}$ and \emph{free basepoints} $\boldsymbol{v}$ as
follows:
\begin{itemize}
\item $\widetilde{L}$ agrees with $L^{link}$ outside the union of the
baseballs.
\item The link $\widetilde{L}$ and the baseballs $B_{1},\cdots,B_{n}$
form a balled link.
\item For each link baseball $B_{b}$, choose two distinct points (\emph{``link
basepoints''}) in $({\rm int}B_{b})\cap\widetilde{L}$. Let $\boldsymbol{u}$
be the set of these link basepoints.
\item For each free baseball $B_{b}$, choose a point (\emph{``free basepoint''})
in $({\rm int}B_{b})\backslash\widetilde{L}$. Let $\boldsymbol{v}$
be the set of these free basepoints.
\end{itemize}
If $(\widetilde{L},\boldsymbol{u},\boldsymbol{v})$ is an auxiliary
datum for $L\subset Y$, then each link component of $\widetilde{L}$
has exactly two link basepoints (i.e. $(\widetilde{L},\boldsymbol{u})$
is minimally pointed), and they divide the component into two components,
exactly one of which is entirely contained in the interior of a link
baseball. Define the \emph{suture datum} $\alpha_{\widetilde{L}}$
on $(\widetilde{L},\boldsymbol{u})$ by choosing the components of
$\widetilde{L}\backslash\boldsymbol{u}$ that are contained in the
link baseballs.

Define equivalence classes on $\boldsymbol{u}\sqcup\boldsymbol{v}$
as in Definition \ref{def:unoriented-suture-datum}, with the order
on the equivalence classes given by the order of the baseballs. We
say that a Heegaard datum \emph{represents the auxiliary datum $(\widetilde{L},\boldsymbol{u},\boldsymbol{v})$}
if it represents $(Y,\boldsymbol{v}),(\widetilde{L},\boldsymbol{u},\alpha_{\widetilde{L}})$.
\end{defn}

We define the \emph{unoriented link Floer chain complex of a balled
link}, and hence the \emph{unoriented link Floer homology of a balled
link}. We will show in Proposition \ref{prop:naturality} that the
chain complex is well-defined up to chain homotopy equivalence, and
also show \emph{naturality}, i.e. that we have a preferred choice
of the homotopy equivalence, up to homotopy.
\begin{defn}
\label{def:unoriented-link-Floer-balled-link}Let $L\subset Y$ be
a balled link, and let $(\widetilde{L},\boldsymbol{u},\boldsymbol{v})$
be an auxiliary datum for $L\subset Y$. The \emph{unoriented link
Floer chain complex} of $L\subset Y$ is 
\[
\boldsymbol{CFL'}^{-}(Y,L):=\boldsymbol{CFL'}^{-}((Y,\boldsymbol{v}),(\widetilde{L},\boldsymbol{u},\alpha_{\widetilde{L}})).
\]
The \emph{coefficient ring} of $Y,L$ is the coefficient ring of $(Y,\boldsymbol{v}),(\widetilde{L},\boldsymbol{u},\alpha_{\widetilde{L}})$
from Definition \ref{def:unoriented-suture-datum}.

Let $R=\mathbb{F}\llbracket U_{1}^{k_{1}},\cdots,U_{n}^{k_{n}}\rrbracket$
be the coefficient ring of $Y,L$. Define the \emph{infinity version}
as 
\[
\boldsymbol{CFL'}^{\infty}(Y,L):=\boldsymbol{CFL'}^{-}(Y,L)\otimes_{R}R^{\infty}=\boldsymbol{CFL'}^{-}(Y,L)\otimes_{R}R[U_{1}^{-1},\cdots,U_{n}^{-1}].
\]
If at least one of the baseballs is marked with $1/2$ or $1$, define
the \emph{hat version} as 
\[
\widehat{CFL'}(Y,L):=\boldsymbol{CFL'}^{-}(Y,L)/(U_{1}^{m_{1}},\cdots,U_{n}^{m_{n}}),
\]
where the $i$th baseball is marked with $m_{i}\in\{1/2,1,\infty\}$,
and $U_{i}^{\infty}:=0$.
\end{defn}

\begin{rem}
In Subsection \ref{subsec:Unoriented-link-Floer}, we defined the
unreduced and reduced hat versions of unoriented link Floer homology
separately. The unreduced hat version (resp. reduced hat version)
corresponds to the hat version of Definition \ref{def:unoriented-link-Floer-balled-link},
where exactly one of the baseballs is marked with $1$ (resp. $1/2$),
and the rest are marked with $\infty$.

Let $L\subset Y$ be a balled link. Then, $\widehat{HFL}$ of the
underlying link is the hat version $\widehat{HFL'}(Y,L)$ of Definition
\ref{def:unoriented-link-Floer-balled-link}, if every link baseball
is marked with $1/2$ and every free baseball is marked with $\infty$.
\end{rem}

As usual, the infinity version is simple. 
\begin{prop}
\label{prop:infinity}Let $L$ be a balled link in $S^{3}$, and let
its coefficient ring be $R=\mathbb{F}\llbracket U_{1}^{k_{1}},\cdots,U_{n}^{k_{n}}\rrbracket$.
Then $\boldsymbol{HFL'}^{\infty}(S^{3},L)$ does not depend on the
underlying link. In other words,
\[
\boldsymbol{HFL'}^{\infty}(S^{3},L)\simeq\mathbb{F}\llbracket U_{1}^{k_{1}},\cdots,U_{n}^{k_{n}},U\rrbracket[U_{1}^{-1},\cdots,U_{n}^{-1},U^{-1}]/(U_{1}-U,\cdots,U_{n}-U).
\]
\end{prop}

\begin{proof}
This argument is standard. Consider a weakly admissible Heegaard diagram
$(\Sigma,\boldsymbol{\alpha},\boldsymbol{\beta},\boldsymbol{u}\sqcup\boldsymbol{v})$
for $(S^{3},L)$. Choose a $w|z$ partition $\boldsymbol{u}=\boldsymbol{w}\sqcup\boldsymbol{z}$
of the minimally pointed link $(L,\boldsymbol{u})$, and also assume
that $(\Sigma,\boldsymbol{\alpha},\boldsymbol{\beta},\boldsymbol{v}\sqcup\boldsymbol{z})$
is also weakly admissible. Let $\boldsymbol{CF}'{}^{\infty}$ be the
usual, unoriented link Floer chain complex, and let the corresponding
weight function be $w'$. Let $\boldsymbol{CF}^{\infty}$ be the chain
complex whose underlying group is the same as that of $\boldsymbol{CF}'{}^{\infty}$,
but if the $\ell$th baseball $B_{\ell}$ is a link baseball, then
assign the weight $U_{\ell}$ to the $z$-basepoint corresponding
to $B_{\ell}$ and assign the weight $1$ to the $w$-basepoint corresponding
to $B_{\ell}$ . If the $\ell$th baseball $B_{\ell}$ is a free baseball,
then assign the weight $U_{\ell}$ to the free basepoint corresponding
to $B_{\ell}$. Call this new weight function $w$.

We claim that $\boldsymbol{CF}'{}^{\infty}$ and $\boldsymbol{CF}^{\infty}$
are isomorphic as chain complexes over $R^{\infty}$. Fix an intersection
point ${\bf x}\in\boldsymbol{\alpha}\cap\boldsymbol{\beta}$, and
for each ${\bf y}\in\boldsymbol{\alpha}\cap\boldsymbol{\beta}$, choose
a domain ${\cal D}_{{\bf y}}\in D({\bf x},{\bf y})$. Then the $R^{\infty}$-linear
map $\boldsymbol{CF}'{}^{\infty}\to\boldsymbol{CF}^{\infty}$ given
by 
\[
{\bf y}\mapsto\frac{w({\cal D}_{{\bf y}})}{w'({\cal D}_{{\bf y}})}{\bf y}
\]
is a chain map, since $w'({\cal P})=w({\cal P})$ for any cornerless
two-chain ${\cal P}$. It is an isomorphism since it has an obvious
inverse.
\end{proof}
Proposition \ref{prop:naturality} shows that $\boldsymbol{HFL'}^{-}(Y,L)$,
$\boldsymbol{HFL'}^{\infty}(Y,L)$, and $\widehat{HFL'}(Y,L)$ (in
fact, the chain complexes, up to chain homotopy equivalence) do not
depend on the additional choices we made. Compare \cite{MR4337438},
\cite[Proposition 2.3]{MR3649355}, \cite[Proposition 3.5]{MR3905679}.
\begin{lem}
\label{lem:Let--and}Let $D^{3}=\{{\bf x}\in\mathbb{R}^{3}:|{\bf x}|\le1\}$
and $T=\{(x,0,0):|x|\le1\}\subset D^{3}$. Then, any diffeomorphism
of $D^{3}$ that fixes $\partial D^{3}\cup T$ pointwise is isotopic
rel $\partial D^{3}\cup T$ (i.e. the isotopy fixes $\partial D^{3}\cup T$
pointwise) to the identity.
\end{lem}

\begin{proof}
We use Cerf's theorem \cite{ctx41632241560006421}, i.e. that any
diffeomorphism of $D^{3}$ that fixes $\partial D^{3}$ pointwise
is isotopic to the identity rel $\partial D^{3}$. After an initial
isotopy rel $\partial D^{3}\cup T$, we can assume that the diffeomorphism
fixes a neighborhood of $\partial D^{3}\cup T$. Identify the complement
of a small neighborhood of $\partial D^{3}\cup T$ in $D^{3}$ with
$S^{1}\times D^{2}$. Now, isotope (rel boundary) the embedding of
standard disk $\{\ast\}\times D^{2}$ composed with the diffeomorphism,
to the standard disk, and then apply Cerf's theorem.
\end{proof}
\begin{prop}[Naturality]
\label{prop:naturality}Let $L$ be a balled link in a three-manifold
$Y$, and let $R$ be its coefficient ring. Recall that to define
$\boldsymbol{CFL'}^{-}(Y,L)$, we chose a Heegaard datum ${\cal H}$
that represents some auxiliary datum for $L\subset Y$. Given two
such Heegaard data ${\cal H},{\cal H}'$, there is an $R$-linear
chain homotopy equivalence 
\[
\boldsymbol{CF}^{-}({\cal H})\to\boldsymbol{CF}^{-}({\cal H}'),
\]
which is well-defined and functorial up to $R$-linear chain homotopy.
\end{prop}

\begin{proof}
The statement holds if we fix an auxiliary datum and only consider
Heegaard data that represent that auxiliary datum, by \cite{MR4337438}.
For two auxiliary data $(\widetilde{L},\boldsymbol{u},\boldsymbol{v})$
and $(\widetilde{L}',\boldsymbol{u}',\boldsymbol{v}')$, the tuples
$(Y,\widetilde{L},\boldsymbol{u},\boldsymbol{v})$ and $(Y,\widetilde{L}',\boldsymbol{u}',\boldsymbol{v}')$
are diffeomorphic via a diffeomorphism supported in the baseballs.
Define the chain homotopy equivalence between the chain complexes
for the two auxiliary data as the map induced by such a diffeomorphism.
We claim that this chain homotopy equivalence is well-defined up to
chain homotopy: we show that for an auxiliary datum $(\widetilde{L},\boldsymbol{u},\boldsymbol{v})$,
any diffeomorphism of $(Y,\widetilde{L},\boldsymbol{u},\boldsymbol{v})$
that is supported in the baseballs induces a chain map homotopic to
the identity. It is sufficient to show this for the case where the
diffeomorphism is supported in one baseball.

First, if the diffeomorphism is supported in a link baseball, then
the claim follows since it is isotopic to the identity by Lemma \ref{lem:Let--and}.
Also, by Lemma \ref{lem:Let--and}, for $v\in D^{3}\backslash T$,
any diffeomorphism of $D^{3}$ that fixes $\partial D^{3}\cup T\cup\{v\}$
pointwise is isotopic rel $\partial D^{3}\cup T\cup\{v\}$ to such
a diffeomorphism that is given by isotoping $v$ inside $D^{3}\backslash T$
along some loop (``the basepoint moving map''). Hence, if the diffeomorphism
is supported in a free baseball, then the claim follows since the
basepoint moving map corresponding to a meridional loop is homotopic
to the identity by \cite[Proposition 14.20]{1512.01184}. (We thank
Ian Zemke for pointing out that this is a simpler case, since it can
be realized as a free-stabilization region going around the corresponding
link basepoint.)
\end{proof}
\begin{rem}
The space ${\rm Spin}^{c}(Y(\widetilde{L^{sut}}))$ of ${\rm Spin}^{c}$-structures
for all the different choices of an auxiliary datum $(\widetilde{L},\boldsymbol{u},\boldsymbol{v})$
are canonically identified as well, where $\widetilde{L^{sut}}:=(\widetilde{L},\boldsymbol{u},\alpha_{\widetilde{L}})$.
Hence, we can define ${\rm Spin}^{c}(Y(L))$ as ${\rm Spin}^{c}(Y(\widetilde{L^{sut}}))$
and define $\boldsymbol{CFL'}^{-}(Y,L;\mathfrak{s})$ for $\mathfrak{s}\in{\rm Spin}^{c}(Y(L))$.
These are also well-defined (natural) in the above sense.

If we only want to consider finitely many holomorphic disks, then
we can work with $\mathfrak{s}$-strongly admissible Heegaard diagrams
(Definition \ref{def:strongly-admissible}). Hence, we can also work
over polynomial rings and define $CFL'{}^{-}(Y,L;\mathfrak{s})$. 
\end{rem}

\begin{defn}
\label{def:A-band-on}A \emph{band on a balled link $L\subset Y$}
is a band $B$ on the underlying link $L^{link}$ that does not intersect
any of the baseballs. We say \emph{$B$ is a band from the balled
link $L$ to the balled link $L'$} if the underlying link $L'{}^{link}$
is obtained by surgering $L^{link}$ along $B$, the orders and the
markings of the baseballs of $L$ and $L'$ are the same, and the
types of the baseballs of $L$ and $L'$ are the same, except in the
following cases:
\begin{itemize}
\item If $B$ is a merge band, then exactly one link baseball of $L$ becomes
a free baseball of $L'$.
\item If $B$ is a split band, then exactly one free baseball of $L$ becomes
a link baseball of $L'$.
\end{itemize}
\end{defn}

Note that if the band is a merge or a split band, then the type of
at least one baseball has to change in order for $L'$ to be a balled
link. Also note that the baseball whose type changes must be marked
with $1$ or $\infty$.

We handle each band type separately and define band maps using Heegaard
triple diagrams. We closely follow \cite[Section 6.1]{MR3905679};
note that our bands are always $\beta$-bands. Also compare \cite[Section 4.1]{MR2222356}.

\subsection{\label{subsec:Non-orientable-bands}Non-orientable bands}

\begin{figure}[h]
\begin{centering}
\includegraphics{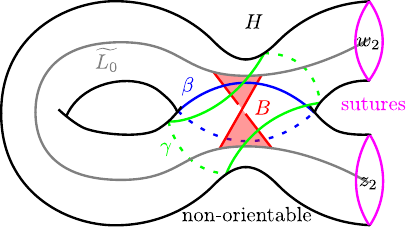}
\par\end{centering}
\caption{\label{fig:sutured-band-1}The region $H$, the tangle $\widetilde{L_{0}}$,
and the $\beta$ and $\gamma$ circles for Definition \ref{def:triple-non-orientable-band}
and \ref{def:triple-merge-band}.}
\end{figure}

Let $B$ be a non-orientable band from a balled link $L$ to $L'$
in a three-manifold $Y$. Then $L$ and $L'$ have the same coefficient
ring; denote it as $R$. We will define the band map for the minus
version as an $R$-linear map
\[
\boldsymbol{CFL'}^{-}(Y,L)\to\boldsymbol{CFL'}^{-}(Y,L')
\]
which is well-defined up to $R$-linear chain homotopy. Given this,
define the band map for the hat version by quotienting by $(U_{1}^{m_{1}},\cdots,U_{n}^{m_{n}})$
as in Definition \ref{def:unoriented-link-Floer-balled-link}. Recall
from Definition \ref{def:A-band-on} that the markings of the baseballs
of $L$ and $L'$ are the same; hence we indeed quotient by the same
variables for $L$ and $L'$. Define the band map for the infinity
version by localizing as in Definition \ref{def:unoriented-link-Floer-balled-link}.
The well-definedness of band maps for the hat and infinity versions
follows from that of the minus version. Therefore, we focus on the
minus version.

Let $(\widetilde{L},\boldsymbol{u},\boldsymbol{v})$ be an auxiliary
datum for $L$. Note that if $\widetilde{L'}$ is the link obtained
by surgering $\widetilde{L}$ along $B$, then $(\widetilde{L'},\boldsymbol{u},\boldsymbol{v})$
is an auxiliary datum for $L'$. Define the sutured manifold obtained
from $Y\backslash N(\widetilde{L}\cup B\cup\boldsymbol{v})$ as follows.
On $\partial N(\widetilde{L}\cup B)$, we add ``meridional'' sutures
for each basepoint of $\boldsymbol{u}$, such that $R_{-}$\footnote{Our convention is that $R_{-}$ corresponds to the $\alpha$-handlebody
part.} corresponds to the parts of $\partial N(\widetilde{L}\cup B)$ that
live over $\alpha_{\widetilde{L}}$. Also, add one suture to each
connected component of $\partial N(\boldsymbol{v})$.

Let $\widetilde{L_{0}}$ be the connected component of $\widetilde{L}\backslash\boldsymbol{u}$
that intersects the band $B$, let $H\subset\overline{N}(\widetilde{L}\cup B)$
be the points that live over $\widetilde{L_{0}}\cup B$, and let $S=\partial H\cap\partial N(\widetilde{L}\cup B)=\partial H\cap R_{+}$.
In particular, $H$ is a solid torus, and $S$ is a torus minus two
disks. See Figure \ref{fig:sutured-band-1} for a schematic.
\begin{defn}
\label{def:triple-non-orientable-band}Let $(\widetilde{L},\boldsymbol{u},\alpha_{\widetilde{L}})$
be a minimally pointed link equipped with a suture datum, inside a
pointed three-manifold $(Y,\boldsymbol{v})$. We say that the Heegaard
triple 
\[
(\Sigma,\boldsymbol{\alpha}=\{\alpha^{1},\cdots,\alpha^{n}\},\boldsymbol{\beta}=\{\beta^{1},\cdots,\beta^{n}\},\boldsymbol{\gamma}=\{\gamma^{1},\cdots,\gamma^{n}\},\boldsymbol{u}\sqcup\boldsymbol{v})
\]
is \emph{subordinate to $B$}, if:
\begin{enumerate}
\item The diagram $(\Sigma,\{\alpha^{1},\cdots,\alpha^{n}\},\{\beta^{2},\cdots,\beta^{n}\},\boldsymbol{u}\sqcup\boldsymbol{v})$
corresponds to the sutured manifold $Y\backslash N(\widetilde{L}\cup B\cup\boldsymbol{v})$
defined above.
\item The circles $\gamma^{2},\cdots,\gamma^{n}$ are standard translates
of $\beta^{2},\cdots,\beta^{n}$.
\item Let $\widetilde{L_{0}},H,S$ be as above. Let $\beta\subset S$ be
the closed curve that is defined up to isotopy (one can show this,
for instance, by an innermost disk/arc argument) by the property that
$\beta$ bounds a disk $D_{\beta}$ in $H$, such that $D_{\beta}$
does not intersect $\widetilde{L}$ and is non-separating in $H$.
Similarly, let $\gamma\subset S$ be the closed curve that bounds
a non-separating disk $D_{\gamma}$ in $H$, that does not intersect
$\widetilde{L'}$.

Then, $\beta^{1}$ is a closed curve that is obtained by projecting
$\beta$ onto $\Sigma\backslash(\beta^{2}\cup\cdots\cup\beta^{n})$
and $\gamma^{1}$ is a curve that is obtained by projecting $\gamma$
onto $\Sigma\backslash(\gamma^{2}\cup\cdots\cup\gamma^{n})$.

\end{enumerate}
A Heegaard datum is \emph{subordinate to $B$} if the corresponding
Heegaard diagram is subordinate to $B$ and the sub Heegaard datum
with attaching curves $\boldsymbol{\alpha},\boldsymbol{\beta}$ represents
$(Y,\boldsymbol{v}),(\widetilde{L},\boldsymbol{u},\alpha_{\widetilde{L}})$.
\end{defn}

\begin{figure}[h]
\begin{centering}
\includegraphics{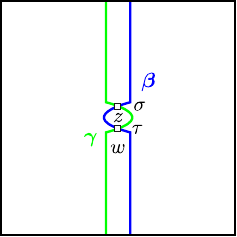} 
\par\end{centering}
\caption{\label{fig:local-band-1}A local Heegaard diagram for non-orientable
bands. Work over $\mathbb{F}\llbracket U^{1/2}\rrbracket$ and assign
the weight $U^{1/2}$ to both $w$ and $z$.}
\end{figure}

Given a Heegaard datum with underlying Heegaard diagram $(\Sigma,\boldsymbol{\alpha},\boldsymbol{\beta},\boldsymbol{\gamma},\boldsymbol{u}\sqcup\boldsymbol{v})$
subordinate to a non-orientable band $B$ from $L$ to $L'$, we would
like to define a \emph{``canonical element''} $\Theta_{B}\in\boldsymbol{HF}^{-}(\boldsymbol{\beta},\boldsymbol{\gamma})$.
First, we can show, as in \cite[Lemma 6.4]{MR3905679}, that the sub
Heegaard datum with attaching curves $\boldsymbol{\beta},\boldsymbol{\gamma}$
can be obtained from a stabilization of Figure \ref{fig:local-band-1}
by handleslides and isotopies.

Let us compute the homology group of the Heegaard datum of Figure
\ref{fig:local-band-1} (recall Example \ref{exa:simple-obstruction}).
The Heegaard diagram is homologically $\mathbb{Z}$-gradable and Alexander
$\mathbb{Z}/2$-splittable. The homology group is
\[
\boldsymbol{HF}^{-}(\boldsymbol{\beta},\boldsymbol{\gamma})\simeq\sigma\mathbb{F}\llbracket U^{1/2}\rrbracket\oplus\tau\mathbb{F}\llbracket U^{1/2}\rrbracket.
\]
Both $\sigma$ and $\tau$ lie in torsion ${\rm Spin}^{c}$-structures
and are homogeneous with respect to the homological $\mathbb{Z}$-grading,
the ${\rm Spin}^{c}$-splitting, and the Alexander $\mathbb{Z}/2$-grading.
They lie in the same homological $\mathbb{Z}$-grading, but have different
${\rm Spin}^{c}$-structures and different Alexander $\mathbb{Z}/2$-gradings.

Hence, given a Heegaard datum with underlying Heegaard diagram $(\Sigma,\boldsymbol{\alpha},\boldsymbol{\beta},\boldsymbol{\gamma},\boldsymbol{u}\sqcup\boldsymbol{v})$
subordinate to a non-orientable band $B$ from $L$ to $L'$, we can
compute $\boldsymbol{HF}^{-}(\boldsymbol{\beta},\boldsymbol{\gamma})$:
the top homological $\mathbb{Z}$-grading part of 
\[
\boldsymbol{HF}^{-}(\boldsymbol{\beta},\boldsymbol{\gamma};0):=\bigoplus_{c_{1}(\mathfrak{s})=0}\boldsymbol{HF}^{-}(\boldsymbol{\beta},\boldsymbol{\gamma};\mathfrak{s})
\]
has rank $2$, say $f\mathbb{F}\oplus g\mathbb{F}$, where $f,g$
are homogeneous with respect to the ${\rm Spin}^{c}$-splitting. These
correspond to $S(\sigma),S(\tau)$ if the Heegaard datum is a stabilization
of Figure \ref{fig:local-band-1}.
\begin{rem}
\label{rem:The-elements-}The elements $S(\sigma),S(\tau)$ are indeed
cycles in any almost complex structure since they are cycles in a
pinched almost complex structure by Proposition \ref{prop:stabilization},
and they are invariant under changing the almost complex structure
since there are no relevant Maslov index $0$ domains.
\end{rem}

We can distinguish $f$ and $g$ using the following. Let ${\cal R}=\mathbb{F}\llbracket\{P_{p}\}_{p\in\boldsymbol{u}\sqcup\boldsymbol{v}}\rrbracket$
be the power series ring with one indeterminate $P_{p}$ for each
basepoint $p\in\boldsymbol{u}\sqcup\boldsymbol{v}$, and consider
the ${\cal R}$-chain complex ${\cal CF}^{-}(\boldsymbol{\beta},\boldsymbol{\gamma})$
given by working over the coefficient ring ${\cal R}$ and assigning
the weight $P_{p}$ to $p$. View $R$ as an ${\cal R}$-module ($R$
is a quotient of ${\cal R}$) by identifying the weight of $p$ in
$R$ with $P_{p}$. Then, we have a map 
\[
{\cal HF}^{-}(\boldsymbol{\beta},\boldsymbol{\gamma})\otimes_{{\cal R}}R\to\boldsymbol{HF}^{-}(\boldsymbol{\beta},\boldsymbol{\gamma}).
\]
The element $g$ is in the image of this map, but $f$ is not. We
let $\Theta_{B}:=g$, i.e. $S(\tau)$.

To summarize, we define the \emph{canonical element} as follows.
\begin{defn}
\label{def:characterize-thetaB-1}Let a Heegaard datum with underlying
Heegaard diagram $(\Sigma,\boldsymbol{\beta},\boldsymbol{\gamma},\boldsymbol{u}\sqcup\boldsymbol{v})$
be obtained from a stabilization of Figure \ref{fig:local-band-1}
by handleslides and isotopies. Define the \emph{canonical element}
$\Theta_{B}\in\boldsymbol{HF}^{-}(\boldsymbol{\beta},\boldsymbol{\gamma})$
as the unique nonzero element that satisfies the following properties:
\end{defn}

\begin{itemize}
\item It is homogeneous with respect to the ${\rm Spin}^{c}$-splitting.
\item It lies in the top homological $\mathbb{Z}$-grading part of $\boldsymbol{HF}^{-}(\boldsymbol{\beta},\boldsymbol{\gamma};0)$.
\item It is in the image of the map ${\cal HF}^{-}(\boldsymbol{\beta},\boldsymbol{\gamma})\to\boldsymbol{HF}^{-}(\boldsymbol{\beta},\boldsymbol{\gamma})$.
\end{itemize}
\begin{defn}
\label{def:band-maps-1}Let $B$ be a non-orientable band from a balled
link $L$ to $L'$, let $(\widetilde{L},\boldsymbol{u},\boldsymbol{v})$
be an auxiliary datum for $L$, and let $\widetilde{L'}$ be the link
obtained by surgering $\widetilde{L}$ along $B$. Choose a weakly
admissible Heegaard datum subordinate to $B$ (for $(\widetilde{L},\boldsymbol{u},\alpha_{\widetilde{L}})$),
with attaching curves $\boldsymbol{\alpha},\boldsymbol{\beta},\boldsymbol{\gamma}$.
Then, the \emph{band map} is the triangle counting map
\[
\mu_{2}(-,\Theta_{B}):\boldsymbol{CF}^{-}(\boldsymbol{\alpha},\boldsymbol{\beta})\to\boldsymbol{CF}^{-}(\boldsymbol{\alpha},\boldsymbol{\gamma}),
\]
where $\Theta_{B}\in\boldsymbol{CF}^{-}(\boldsymbol{\beta},\boldsymbol{\gamma})$
is any cycle that represents the canonical element $\Theta_{B}\in\boldsymbol{HF}^{-}(\boldsymbol{\beta},\boldsymbol{\gamma})$.
\end{defn}

\begin{prop}
\label{prop:band-map-well-defined-1}Band maps for non-orientable
bands are well-defined up to $R$-linear chain homotopy.
\end{prop}

\begin{proof}
Let $B$ be a band from a balled link $L$ to $L'$. Let us first
fix an arbitrary choice of auxiliary datum for $L$ and show that
the band map does not depend, up to homotopy, on the Heegaard triple
subordinate to $B$ (for this auxiliary datum). The proof is basically
the same as \cite[Proposition 4.6]{MR2222356} and \cite[Lemma 6.5]{MR3905679}.
One first shows that two Heegaard triples subordinate to the band
can be related by some moves as in \cite[Lemma 4.5]{MR2222356} and
\cite[Lemma 6.3]{MR3905679}. The maps are invariant under (de)stabilization
(compare Proposition \ref{prop:stabilization-1}), various handleslides,
and isotopies since the conditions in Definition \ref{def:characterize-thetaB-1}
are invariant under these moves. 

Now, to show the proposition, we recall the following. Let ${\cal H}_{i}$
for $i=1,2$ be a Heegaard datum (with two attaching curves) that
represents $(Y_{i},\boldsymbol{v}_{i}),(L_{i},\boldsymbol{u}_{i},\alpha_{L_{i}})$,
where $(L_{i},\boldsymbol{u}_{i},\alpha_{L_{i}})$ is a minimally
pointed link equipped with a suture datum inside the pointed three-manifold
$(Y_{i},\boldsymbol{v}_{i})$. Assume that ${\cal H}_{2}$ is obtained
from ${\cal H}_{1}$ by moving the basepoints inside each elementary
domain. Moving the basepoints induces a diffeomorphism between $((Y_{1},\boldsymbol{v}_{1}),(L_{1},\boldsymbol{u}_{1},\alpha_{L_{1}}))$
and $((Y_{2},\boldsymbol{v}_{2}),(L_{2},\boldsymbol{u}_{2},\alpha_{L_{2}}))$,
and hence induces a chain map between the corresponding unoriented
link Floer chain complexes. Since the intersection points of ${\cal H}_{1}$
and ${\cal H}_{2}$ are the same, and the differentials on the chain
complexes count the same bigons, the ``identity map'' $\boldsymbol{CF}^{-}({\cal H}_{1})\to\boldsymbol{CF}^{-}({\cal H}_{2})$
is a chain map. \cite[Proposition 9.27]{MR4337438} implies that these
two chain maps are homotopic.

Using this, we can show that the band maps do not depend on the auxiliary
datum. For any two choices of auxiliary data for $L$, there exists
a finite sequence of auxiliary data for $L$ such that for each consecutive
pair, there exists a Heegaard triple that is subordinate to $B$ for
both auxiliary data modulo moving the basepoints inside each elementary
two-chain of the Heegaard triple diagram.
\end{proof}

\subsection{Merge and split bands}

\begin{figure}[h]
\begin{centering}
\includegraphics{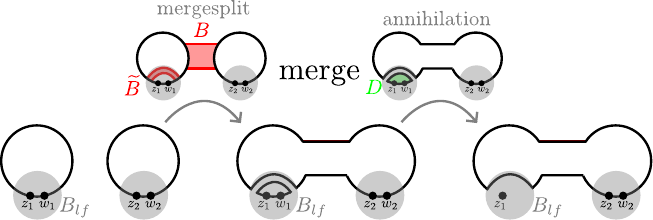}
\par\end{centering}
\begin{centering}
\includegraphics{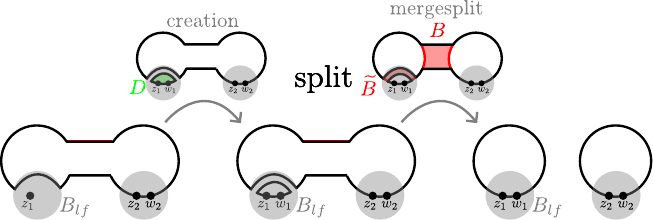}
\par\end{centering}
\caption{\label{fig:merge-composition}We view merge and split maps as the
composition of two maps}
\end{figure}

Merge and split bands are more complicated to define. We will conceptually
think of them as the composition of a map that is induced by two bands
(a \emph{``mergesplit pair''}), and a birth/death type map (we call
these \emph{creation} and \emph{annihilation}, since they are different
from the birth and death cobordism maps \cite{MR3905679}), as in
Figure \ref{fig:merge-composition}. We first define these maps separately
in Subsubsections \ref{subsec:Birth-and-death} and \ref{subsec:mergesplit}.

Similarly to Subsection \ref{subsec:Non-orientable-bands}, we will
focus on the minus version: let $B$ be a merge band (resp. split
band) from a balled link $L$ to $L'$ in a three-manifold $Y$, and
let $S$ be the coefficient ring of $L'$ (resp. $L$). Then, we will
define the band map for $B$ for the minus version as an $S$-linear
map 
\[
\boldsymbol{CFL'}^{-}(Y,L)\to\boldsymbol{CFL'}^{-}(Y,L')
\]
which is well-defined up to $S$-linear chain homotopy. Given this,
define the band map for the hat version by quotienting by $(U_{1}^{m_{1}},\cdots,U_{n}^{m_{n}})$
as in Definition \ref{def:unoriented-link-Floer-balled-link}, and
define the band map for the infinity version by localizing. The well-definedness
of band maps for the hat and infinity versions follows from that of
the minus version.

\subsubsection{\label{subsec:Birth-and-death}Creation and annihilation maps}

Let $(L\sqcup K,\boldsymbol{u},\alpha_{L\sqcup K})$ be a minimally
pointed link equipped with a suture datum, inside a pointed three-manifold
$(Y,\boldsymbol{v})$. Assume that $K$ bounds a disk $D$ that is
disjoint from $L$ and $\boldsymbol{v}$, and let $w,z\in\boldsymbol{u}$
be the two basepoints on $K$. Consider the link $(L,\boldsymbol{u}\backslash\{w,z\},\alpha_{L})$
in $(Y,\boldsymbol{v}\sqcup\{z\})$. Order the equivalence classes
of $\boldsymbol{u}\sqcup\boldsymbol{v}$, which also orders the equivalence
classes of $(\boldsymbol{u}\sqcup\boldsymbol{v})\backslash\{w\}$.
In this subsubsection, we define creation and annihilation maps (compare
\cite[Equation (5.1) and Remark 5.2]{https://doi.org/10.1112/topo.12338})
\begin{gather*}
\boldsymbol{CFL'}^{-}((Y,\boldsymbol{v}\sqcup\{z\}),(L,\boldsymbol{u}\backslash\{w,z\},\alpha_{L}))\xrightarrow{{\rm creation}}\boldsymbol{CFL'}^{-}((Y,\boldsymbol{v}),(L\sqcup K,\boldsymbol{u},\alpha_{L\sqcup K}))\\
\boldsymbol{CFL'}^{-}((Y,\boldsymbol{v}),(L\sqcup K,\boldsymbol{u},\alpha_{L\sqcup K}))\xrightarrow{{\rm annihilation}}\boldsymbol{CFL'}^{-}((Y,\boldsymbol{v}\sqcup\{z\}),(L,\boldsymbol{u}\backslash\{w,z\},\alpha_{L})).
\end{gather*}

Let $R$ (resp. $S$) be the coefficient rings for $(L\sqcup K,\boldsymbol{u},\alpha_{L\sqcup K})$
inside $(Y,\boldsymbol{v})$ (resp. $(L,\boldsymbol{u}\backslash\{w,z\},\alpha_{L})$
inside $(Y,\boldsymbol{v}\sqcup\{z\})$). If $\{w,z\}$ is the $k$th
equivalence class of $\boldsymbol{u}\sqcup\boldsymbol{v}$, then $R=S\oplus U_{k}^{1/2}S$
as $S$-modules.

We say that a Heegaard diagram $(\Sigma,\boldsymbol{\alpha},\boldsymbol{\beta},\boldsymbol{u}\sqcup\boldsymbol{v})$
that represents $(Y,\boldsymbol{v}),(L\sqcup K,\boldsymbol{u},\alpha_{L\sqcup K})$
is \emph{subordinate to the disk $D$} if the two basepoints $w,z$
lie in the same connected component of $\Sigma\backslash(\boldsymbol{\alpha}\cup\boldsymbol{\beta})$,
and there exists a path $P$ between $w$ and $z$ inside that connected
component of $\Sigma\backslash(\boldsymbol{\alpha}\cup\boldsymbol{\beta})$,
such that $P$ represents the disk $D$.

Consider a Heegaard datum ${\cal H}_{L\sqcup K}$ that represents
$(Y,\boldsymbol{v}),(L\sqcup K,\boldsymbol{u},\alpha_{L\sqcup K})$,
whose underlying Heegaard diagram is subordinate to the disk $D$.
Consider the Heegaard datum ${\cal H}_{L}$ obtained from ${\cal H}_{L\sqcup K}$
by deleting the basepoint $w$, changing the coefficient ring to $S$,
and changing the weight of the basepoint $z$ to $U_{k}$. Then, since
$w$ and $z$ lie in the same connected component of $\Sigma\backslash(\boldsymbol{\alpha}\cup\boldsymbol{\beta})$,
\[
\boldsymbol{CF}_{R}^{-}({\cal H}_{L\sqcup K})\simeq\boldsymbol{CF}_{S}^{-}({\cal H}_{L})\oplus U_{k}^{1/2}\boldsymbol{CF}_{S}^{-}({\cal H}_{L})
\]
as chain complexes over $S$. Define the \emph{creation map} as inclusion
onto the first summand
\[
\boldsymbol{CF}_{S}^{-}({\cal H}_{L})\to\boldsymbol{CF}_{R}^{-}({\cal H}_{L\sqcup K}),
\]
and the \emph{annihilation map }as the projection onto the second
summand composed with multiplication by $U_{k}^{-1/2}$:
\[
\boldsymbol{CF}_{R}^{-}({\cal H}_{L\sqcup K})\to U_{k}^{1/2}\boldsymbol{CF}_{S}^{-}({\cal H}_{L})\xrightarrow{\cdot U_{k}^{-1/2}}\boldsymbol{CF}_{S}^{-}({\cal H}_{L}).
\]
If the disk $D$ is fixed, then these $S$-linear maps are well-defined
up to $S$-linear chain homotopy.

\subsubsection{\label{subsec:mergesplit}The map for a mergesplit pair}

\begin{figure}[h]
\begin{centering}
\includegraphics[scale=1.5]{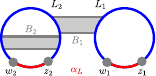}
\par\end{centering}
\caption{\label{fig:sutured-band-2}A schematic of a mergesplit pair $B_{1},B_{2}$
on $L$}
\end{figure}

\begin{defn}
Let $(L,\boldsymbol{u},\alpha_{L})$ be a minimally pointed link equipped
with a suture datum inside a pointed three-manifold $(Y,\boldsymbol{v})$.
A \emph{mergesplit pair on $L$} is a pair of disjoint bands $B_{1},B_{2}$
disjoint from $\alpha_{L}$, whose configuration is as in Figure \ref{fig:sutured-band-2}:
there exist two link components $L_{1}$ and $L_{2}$ of $L$, such
that $B_{2}$ splits $L_{2}$ into two components, and if $M$ is
the component that does not intersect $\alpha_{L}$, then $B_{1}$
merges $M$ and $L_{1}$.
\end{defn}

Let $B_{1},B_{2}$ be a mergesplit pair on $(L,\boldsymbol{u},\alpha_{L})$
in $(Y,\boldsymbol{v})$, let $L'$ be the link given by surgering
$L$ along both $B_{1}$ and $B_{2}$, and let the suture datum $\alpha_{L'}$
be the one induced by $\alpha_{L}$. Let $R$ be the coefficient ring
for both $L$ and $L'$. In this subsubsection, we define an $R$-linear
map 
\[
\boldsymbol{CFL'}^{-}((Y,\boldsymbol{v}),(L,\boldsymbol{u},\alpha_{L}))\to\boldsymbol{CFL'}^{-}((Y,\boldsymbol{v}),(L',\boldsymbol{u},\alpha_{L'}))
\]
which is well-defined up to $R$-linear chain homotopy. This subsubsection
is organized similarly to Subsection \ref{subsec:Non-orientable-bands}.

\begin{figure}[h]
\begin{centering}
\includegraphics{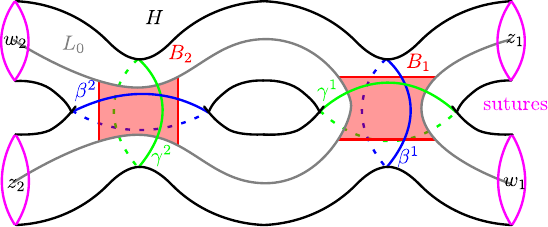}
\par\end{centering}
\caption{\label{fig:sutured-band}The region $H$, the tangle $L_{0}$, and
the $\beta$ and $\gamma$ circles for Definition \ref{def:triple-non-orientable-band}.}
\end{figure}

Define the sutured manifold obtained from $Y\backslash N(L\cup B_{1}\cup B_{2}\cup\boldsymbol{v})$
as follows. On $\partial N(L\cup B_{1}\cup B_{2})$, we add ``meridional''
sutures for each basepoint of $\boldsymbol{u}$, such that $R_{-}$
corresponds to the parts of $\partial N(L\cup B_{1}\cup B_{2})$ that
live over $\alpha_{L}$. Also, add one suture to each connected component
of $\partial N(\boldsymbol{v})$.

Let $L_{0}$ be the union of the connected components of $L\backslash\boldsymbol{u}$
that intersects the band $B$, let $H\subset\overline{N}(L\cup B_{1}\cup B_{2})$
be the points that live over $L_{0}\cup B_{1}\cup B_{2}$, and let
$S=\partial H\cap\partial N(L\cup B_{1}\cup B_{2})=\partial H\cap R_{+}$.
In particular, $H$ is a solid torus, and $S$ is a torus minus four
disks. See Figure \ref{fig:sutured-band} for a schematic.
\begin{defn}
\label{def:triple-merge-band}Let $(L,\boldsymbol{u},\alpha_{L})$
be a minimally pointed link equipped with a suture datum, inside a
pointed three-manifold $(Y,\boldsymbol{v})$, and let $B_{1}$, $B_{2}$
be a mergesplit pair on $L$. We say that the Heegaard triple 
\[
(\Sigma,\boldsymbol{\alpha}=\{\alpha^{1},\cdots,\alpha^{n}\},\boldsymbol{\beta}=\{\beta^{1},\cdots,\beta^{n}\},\boldsymbol{\gamma}=\{\gamma^{1},\cdots,\gamma^{n}\},\boldsymbol{u}\sqcup\boldsymbol{v})
\]
is \emph{subordinate to $B_{1}\cup B_{2}$} if:
\begin{enumerate}
\item The diagram $(\Sigma,\{\alpha^{1},\cdots,\alpha^{n}\},\{\beta^{3},\cdots,\beta^{n}\},\boldsymbol{u}\sqcup\boldsymbol{v})$
corresponds to the sutured manifold $Y\backslash N(L\cup B_{1}\cup B_{2}\cup\boldsymbol{v})$
defined above.
\item The circles $\gamma^{3},\cdots,\gamma^{n}$ are standard translates
of $\beta^{3},\cdots,\beta^{n}$.
\item Let $L_{0},H,S$ be as above. Define closed curves $\beta^{1},\beta^{2},\gamma^{1},\gamma^{2}\subset S$
according to Figure \ref{fig:sutured-band}, similarly to Definition
\ref{def:triple-non-orientable-band}. Let $\beta^{1},\beta^{2}\subset\Sigma\backslash(\beta^{3}\cup\cdots\cup\beta^{n})$
and $\gamma^{1},\gamma^{2}\subset\Sigma\backslash(\gamma^{3}\cup\cdots\cup\gamma^{n})$
be the closed curves obtained by projecting the corresponding curves
on $S$.
\end{enumerate}
A Heegaard datum is \emph{subordinate to $B_{1}\cup B_{2}$} if the
corresponding Heegaard diagram is subordinate to \emph{$B_{1}\cup B_{2}$}
and the sub Heegaard datum with attaching curves $\boldsymbol{\alpha},\boldsymbol{\beta}$
represents $(Y,\boldsymbol{v}),(L,\boldsymbol{u},\alpha_{L})$.
\end{defn}

\begin{figure}[h]
\begin{centering}
\includegraphics{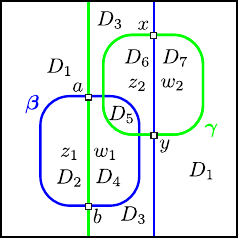}
\par\end{centering}
\caption{\label{fig:local-band}A local Heegaard diagram for a mergesplit pair.
Unless otherwise specified, work over $\mathbb{F}\llbracket U_{1}^{1/2},U_{2}^{1/2}\rrbracket$
and assign the weight $U_{i}^{1/2}$ to $w_{i}$ and $z_{i}$.}
\end{figure}

Given a Heegaard datum with underlying Heegaard diagram $(\Sigma,\boldsymbol{\alpha},\boldsymbol{\beta},\boldsymbol{\gamma},\boldsymbol{u}\sqcup\boldsymbol{v})$
subordinate to a mergesplit pair, we would like to define a \emph{``canonical
element''} $\Theta_{B}\in\boldsymbol{HF}^{-}(\boldsymbol{\beta},\boldsymbol{\gamma})$.
The sub Heegaard datum with attaching curves $\boldsymbol{\beta},\boldsymbol{\gamma}$
can be obtained from a stabilization of Figure \ref{fig:local-band}
by handleslides and isotopies.

Let us study the Heegaard diagram of Figure \ref{fig:local-band}.
\begin{lem}
\label{lem:merge-split-diagram}Consider the Heegaard diagram $(\mathbb{T}^{2},\boldsymbol{\beta},\boldsymbol{\gamma},\{w_{1},z_{1},w_{2},z_{2}\})$
given by Figure \ref{fig:local-band}. Let us work over the ring $R=\mathbb{F}[W_{1},Z_{1},W_{2},Z_{2}]$,
and assign weights $W_{1},Z_{1},W_{2},Z_{2}$ to $w_{1},z_{1},w_{2},z_{2}$,
respectively. Then $CF^{-}(\boldsymbol{\beta},\boldsymbol{\gamma})$
is one of the following, depending on the almost complex structure.% https://q.uiver.app/#q=WzAsOCxbMCwwLCJheCJdLFsxLDAsImF5Il0sWzAsMSwiYngiXSxbMSwxLCJieSJdLFsyLDAsImF4Il0sWzMsMCwiYXkiXSxbMiwxLCJieCJdLFszLDEsImJ5Il0sWzAsMSwid18xICsgd18yICJdLFsyLDMsIndfMSArIHdfMiAiXSxbMCwyLCJ6XzEgKyB6XzIiLDJdLFsxLDMsInpfMSArIHpfMiIsMl0sWzQsNSwid18xICsgd18yICIsMCx7ImN1cnZlIjotMX1dLFs2LDcsIndfMSArIHdfMiAiLDAseyJjdXJ2ZSI6LTF9XSxbNCw2LCJ6XzEgKyB6XzIiLDIseyJjdXJ2ZSI6MX1dLFs1LDcsInpfMSArIHpfMiIsMix7ImN1cnZlIjoxfV0sWzcsNiwiel8xICsgel8yICIsMCx7ImN1cnZlIjotMX1dLFs3LDUsIndfMSArIHdfMiIsMix7ImN1cnZlIjoxfV0sWzUsNCwiel8xICsgel8yICIsMCx7ImN1cnZlIjotMX1dLFs2LDQsIndfMSArIHdfMiIsMix7ImN1cnZlIjoxfV1d
\[\begin{tikzcd}[sep=huge]
	ax & ay & ax & ay \\
	bx & by & bx & by
	\arrow["{w_1 + w_2 }", from=1-1, to=1-2]
	\arrow["{z_1 + z_2}"', from=1-1, to=2-1]
	\arrow["{z_1 + z_2}"', from=1-2, to=2-2]
	\arrow["{w_1 + w_2 }", curve={height=-6pt}, from=1-3, to=1-4]
	\arrow["{z_1 + z_2}"', curve={height=6pt}, from=1-3, to=2-3]
	\arrow["{z_1 + z_2 }", curve={height=-6pt}, from=1-4, to=1-3]
	\arrow["{z_1 + z_2}"', curve={height=6pt}, from=1-4, to=2-4]
	\arrow["{w_1 + w_2 }", from=2-1, to=2-2]
	\arrow["{w_1 + w_2}"', curve={height=6pt}, from=2-3, to=1-3]
	\arrow["{w_1 + w_2 }", curve={height=-6pt}, from=2-3, to=2-4]
	\arrow["{w_1 + w_2}"', curve={height=6pt}, from=2-4, to=1-4]
	\arrow["{z_1 + z_2 }", curve={height=-6pt}, from=2-4, to=2-3]
\end{tikzcd}\]
\end{lem}

\begin{proof}
The Heegaard diagrams $(\mathbb{T}^{2},\boldsymbol{\beta},\boldsymbol{\gamma},\{z_{1},z_{2}\})$
and $(\mathbb{T}^{2},\boldsymbol{\beta},\boldsymbol{\gamma},\{w_{1},w_{2}\})$
are weakly admissible, and the \emph{``fully blocked''} chain complex
$\widetilde{CF}(\boldsymbol{\beta},\boldsymbol{\gamma})$ (i.e. the
chain complex is freely generated over $\mathbb{F}$ and we only count
differentials without any basepoints) of either of these Heegaard
diagrams have no differential, since there are four generators and
$\widetilde{HF}$ has rank four.

The following are the two-chains of all the Maslov index $1$ domains:
\begin{itemize}
\item $ax\to bx,ay\to by$: $D_{2}$, $D_{3}+D_{6}$
\item $bx\to ax,by\to ay$: $D_{4}+D_{5}$, $D_{1}+D_{3}+D_{4}$, $D_{1}+D_{7}$
\item $ax\to ay,bx\to by$: $D_{3}+D_{4}$, $D_{7}$
\item $ay\to ax,by\to bx$: $D_{1}+D_{2}$, $D_{5}+D_{6}$, $D_{1}+D_{3}+D_{6}$
\end{itemize}
Among these, the domains with two-chain $D_{2}$, $D_{3}+D_{6}$,
$D_{4}+D_{5}$, $D_{3}+D_{4}$, $D_{7}$, or $D_{5}+D_{6}$ always
have exactly one holomorphic representative regardless of the almost
complex structure.

However, $D_{1}+D_{7}\in D(bx,ax)$, for instance, depends on the
almost complex structure. Note that exactly one of $D_{1}+D_{3}+D_{4}$
and $D_{1}+D_{7}\in D(bx,ax)$ has an odd number of holomorphic representatives,
since $\widetilde{CF}$ of $(\mathbb{T}^{2},\boldsymbol{\beta},\boldsymbol{\gamma},\{z_{1},z_{2}\})$
has no differentials. Similarly, exactly one of these two domains,
viewed as domains in $D(by,ay)$, contributes to the differential,
and the corresponding statement holds for $D_{1}+D_{2}$ and $D_{1}+D_{3}+D_{6}$.

Now, the lemma follows from that $\left(\partial^{-}\right)^{2}=0$
for the chain complex $CF^{-}(\boldsymbol{\beta},\boldsymbol{\gamma})$
over $R$. Indeed, we get two cases depending on whether $D_{1}+D_{3}+D_{4}$
or $D_{1}+D_{7}$ contributes to the differential.
\end{proof}
Hence, we can compute the homology group of the Heegaard datum of
Figure \ref{fig:local-band}. The Heegaard diagram is homologically
$\mathbb{Z}$-gradable and Alexander $\mathbb{Z}/2$-splittable. We
have
\[
\boldsymbol{HF}^{-}(\boldsymbol{\beta},\boldsymbol{\gamma})\simeq X\mathbb{F}\llbracket U^{1/2}\rrbracket\oplus(ay+bx)\mathbb{F}\llbracket U^{1/2}\rrbracket,
\]
where $X=by$ or $ax+by$ depending on the almost complex structure,
and 
\[
\mathbb{F}\llbracket U^{1/2}\rrbracket:=\mathbb{F}\llbracket U_{1}^{1/2},U_{2}^{1/2},U^{1/2}\rrbracket/(U_{1}^{1/2}-U^{1/2},U_{2}^{1/2}-U^{1/2}),
\]
i.e. $U_{i}^{1/2}$ acts by multiplication by $U^{1/2}$. Both $X$
and $ay+bx$ lie in torsion ${\rm Spin}^{c}$-structures and are homogeneous
with respect to the homological $\mathbb{Z}$-grading, the ${\rm Spin}^{c}$-splitting,
and the Alexander $\mathbb{Z}/2$-grading. They lie in the same homological
$\mathbb{Z}$-grading, but have different ${\rm Spin}^{c}$-structures
and different Alexander $\mathbb{Z}/2$-gradings. This time, we choose
$S(ay+bx)$ (for a stabilization of Figure \ref{fig:local-band}).
\begin{rem}
\label{rem:The-element-}The element $S(ay+bx)$ is indeed a cycle
in any almost complex structure by the same argument as in Remark
\ref{rem:The-elements-}.
\end{rem}

\begin{defn}
\label{lem:characterize-thetaB-1-1}Let a Heegaard datum with underlying
Heegaard diagram $(\Sigma,\boldsymbol{\beta},\boldsymbol{\gamma},\boldsymbol{u}\sqcup\boldsymbol{v})$
be obtained from a stabilization of Figure \ref{fig:local-band} by
handleslides and isotopies. Define the \emph{canonical element} $\Theta_{B}\in\boldsymbol{HF}^{-}(\boldsymbol{\beta},\boldsymbol{\gamma})$
as the unique nonzero element that satisfies the following properties:
\end{defn}

\begin{itemize}
\item It is homogeneous with respect to the ${\rm Spin}^{c}$-splitting.
\item It lies in the top homological $\mathbb{Z}$-grading part of $\boldsymbol{HF}^{-}(\boldsymbol{\beta},\boldsymbol{\gamma};0)$.
\item It is \emph{not} in the image of the map ${\cal HF}^{-}(\boldsymbol{\beta},\boldsymbol{\gamma})\to\boldsymbol{HF}^{-}(\boldsymbol{\beta},\boldsymbol{\gamma})$. 
\end{itemize}
\begin{defn}
\label{def:band-maps-1-1}Let $(L,\boldsymbol{u},\alpha_{L})$ be
a minimally pointed link equipped with a suture datum inside a pointed
three-manifold $(Y,\boldsymbol{v})$, $B_{1}$ and $B_{2}$ be a mergesplit
pair on $L$, and let $L'$ be the link given by surgering $L$ along
$B_{1}$ and $B_{2}$. Choose a weakly admissible Heegaard datum subordinate
to $B_{1}\cup B_{2}$, with attaching curves $\boldsymbol{\alpha},\boldsymbol{\beta},\boldsymbol{\gamma}$.
Then, the\emph{ mergesplit map} is the triangle counting map
\[
\mu_{2}(-,\Theta_{B}):\boldsymbol{CF}^{-}(\boldsymbol{\alpha},\boldsymbol{\beta})\to\boldsymbol{CF}^{-}(\boldsymbol{\alpha},\boldsymbol{\gamma}),
\]
where $\Theta_{B}\in\boldsymbol{CF}^{-}(\boldsymbol{\beta},\boldsymbol{\gamma})$
is any cycle that represents the canonical element $\Theta_{B}\in\boldsymbol{HF}^{-}(\boldsymbol{\beta},\boldsymbol{\gamma})$.
\end{defn}

As in the first paragraph of the proof of Proposition \ref{prop:band-map-well-defined-1},
this map is well-defined up to $R$-linear chain homotopy, where $R$
is the coefficient ring for both $L$ and $L'$.

\subsubsection{\label{subsec:Merge-bands}Merge bands}

Let $B$ be a merge band from a balled link $L$ to $L'$, and let
$B_{lf}$ be the link baseball of $L$ that becomes a free baseball
in $L'$. To define the merge map, we choose an auxiliary datum $(\widetilde{L},\boldsymbol{u},\boldsymbol{v})$
for $L$ together with two more auxiliary data: an auxiliary band
$\widetilde{B}$ and an auxiliary disk $D$ (see Figure \ref{fig:merge-composition}).
We assume that they are chosen in the following way.

Let $L_{1}\subset L$ be the component of $L$ that $B_{lf}$ intersects.
Then, we consider an embedded disk 
\[
\{(x,y)\in\mathbb{R}^{2}:|(x,y)|\le1\}\hookrightarrow B_{lf}
\]
such that the right hand side boundary $\{(x,y):|(x,y)|=1,x\ge0\}$
maps to $\widetilde{L}\cap B_{lf}$, and contains the two link basepoints
of $\widetilde{L}$ that are inside $B_{lf}$. Let $\widetilde{B}$
be the image of $\{(x,y):x\le0.5\}$, and let $D$ be the image of
$\{(x,y):x\ge0.5\}$. Note that $B,\widetilde{B}$ is a mergesplit
pair on $\widetilde{L}$. Let $\widetilde{L'}$ be the link such that
$\widetilde{L'}\sqcup\partial D$ is the link obtained by surgering
$\widetilde{L}$ along $B$ and $\widetilde{B}$. If $z,w$ are the
two link basepoints of $\widetilde{L}$ inside $B_{lf}$, then $(\widetilde{L'},\boldsymbol{u}\backslash\{z,w\},\boldsymbol{v}\cup\{z\})$
is an auxiliary datum for $L'$.
\begin{defn}
\label{def:band-maps}Let $B$ be a merge band from a balled link
$L$ to $L'$, let $(\widetilde{L},\boldsymbol{u},\boldsymbol{v})$
be an auxiliary datum for $L$, and let $\widetilde{B}$ and $D$
be an auxiliary band and an auxiliary disk as above. The band map
is the composite
\begin{multline*}
\boldsymbol{CFL'}^{-}((Y,\boldsymbol{v}),(\widetilde{L},\boldsymbol{u},\alpha_{\widetilde{L}}))\xrightarrow{{\rm mergesplit}}\boldsymbol{CFL'}^{-}((Y,\boldsymbol{v}),(\widetilde{L'}\sqcup\partial D,\boldsymbol{u},\alpha_{\widetilde{L'}\sqcup\partial D}))\\
\xrightarrow{{\rm annihilation}}\boldsymbol{CFL'}^{-}((Y,\boldsymbol{v}\sqcup\{z\}),(\widetilde{L'},\boldsymbol{u}\backslash\{z,w\},\alpha_{\widetilde{L'}})).
\end{multline*}
\end{defn}

Let $S$ be the coefficient ring of $L'$. Similarly to Proposition
\ref{prop:band-map-well-defined-1}, the band map for the merge band
$B$ is $S$-linear and is well-defined up to $S$-linear chain homotopy:
indeed, any two choices of $\widetilde{B},D$ are isotopic.

\subsubsection{\label{subsec:Split-bands}Split bands}

We define the band maps for split bands similarly.
\begin{defn}
\label{def:band-maps-2}Let $B$ be a split band from a balled link
$L$ to $L'$, and choose an auxiliary datum $(\widetilde{L'},\boldsymbol{u},\boldsymbol{v})$
for $L'$ and an auxiliary band $\widetilde{B}$ and an auxiliary
disk $D$ on $\widetilde{L'}$ as in Subsubsection \ref{subsec:Merge-bands}
(see Figure \ref{fig:merge-composition}). Let $z,w$ be the two link
basepoints in the link baseball of $L'$ that is a free baseball in
$L$. Let $\widetilde{L}$ be the link such that $\widetilde{L}\sqcup\partial D$
is the link obtained by surgering $\widetilde{L'}$ along $B$ and
$\widetilde{B}$, and consider the suture datum $\alpha_{\widetilde{L}}$
obtained from $\alpha_{\widetilde{L'}}$. Then, the dual bands of
$B,\widetilde{B}$ form a mergesplit pair on $\widetilde{L}\sqcup\partial D$.
The band map is the composite
\begin{multline*}
\boldsymbol{CFL'}^{-}((Y,\boldsymbol{v}\sqcup\{z\}),(\widetilde{L},\boldsymbol{u}\backslash\{z,w\},\alpha_{\widetilde{L}}))\xrightarrow{{\rm creation}}\boldsymbol{CFL'}^{-}((Y,\boldsymbol{v}),(\widetilde{L}\sqcup\partial D,\boldsymbol{u},\alpha_{\widetilde{L}\sqcup\partial D}))\\
\xrightarrow{{\rm mergesplit}}\boldsymbol{CFL'}^{-}((Y,\boldsymbol{v}),(\widetilde{L'},\boldsymbol{u},\alpha_{\widetilde{L'}})).
\end{multline*}
\end{defn}

Let $S$ be the coefficient ring of $L$. The band map for the split
band $B$ is $S$-linear and is well-defined up to $S$-linear chain
homotopy.

\subsubsection{\label{subsec:Merge-and-split}Merge and split maps in one go}

It is possible to define the merge and split band maps using a single
triple Heegaard diagram.

Let $B$ be a \textbf{merge} band from $L$ to $L'$. We say that
a Heegaard triple diagram $(\Sigma,\boldsymbol{\alpha},\boldsymbol{\beta},\boldsymbol{\gamma},\boldsymbol{u}\sqcup\boldsymbol{v})$
is \emph{subordinate to $B$} if there exists some choice of auxiliary
data $(\widetilde{L},\boldsymbol{u},\boldsymbol{v})$, $\widetilde{B}$,
and $D$ as in Subsubsection \ref{subsec:Merge-bands}, such that:
if $z,w$ are the two basepoints on the link component of $\widetilde{L}$
that intersects $D$, then
\begin{itemize}
\item the Heegaard diagram is subordinate to the mergesplit pair $B,\widetilde{B}$,
and
\item there exists a path $P$ between $z$ and $w$ in a connected component
of $\Sigma\backslash(\boldsymbol{\alpha}\cup\boldsymbol{\gamma})$,
such that $P$ represents the disk $D$ and $|P\cap\boldsymbol{\beta}|=1$\footnote{The condition $|P\cap\boldsymbol{\beta}|=1$ is not necessary, but
we assume this for Remark \ref{rem:band-non-trivial-local-system}
and Section \ref{sec:An-unoriented-skein}.}.
\end{itemize}

Choose a weakly admissible Heegaard datum whose Heegaard diagram is
subordinate to $B$, such that the sub Heegaard datum with attaching
curves $\boldsymbol{\alpha},\boldsymbol{\beta}$ represents $(\widetilde{L},\boldsymbol{u},\boldsymbol{v})$.
Let $R$ (resp. $S$) be the coefficient ring for $L$ (resp. $L'$).
If $z,w$ have weight $U_{k}^{1/2}$, then $R=S\oplus U_{k}^{1/2}S$
as $S$-modules. Since $z,w$ lie in the same connected component
of $\Sigma\backslash(\boldsymbol{\alpha}\cup\boldsymbol{\gamma})$,
we can define $\boldsymbol{CF}_{S}^{-}(\boldsymbol{\alpha},\boldsymbol{\gamma})$,
and we have an $S$-linear splitting of chain complexes $\boldsymbol{CF}_{R}^{-}(\boldsymbol{\alpha},\boldsymbol{\gamma})\simeq\boldsymbol{CF}_{S}^{-}(\boldsymbol{\alpha},\boldsymbol{\gamma})\oplus U_{k}^{1/2}\boldsymbol{CF}_{S}^{-}(\boldsymbol{\alpha},\boldsymbol{\gamma})$.
The merge map is the composite
\[
\boldsymbol{CF}_{R}^{-}(\boldsymbol{\alpha},\boldsymbol{\beta})\xrightarrow{\mu_{2}(-\otimes\Theta_{B})}\boldsymbol{CF}_{R}^{-}(\boldsymbol{\alpha},\boldsymbol{\gamma})\to U_{k}^{1/2}\boldsymbol{CF}_{S}^{-}(\boldsymbol{\alpha},\boldsymbol{\gamma})\xrightarrow{\cdot U_{k}^{-1/2}}\boldsymbol{CF}_{S}^{-}(\boldsymbol{\alpha},\boldsymbol{\gamma})
\]
where the second map is projection onto the second summand.

Similarly, if $B$ is a \textbf{split} band from $L$ to $L'$, we
say that a Heegaard triple diagram $(\Sigma,\boldsymbol{\alpha},\boldsymbol{\beta},\boldsymbol{\gamma},\boldsymbol{u}\sqcup\boldsymbol{v})$
is \emph{subordinate to $B$} if there exists some choice of auxiliary
data $(\widetilde{L'},\boldsymbol{u},\boldsymbol{v})$, $\widetilde{B}$,
and $D$ as in Subsubsection \ref{subsec:Split-bands}, such that:
if $z,w$ are the two basepoints on the link component of $\widetilde{L'}$
that intersects $D$, then
\begin{itemize}
\item the Heegaard diagram is subordinate to the mergesplit pair given by
the dual bands of $B,\widetilde{B}$, and
\item there exists a path $P$ between $z$ and $w$ in a connected component
of $\Sigma\backslash(\boldsymbol{\alpha}\cup\boldsymbol{\beta})$,
such that $P$ represents the disk $D$ and $|P\cap\boldsymbol{\gamma}|=1$.
\end{itemize}

Choose a weakly admissible Heegaard datum whose Heegaard diagram is
subordinate to $B$, such that the sub Heegaard datum with attaching
curves $\boldsymbol{\alpha},\boldsymbol{\gamma}$ represents $(\widetilde{L'},\boldsymbol{u},\boldsymbol{v})$.
Let $R$ (resp. $S$) be the coefficient ring for $L'$ (resp. $L$).
If $z,w$ have weight $U_{k}^{1/2}$, then as before, we have an $S$-linear
splitting of chain complexes $\boldsymbol{CF}_{R}^{-}(\boldsymbol{\alpha},\boldsymbol{\beta})\simeq\boldsymbol{CF}_{S}^{-}(\boldsymbol{\alpha},\boldsymbol{\beta})\oplus U_{k}^{1/2}\boldsymbol{CF}_{S}^{-}(\boldsymbol{\alpha},\boldsymbol{\beta})$.
The split map is the composite
\[
\boldsymbol{CF}_{S}^{-}(\boldsymbol{\alpha},\boldsymbol{\beta})\to\boldsymbol{CF}_{R}^{-}(\boldsymbol{\alpha},\boldsymbol{\beta})\xrightarrow{\mu_{2}(-\otimes\Theta_{B})}\boldsymbol{CF}_{R}^{-}(\boldsymbol{\alpha},\boldsymbol{\gamma}),
\]
where the first map is inclusion into the first summand.
\begin{rem}
\label{rem:band-non-trivial-local-system}We can describe the merge
and split band maps using local systems. We consider the merge band
map first. Let us consider a Heegaard triple diagram subordinate to
$B$ and let $R$ and $S$ be as above. Work over the coefficient
ring $S$. Instead of considering both basepoints $w$ and $z$, we
only consider one of them, say $z$. Let $G$ be the oriented arc
from $z$ to $w$ given by the path $P$. Consider the rank $2$ local
system $E$ on $\boldsymbol{\beta}$ described in Subsection \ref{subsec:unoriented-local-system}.
By Remark \ref{rem:same-as-unoriented}, $\boldsymbol{CF}_{R}^{-}(\boldsymbol{\alpha},\boldsymbol{\beta})\simeq\boldsymbol{CF}_{S}^{-}(\boldsymbol{\alpha},\boldsymbol{\beta}^{E})$
and $\boldsymbol{CF}_{R}^{-}(\boldsymbol{\beta},\boldsymbol{\gamma})\simeq\boldsymbol{CF}_{S}^{-}(\boldsymbol{\beta}^{E},\boldsymbol{\gamma})$,
and under this identification, the merge band map is just
\[
\mu_{2}(-\otimes\Theta_{B}):\boldsymbol{CF}_{S}^{-}(\boldsymbol{\alpha},\boldsymbol{\beta}^{E})\to\boldsymbol{CF}_{S}^{-}(\boldsymbol{\alpha},\boldsymbol{\gamma}).
\]

The split band map is similar: in this case, $\boldsymbol{\gamma}$
is the attaching curve that intersects $G$ and has a nontrivial local
system, and the split band map is just
\[
\mu_{2}(-\otimes\Theta_{B}):\boldsymbol{CF}_{S}^{-}(\boldsymbol{\alpha},\boldsymbol{\beta})\to\boldsymbol{CF}_{S}^{-}(\boldsymbol{\alpha},\boldsymbol{\gamma}^{E}).
\]
\end{rem}

If $B$ is a merge or split band, then the band map is a priori only
$S$-linear. Let the $k$th and $\ell$th baseballs be the link baseballs
on the link components of $L$ (if $B$ is a merge band) or $L'$
(if $B$ is a split band) that $B$ intersects. If we view $S$ as
the $R$-module $R/(U_{k}^{1/2}-U_{\ell}^{1/2})$, then the merge
and split band maps are $R$-linear on homology. In other words, we
have the following.
\begin{prop}
\label{prop:actions-commute}Let $F_{B}$ be the band map between
the chain complexes. If $B$ is a merge band, then $F_{B}(U_{k}^{1/2}\cdot-)$
and $F_{B}(U_{\ell}^{1/2}\cdot-)$ are $S$-linearly chain homotopic.
If $B$ is a split band, then $U_{k}^{1/2}F_{B}$ and $U_{\ell}^{1/2}F_{B}$
are $S$-linearly chain homotopic.
\end{prop}

\begin{proof}
There exists some $\zeta\in\boldsymbol{CF}_{R}^{-}(\boldsymbol{\beta},\boldsymbol{\gamma})$
such that $\partial\zeta=U_{k}^{1/2}\Theta_{B}-U_{\ell}^{1/2}\Theta_{B}$.
Hence, 
\[
\mu_{2}(-\otimes\zeta):\boldsymbol{CF}_{R}^{-}(\boldsymbol{\alpha},\boldsymbol{\beta})\to\boldsymbol{CF}_{R}^{-}(\boldsymbol{\alpha},\boldsymbol{\gamma})
\]
is an $R$-linear chain homotopy between $U_{k}^{1/2}\mu_{2}(-\otimes\Theta_{B})$
and $U_{\ell}^{1/2}\mu_{2}(-\otimes\Theta_{B})$. The proposition
follows from this.
\end{proof}

\subsection{\label{subsec:spinc-structures-and-grading}${\rm Spin}^{c}$-structures
and relative homological gradings}

Let $B$ be a band from a balled link $L$ to $L'$, and let $(\Sigma,\boldsymbol{\alpha},\boldsymbol{\beta},\boldsymbol{\gamma},\boldsymbol{u}\sqcup\boldsymbol{v})$
be a Heegaard triple diagram that is \emph{subordinate to $B$}. Then
$\boldsymbol{\beta}$ and $\boldsymbol{\gamma}$ are handlebody-equivalent
and $\Theta_{B}\in\boldsymbol{HF}^{-}(\boldsymbol{\beta},\boldsymbol{\gamma};0)$,
and so in Definitions \ref{def:band-maps-1} and \ref{def:band-maps-1-1},
if we chose a cycle $\Theta_{B}\in\boldsymbol{CF}^{-}(\boldsymbol{\beta},\boldsymbol{\gamma};0)$
(that represents the canonical element $\Theta_{B}\in\boldsymbol{HF}^{-}(\boldsymbol{\beta},\boldsymbol{\gamma};0)$)
that is homogeneous with respect to the relative homological $\mathbb{Z}$-grading,
then Proposition \ref{prop:torsion} implies that for each $c\in H^{2}(Y;\mathbb{Z})$,
the band maps restrict to 
\[
\boldsymbol{HFL'}^{-}(Y,L;c)\to\boldsymbol{HFL'}^{-}(Y,L';c),
\]
(one can show that the creation and annihilation maps preserve the
Chern class summands using Equation (\ref{eq:c1-average})) and the
band maps preserve relative homological gradings.

We can work with $c$-strongly admissible triples (Definition \ref{def:strongly-admissible-multi})
and define $HFL'{}^{-}(Y,L;c)\to HFL'{}^{-}(Y,L';c)$ as well.

\section{\label{sec:An-unoriented-skein}An unoriented skein exact triangle}

We can now state Theorem \ref{thm:unoriented-exact} precisely. Recall
that for merge and split bands between balled links, the type of exactly
one baseball had to change.
\begin{defn}
\label{def:unoriented-skein-triple-balled}Three balled links $L_{a},L_{b},L_{c}\subset Y$
form an \emph{unoriented skein triple }if the following hold:
\end{defn}

\begin{itemize}
\item there exists a ball $B^{3}\subset Y$ that is disjoint from all the
baseballs, such that the underlying links are identical outside $B^{3}$,
and they differ as in Figure \ref{fig:skein-moves} inside $B^{3}$,
and 
\item the baseballs and the orders, markings, and types of the baseballs
of $L_{a},L_{b},L_{c}$ are identical, except that there exists exactly
one baseball whose type can change.
\end{itemize}
\begin{thm}
\label{thm:unoriented-exact-triangle}If three balled links $L_{a},L_{b},L_{c}\subset Y$
form an unoriented skein triple, then the band maps given by the bands
described in Figure \ref{fig:skein-moves-bands} give rise to exact
triangles 
\begin{gather*}
\cdots\to\boldsymbol{HFL'}^{-}(Y,L_{a})\to\boldsymbol{HFL'}^{-}(Y,L_{b})\to\boldsymbol{HFL'}^{-}(Y,L_{c})\to\boldsymbol{HFL'}^{-}(Y,L_{a})\to\cdots\\
\cdots\to\boldsymbol{HFL'}^{\infty}(Y,L_{a})\to\boldsymbol{HFL'}^{\infty}(Y,L_{b})\to\boldsymbol{HFL'}^{\infty}(Y,L_{c})\to\boldsymbol{HFL'}^{\infty}(Y,L_{a})\to\cdots\\
\cdots\to\widehat{HFL'}(Y,L_{a})\to\widehat{HFL'}(Y,L_{b})\to\widehat{HFL'}(Y,L_{c})\to\widehat{HFL'}(Y,L_{a})\to\cdots.
\end{gather*}
\end{thm}

The proof of Theorem \ref{thm:unoriented-exact-triangle} boils down
to the following local computation on the torus.
\begin{thm}
\label{thm:sym2-t2}Consider the Heegaard datum given in Figure \ref{fig:drawing-genus1}.
The attaching curve $\boldsymbol{\beta}_{c}$ has a nontrivial, rank
$2$ local system $E$ given by the oriented arc $G$ as in Subsection
\ref{subsec:unoriented-local-system}. Let 
\[
\boldsymbol{\rho}:=e_{0}\rho_{1}+e_{0}\rho_{2}\in\boldsymbol{CF}^{-}(\boldsymbol{\beta}_{b},\boldsymbol{\beta}_{c}^{E}),\ \boldsymbol{\sigma}:=e_{1}^{\ast}\sigma_{1}+e_{1}^{\ast}\sigma_{2}\in\boldsymbol{CF}^{-}(\boldsymbol{\beta}_{c}^{E},\boldsymbol{\beta}_{a}).
\]
Then, $\tau\in\boldsymbol{CF}^{-}(\boldsymbol{\beta}_{a},\boldsymbol{\beta}_{b})$
is a cycle, and the following dashed maps between the twisted complexes
$\boldsymbol{\beta}_{a}\xrightarrow{\tau}\boldsymbol{\beta}_{b}$
and $\boldsymbol{\beta}_{c}^{E}$ are morally homotopy inverses (Definition
\ref{def:moral-def}), i.e. if we write them as $\underline{\boldsymbol{\rho}}$
and $\underline{\boldsymbol{\sigma}}$, respectively, then they are
cycles and their compositions $\mu_{2}(\underline{\boldsymbol{\rho}},\underline{\boldsymbol{\sigma}'})$
and $\mu_{2}(\underline{\boldsymbol{\sigma}},\underline{\boldsymbol{\rho}'})$
are $\Theta^{+}$.% https://q.uiver.app/#q=WzAsNixbMCwwLCJcXGJvbGRzeW1ib2x7XFxiZXRhfV97YX0iXSxbMSwwLCJcXGJvbGRzeW1ib2x7XFxiZXRhfV97Yn0iXSxbMSwxLCJcXGJvbGRzeW1ib2x7XFxiZXRhfV97Y31eRSJdLFszLDEsIlxcYm9sZHN5bWJvbHtcXGJldGF9X3tifSJdLFsyLDAsIlxcYm9sZHN5bWJvbHtcXGJldGF9X3tjfV5FIl0sWzIsMSwiXFxib2xkc3ltYm9se1xcYmV0YX1fe2F9Il0sWzAsMSwiXFx0YXUgIl0sWzEsMiwiXFxib2xkc3ltYm9se1xccmhvfSIsMCx7InN0eWxlIjp7ImJvZHkiOnsibmFtZSI6ImRhc2hlZCJ9fX1dLFs0LDUsIlxcYm9sZHN5bWJvbHtcXHNpZ21hfSIsMCx7InN0eWxlIjp7ImJvZHkiOnsibmFtZSI6ImRhc2hlZCJ9fX1dLFs1LDMsIlxcdGF1Il1d
\[\begin{tikzcd}[sep=large]
	{\boldsymbol{\beta}_{a}} & {\boldsymbol{\beta}_{b}} & {\boldsymbol{\beta}_{c}^E} \\
	& {\boldsymbol{\beta}_{c}^E} & {\boldsymbol{\beta}_{a}} & {\boldsymbol{\beta}_{b}}
	\arrow["{\tau }", from=1-1, to=1-2]
	\arrow["{\boldsymbol{\rho}}", dashed, from=1-2, to=2-2]
	\arrow["{\boldsymbol{\sigma}}", dashed, from=1-3, to=2-3]
	\arrow["\tau", from=2-3, to=2-4]
\end{tikzcd}\]
\begin{figure}[h]
\begin{centering}
\includegraphics{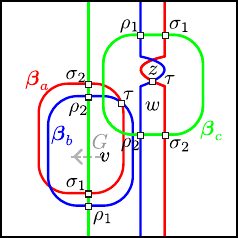}
\par\end{centering}
\caption{\label{fig:drawing-genus1}The genus $1$ diagram for Theorem \ref{thm:sym2-t2}.
Unless otherwise specified, work over $\mathbb{F}\llbracket U_{1},U_{2}^{1/2}\rrbracket$
and assign the weight $U_{1}$ to $v$, and $U_{2}^{1/2}$ to $w$
and $z$.}
\end{figure}
\end{thm}

We carry out the proof of Theorem \ref{thm:sym2-t2} in Sections \ref{sec:The-plan},
\ref{sec:proof-step1}, \ref{sec:proof-step3}, and \ref{sec:composition}.
\begin{proof}[Proof of Theorem \ref{thm:unoriented-exact-triangle} assuming Theorem
\ref{thm:sym2-t2}]
We focus on the minus version: the statements for the hat and infinity
versions can be deduced using Lemma \ref{lem:hat-infty-alg}.

Without loss of generality, assume that $L_{c}$ has one more link
component than $L_{a}$ and $L_{b}$. This is equivalent to that the
two endpoints on the left hand side of each tangle in Figure \ref{fig:skein-moves-bands}
are connected to each other and the two endpoints on the right hand
side are connected to each other. Choose a weakly admissible Heegaard
datum with Heegaard diagram $(\Sigma,\boldsymbol{\alpha},S(\boldsymbol{\beta}_{a}),S(\boldsymbol{\beta}_{b}),S(\boldsymbol{\beta}_{c}),\boldsymbol{p},G)$
such that
\begin{itemize}
\item its coefficient ring and its weight function together with the Heegaard
diagram $(\Sigma,\boldsymbol{\alpha},S(\boldsymbol{\beta}_{a}),\boldsymbol{p})$
represent the balled link $L_{a}$ in $Y$,
\item the sub Heegaard datum with attaching curves $S(\boldsymbol{\beta}_{a}),S(\boldsymbol{\beta}_{b}),S(\boldsymbol{\beta}_{c})$
is a stabilization of Figure \ref{fig:drawing-genus1},
\item the triple diagram $(\Sigma,\boldsymbol{\alpha},S(\boldsymbol{\beta}_{a}),S(\boldsymbol{\beta}_{b}),\boldsymbol{p})$
(for some choice of auxiliary datum) is subordinate to the corresponding
non-orientable band, and
\item the triple diagrams $(\Sigma,\boldsymbol{\alpha},S(\boldsymbol{\beta}_{b}),S(\boldsymbol{\beta}_{c}),\boldsymbol{p}\sqcup\{\partial_{+}G\})$
and $(\Sigma,\boldsymbol{\alpha},S(\boldsymbol{\beta}_{c}),S(\boldsymbol{\beta}_{a}),\boldsymbol{p}\sqcup\{\partial_{+}G\})$
are subordinate to the corresponding bands, where $G$ corresponds
to the path $P$ from Subsubsection \ref{subsec:Merge-and-split}
as in Remark \ref{rem:band-non-trivial-local-system}.
\end{itemize}

Note that $\boldsymbol{CF}^{-}(\boldsymbol{\alpha},S(\boldsymbol{\beta}_{a}))=\boldsymbol{CFL'}^{-}(Y,L_{a})$,
$\boldsymbol{CF}^{-}(\boldsymbol{\alpha},S(\boldsymbol{\beta}_{b}))=\boldsymbol{CFL'}^{-}(Y,L_{b})$,
and $\boldsymbol{CF}^{-}(\boldsymbol{\alpha},S(\boldsymbol{\beta}_{c}^{E}))=\boldsymbol{CFL'}^{-}(Y,L_{c})$.
Also, the cycles $S(\boldsymbol{\rho})\in\boldsymbol{CF}^{-}(S(\boldsymbol{\beta}_{b}),S(\boldsymbol{\beta}_{c}^{E}))$,
$S(\boldsymbol{\sigma})\in\boldsymbol{CF}^{-}(S(\boldsymbol{\beta}_{c}^{E}),S(\boldsymbol{\beta}_{a}))$,
and $S(\tau)\in\boldsymbol{CF}^{-}(S(\boldsymbol{\beta}_{a}),S(\boldsymbol{\beta}_{b}))$
are the cycles that define the band maps. (Recall from Remarks \ref{rem:The-elements-}
and \ref{rem:The-element-} that they are indeed cycles.)

Let $\underline{S(\boldsymbol{\beta}_{ab})}:=S(\boldsymbol{\beta}_{a})\xrightarrow{S(\tau)}S(\boldsymbol{\beta}_{b})$,
and let $\underline{S(\boldsymbol{\rho})}:\underline{S(\boldsymbol{\beta}_{ab})}\to S(\boldsymbol{\beta}_{c}^{E})$
and $\underline{S(\boldsymbol{\sigma})}:S(\boldsymbol{\beta}_{c}^{E})\to\underline{S(\boldsymbol{\beta}_{ab})}$
be given by $S(\boldsymbol{\rho})$ and $S(\boldsymbol{\sigma})$,
respectively. By Proposition \ref{prop:stabilization}, $\underline{S(\boldsymbol{\rho})}$
and $\underline{S(\boldsymbol{\sigma})}$ are cycles and $\mu_{2}(\underline{S(\boldsymbol{\rho})},\underline{S(\boldsymbol{\sigma})'})$
and $\mu_{2}(\underline{S(\boldsymbol{\sigma})},\underline{S(\boldsymbol{\rho})'})$
are $\Theta^{+}$ in a pinched almost complex structure.

To consider $\boldsymbol{\alpha}$ as well, we work in an unpinched
almost complex structure; consider the $A_{\infty}$-functor given
by changing the almost complex structures. The cycles $S(\tau)$,
$S(\boldsymbol{\rho})$, and $S(\boldsymbol{\sigma})$ are invariant
under the $A_{\infty}$-functor as there are no relevant Maslov index
$0$ domains. Hence, the twisted complex $\underline{S(\boldsymbol{\beta}_{ab})}$
is invariant under the $A_{\infty}$-functor, and if we let $\underline{f}$
(resp. $\underline{g}$) be the image of $\underline{S(\boldsymbol{\rho})}$
(resp. $\underline{S(\boldsymbol{\sigma})}$) under the $A_{\infty}$-functor,
then the $S(\boldsymbol{\beta}_{b})\to S(\boldsymbol{\beta}_{c}^{E})$
component of $\underline{f}$ is still $S(\boldsymbol{\rho})$\footnote{This is all we need, but actually there is no $S(\boldsymbol{\beta}_{a})\to S(\boldsymbol{\beta}_{c}^{E})$
component because of grading reasons.} and the $S(\boldsymbol{\beta}_{c}^{E})\to S(\boldsymbol{\beta}_{a})$
component of $\underline{g}$ is still $S(\boldsymbol{\sigma})$.
In particular, these define the band maps. Let $\underline{f'}$ (resp.
$\underline{g'}$) be the image of $\underline{S(\boldsymbol{\rho})'}$
(resp. $\underline{S(\boldsymbol{\sigma})'}$) under the $A_{\infty}$-functor;
then similar statements hold for $\underline{f'},\underline{g'}$.

The following maps are quasi-isomorphisms by Lemma \ref{lem:theta-quasi-iso}:
\begin{gather*}
\mu_{2}(-,\mu_{2}(\underline{f},\underline{g'})):\boldsymbol{CF}^{-}(\boldsymbol{\alpha},\underline{S(\boldsymbol{\beta}_{ab})})\to\boldsymbol{CF}^{-}(\boldsymbol{\alpha},\underline{S(\boldsymbol{\beta}_{ab})'})\\
\mu_{2}(-,\mu_{2}(\underline{g'},\underline{f})):\boldsymbol{CF}^{-}(\boldsymbol{\alpha},S(\boldsymbol{\beta}_{c}^{E})')\to\boldsymbol{CF}^{-}(\boldsymbol{\alpha},S(\boldsymbol{\beta}_{c}^{E})).
\end{gather*}
Hence, 
\[
\mu_{2}(-,\underline{f}):\boldsymbol{CF}^{-}(\boldsymbol{\alpha},\underline{S(\boldsymbol{\beta}_{ab})})\to\boldsymbol{CF}^{-}(\boldsymbol{\alpha},S(\boldsymbol{\beta}_{c}^{E}))
\]
is a quasi-isomorphism. Similarly, $\mu_{2}(-,\underline{g})$ is
also a quasi-isomorphism. Hence, the triangles in Theorem \ref{thm:unoriented-exact-triangle}
are exact at $L_{b}$ and $L_{a}$ by Lemmas \ref{lem:exact-triangle}
and \ref{lem:hat-infty-alg}.

That the triangles are exact at $L_{c}$ follows similarly, using
the following: since $\mu_{2}(\underline{\boldsymbol{\rho}},\underline{\boldsymbol{\sigma}'})$
and $\mu_{2}(\underline{\boldsymbol{\sigma}},\underline{\boldsymbol{\rho}'})$
are $\Theta^{+}$, if we consider the twisted complex $\text{\ensuremath{\underline{\boldsymbol{\beta}_{bc}}:=\boldsymbol{\beta}_{b}\xrightarrow{\boldsymbol{\rho}}\boldsymbol{\beta}_{c}^{E}} }$
and consider the maps $\underline{\tau}:\boldsymbol{\beta}_{a}\to\underline{\boldsymbol{\beta}_{bc}}$
and $\underline{\boldsymbol{\sigma}}:\underline{\boldsymbol{\beta}_{bc}}\to\boldsymbol{\beta}_{a}$
given by $\tau$ and $\boldsymbol{\sigma}$, respectively, then they
are cycles and $\mu_{2}(\underline{\tau},\underline{\boldsymbol{\sigma}'})$
and $\mu_{2}(\underline{\boldsymbol{\sigma}},\underline{\tau'})$
are also $\Theta^{+}$.
\end{proof}
\begin{rem}
\label{rem:skein-exact-hfl}Recall from Subsection \ref{subsec:spinc-structures-and-grading}
that we can define band maps on $HFL'{}^{-}(Y,L;c)$ for $c\in H^{2}(Y;\mathbb{Z})$.
The above proof works for $HFL'{}^{-}$ as well, if we work with $c$-strongly
admissible diagrams (Subsubsection \ref{subsec:Strong-admissibility-c}).
\end{rem}

Although it should be possible to show Theorem \ref{thm:sym2-t2}
directly, we find it easier to use a trick than to carry out the necessary
triangle and quadrilateral map computations in ${\rm Fuk}({\rm Sym}^{2}(\mathbb{T}^{2}))$.
Our trick is introducing intermediate objects that are morally quasi-isomorphic
to both $\boldsymbol{\beta}_{a}\xrightarrow{\tau}\boldsymbol{\beta}_{b}$
and $\boldsymbol{\beta}_{c}^{E}$, which reduces the local computations
to local computations in the Fukaya category of $\mathbb{T}^{2}$,
which is combinatorial.

These intermediate objects have nice topological interpretations.
Indeed, we get the following theorem as a corollary.
\begin{thm}
\label{thm:2surgery}Let $Y$ be a closed, oriented three-manifold
$Y$, let $K\subset Y$ be a knot, and let $L\subset Y$ be a link
(which may be empty). Consider a balled link in $Y$ whose underlying
link is $K\sqcup L$, such that the link baseball of $K$ is marked
with $1$ or $\infty$. Denote this balled link also as $K\sqcup L$.

Let $\lambda$ be a framing of $K$, and let $\mu$ be the meridian
of $K$. Consider $L\subset Y_{\lambda}(K)$ and $L\subset Y_{\lambda+2\mu}(K)$
as balled links, where the baseballs and the orders, markings, and
types of the baseballs are the same as those of $K\sqcup L\subset Y$,
except that the link baseball of $K$ is a free baseball for $L\subset Y_{\lambda}(K)$
and $L\subset Y_{\lambda+2\mu}(K)$.

There exist exact triangles
\begin{gather*}
\cdots\to\boldsymbol{HFL'}^{-}(Y,K\sqcup L)\to\boldsymbol{HFL'}^{-}(Y_{\lambda}(K),L)\to\boldsymbol{HFL'}^{-}(Y_{\lambda+2\mu}(K),L)\to\boldsymbol{HFL'}^{-}(Y,K\sqcup L)\to\cdots\\
\cdots\to\boldsymbol{HFL'}^{\infty}(Y,K\sqcup L)\to\boldsymbol{HFL'}^{\infty}(Y_{\lambda}(K),L)\to\boldsymbol{HFL'}^{\infty}(Y_{\lambda+2\mu}(K),L)\to\boldsymbol{HFL'}^{\infty}(Y,K\sqcup L)\to\cdots\\
\cdots\to\widehat{HFL'}(Y,K\sqcup L)\to\widehat{HFL'}(Y_{\lambda}(K),L)\to\widehat{HFL'}(Y_{\lambda+2\mu}(K),L)\to\widehat{HFL'}(Y,K\sqcup L)\to\cdots.
\end{gather*}
These maps come from Heegaard triple diagrams, and are $S$-linear
where $S$ is the coefficient ring of $L\subset Y_{\lambda}(K)$.
\end{thm}

\begin{proof}
This reduces to the local computation carried out in Section \ref{sec:proof-step3},
using an argument similar to the above proof of Theorem \ref{thm:unoriented-exact-triangle}
assuming Theorem \ref{thm:sym2-t2}.
\end{proof}
We recover Theorem \ref{thm:2-surgery-intro} from the hat version
of Theorem \ref{thm:2surgery} by taking $L$ to be the empty link
and letting the balled link $K\sqcup L$ have exactly one baseball,
which is marked with $1$.
\begin{rem}
\label{rem:2-exact-hfl}Theorem \ref{thm:2surgery} does not work
for $HFL'{}^{-}$\footnote{Indeed, we are not in the setting of Subsection \ref{subsec:Homologous-attaching-curves}.},
just like Ozsv\'{a}th and Szab\'{o}'s surgery exact triangles \cite[Section 9]{MR2113020}.
This is because, morally speaking, the maps correspond to cobordism
maps where we sum over infinitely many different ${\rm Spin}^{c}$-structures.
\end{rem}

\section{\label{sec:Computation-for-unlinks}Computation for planar links}

In this section, we deduce from Theorem \ref{thm:unoriented-exact-triangle}
that the band maps for planar links (\emph{``planar band maps''})
in the various versions of link Floer homology coincide with band
maps in the various versions of Khovanov homology. More precisely,
we consider bands $B:L\to L'$ between balled links in $S^{3}$, such
that the band $B$ and the underlying links of $L$ and $L'$ (ignoring
the baseballs) all lie in $S^{2}\subset S^{3}$. Note that $L$ and
$L'$ are unlinks, and that $B$ is necessarily a merge or split band.
We show the following:
\begin{itemize}
\item The planar band maps on $\boldsymbol{HFL'}^{-}$ agree with the planar
band maps on equivariant Khovanov homology given by the Lee deformation.
\item If exactly one baseball is marked with $1$ (resp. $1/2$) and the
rest are marked with $\infty$, then the planar band maps on $\widehat{HFL'}$
agree with the planar band maps on unreduced (resp. reduced) Khovanov
homology.
\end{itemize}

First, let us compute the unoriented link Floer homology of unlinks.
\begin{prop}
\label{prop:Let--be}Let $L$ be a balled link in $S^{3}$, whose
underlying link is an unlink. Let the link baseballs correspond to
variables $U_{1},\cdots,U_{k}$. Then,
\[
\boldsymbol{HFL'}^{-}(S^{3},L)\simeq\mathbb{F}\llbracket U_{1}^{1/2},\cdots,U_{k}^{1/2},U\rrbracket/(U_{1}-U,\cdots,U_{k}-U)
\]
over the coefficient ring, hence in particular over $\mathbb{F}\llbracket U_{1}^{1/2},\cdots,U_{k}^{1/2}\rrbracket$.

If exactly one baseball is marked with $1$ and the rest are marked
with $\infty$, then 
\[
\widehat{HFL'}(S^{3},L)\simeq\mathbb{F}\llbracket U_{1}^{1/2},\cdots,U_{k}^{1/2}\rrbracket/(U_{1},\cdots,U_{k})
\]
over the coefficient ring.

If exactly one baseball is marked with $1/2$ (say the one that corresponds
to the variable $U_{1}$) and the rest are marked with $\infty$,
then 
\[
\widehat{HFL'}(S^{3},L)\simeq\mathbb{F}\llbracket U_{1}^{1/2},\cdots,U_{k}^{1/2}\rrbracket/(U_{1}^{1/2},U_{2},\cdots,U_{k}).
\]

In both cases, the base change map $\boldsymbol{HFL'}^{-}(S^{3},L)\to\widehat{HFL'}(S^{3},L)$
is the $\mathbb{F}\llbracket U_{1}^{1/2},\cdots,U_{k}^{1/2}\rrbracket$-linear
map that maps $1$ to $1$; in particular, it is surjective.
\end{prop}

\begin{proof}
Iterate Proposition \ref{prop:free-stabilization}.
\end{proof}
Since $\boldsymbol{HFL'}^{-}(S^{3},L)\to\widehat{HFL'}(S^{3},L)$
is surjective and agrees with the corresponding map in Khovanov homology
for the two cases of Proposition \ref{prop:Let--be}, to show that
the planar band maps in the various versions of unoriented link Floer
homology agree with various versions of Khovanov homology, it is sufficient
to prove this for $\boldsymbol{HFL'}^{-}$. We work with $\boldsymbol{HFL'}^{-}$
throughout the remainder of this section.

Let us compute non-orientable band maps between unlinks: they agree
with Khovanov homology as well.
\begin{prop}
\label{prop:Non-orientable-hfinf}Non-orientable band maps between
balled links in $S^{3}$ are zero on $\boldsymbol{HFL'}^{\infty}$.
\end{prop}

\begin{proof}
Let $B:L_{c}\to L_{a}$ be a non-orientable band between two links
with $k$ components. We use the infinity version of the unoriented
skein exact triangle: complete $L_{c}\xrightarrow{B}L_{a}$ to an
unoriented skein triple; let $L_{b}$ be the third link. Then $L_{b}$
has $k+1$ components. As $\mathbb{F}\llbracket U\rrbracket[U^{-1}]$-vector
spaces, $\boldsymbol{HFL'}^{\infty}(S^{3},L_{c})$ and $\boldsymbol{HFL'}^{\infty}(S^{3},L_{a})$
have rank $2^{k}$, and $\boldsymbol{HFL'}^{\infty}(S^{3},L_{b})$
has rank $2^{k+1}$ by Proposition \ref{prop:infinity}. Hence, the
band map 
\[
\boldsymbol{HFL'}^{\infty}(S^{3},L_{c})\to\boldsymbol{HFL'}^{\infty}(S^{3},L_{a})
\]
is zero. 
\end{proof}
\begin{cor}
\label{cor:Non-orientable-minus}Non-orientable band maps between
balled unlinks in $S^{3}$ are zero on $\boldsymbol{HFL'}^{-}$.
\end{cor}

\begin{proof}
This follows from Proposition \ref{prop:Non-orientable-hfinf} together
with that for unlinks $UL$, $\boldsymbol{HFL'}^{-}(S^{3},UL)\to\boldsymbol{HFL'}^{\infty}(S^{3},UL)$
is injective.
\end{proof}
Let $B:L\to L'$ be a planar band between two planar balled links,
as in the start of this section. Then, we can complete $L\xrightarrow{B}L'$
to an unoriented skein triple $L_{a},L_{b},L_{c}$, where $(L,L')=(L_{a},L_{b})$
if $B$ is a split band, and $(L,L')=(L_{b},L_{c})$ if $B$ is a
merge band. Without loss of generality, assume that the coefficient
ring for $L_{b}$ is $\mathbb{F}\llbracket U_{1}^{1/2},U_{2}^{1/2},U_{3}^{k_{3}},\cdots,U_{n}^{k_{n}}\rrbracket$,
that the band intersects the link components of $L_{b}$ with link
baseballs that correspond to $U_{1}$ and $U_{2}$, and that the baseball
that corresponds to $U_{2}$ becomes a free baseball for the links
$L_{a}$ and $L_{c}$.

Let $S=\mathbb{F}\llbracket U_{1}^{1/2},U_{2},U_{3}^{k_{3}},\cdots,U_{n}^{k_{n}},U\rrbracket/(U_{1}-U,\cdots,U_{k}-U)$.
Then let us identify $\boldsymbol{HFL'}^{-}(S^{3},L_{a})$ and $\boldsymbol{HFL'}^{-}(S^{3},L_{c})$
with $S$ as $S$-modules, and consider $\boldsymbol{HFL'}^{-}(S^{3},L_{b})$
as a free $S$-module freely generated by $1,U_{2}^{1/2}$.
\begin{prop}
\label{prop:band-maps-unlink}The merge band map for $L_{a}\to L_{b}$
is the $S$-linear map given by $1\mapsto U_{1}^{1/2}+U_{2}^{1/2}$,
and the split band map for $L_{b}\to L_{c}$ is the $S$-linear map
given by $1\mapsto1,\ U_{2}^{1/2}\mapsto U_{1}^{1/2}$.
\end{prop}

\begin{proof}
We use Theorem \ref{thm:unoriented-exact-triangle} for the unoriented
skein triple $(L_{a},L_{b},L_{c})$ together with that the band map
$\boldsymbol{HFL'}^{-}(S^{3},L_{c})\to\boldsymbol{HFL'}^{-}(S^{3},L_{a})$
is zero by Corollary \ref{cor:Non-orientable-minus}. Recall from
Subsection \ref{subsec:spinc-structures-and-grading} that the band
maps preserve the relative homological $\mathbb{Z}$-grading.

The merge band map $F_{m}$ is surjective. Hence, by Proposition \ref{prop:actions-commute},
we have $F_{m}(1)=1$, and so $F_{m}(U_{2}^{1/2})=U_{1}^{1/2}F_{m}(1)=U_{1}^{1/2}$.

The split band map $F_{s}$ is an isomorphism 
\[
\boldsymbol{HFL'}^{-}(S^{3},L_{a})\to\ker F_{m}=(U_{1}^{1/2}+U_{2}^{1/2})S\le\boldsymbol{HFL'}^{-}(S^{3},L_{b}),
\]
and so $F_{s}(1)=U_{1}^{1/2}+U_{2}^{1/2}$.
\end{proof}

\section{\label{sec:The-plan}The plan for proving Theorem \ref{thm:sym2-t2}}

We will first establish three pairs of maps which are morally quasi-inverses
of each other, and we will then use this to prove Theorem \ref{thm:sym2-t2}
in Section \ref{sec:composition}. We consider the giant weakly admissible,
multi-Heegaard diagram given by the three diagrams in Figure \ref{fig:genus1-diagrams},
where the underlying pointed surface is $(\mathbb{T}^{2},\{v,w,z\})$
and the attaching curves are $\boldsymbol{\beta}_{a},\boldsymbol{\beta}_{b},\boldsymbol{\beta}_{0},\boldsymbol{\beta}_{2},\widetilde{\boldsymbol{\beta}_{0}},\widetilde{\boldsymbol{\beta}_{2}},\boldsymbol{\beta}_{c}$,
and their standard translates. Note that $\boldsymbol{\beta}_{\infty}=\boldsymbol{\beta}_{c}$.
The attaching curve $\boldsymbol{\beta}_{\infty}$ has a nontrivial
$\mathbb{F}\llbracket U_{1},U_{2}^{1/2}\rrbracket$-local system $E$,
given by the oriented arc $G$. Also note that it is not necessary
to wind any of the curves to put all of the figures into one weakly
admissible diagram.

\begin{figure}[h]
\begin{centering}
\includegraphics{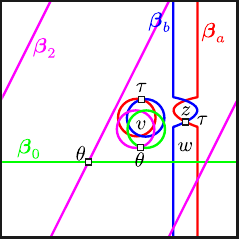}\qquad{}\includegraphics{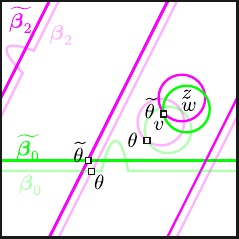}\qquad{}\includegraphics{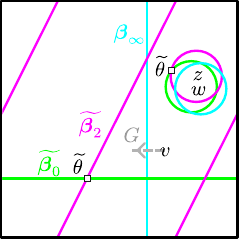}
\par\end{centering}
\caption{\label{fig:genus1-diagrams}The genus $1$ diagrams for Claims \ref{claim:step-1},
\ref{claim:step-2}, and \ref{claim:step-3}. Unless otherwise specified,
work over $\mathbb{F}\llbracket U_{1},U_{2}^{1/2}\rrbracket$ and
assign the weight $U_{1}$ to $v$, and $U_{2}^{1/2}$ to $w$ and
$z$.}
\end{figure}

\begin{claim}
\label{claim:step-1}Consider the cycles $\tau\in\boldsymbol{CF}^{-}(\boldsymbol{\beta}_{a},\boldsymbol{\beta}_{b})$
and $\theta\in\boldsymbol{CF}^{-}(\boldsymbol{\beta}_{0},\boldsymbol{\beta}_{2})$.
The mapping cones $\left(\boldsymbol{\beta}_{a}\xrightarrow{\tau}\boldsymbol{\beta}_{b}\right)$
and $\left(\boldsymbol{\beta}_{0}\xrightarrow{\theta}\boldsymbol{\beta}_{2}\right)$
are morally quasi-isomorphic over $\mathbb{F}\llbracket U_{1},U_{2}^{1/2}\rrbracket$,
if the almost complex structure is pinched along the free stabilization
region.
\end{claim}

\begin{proof}
We show this in Section \ref{sec:proof-step1}: we define maps $\underline{\zeta}$,
$\underline{\xi}$ in a destabilized diagram; $\underline{S(\zeta)}$
and $\underline{S(\xi)}$ are the wanted morally quasi-inverses.
\end{proof}
\begin{claim}
\label{claim:step-2}The following dashed maps between the mapping
cones $\left(\boldsymbol{\beta}_{0}\xrightarrow{\theta}\boldsymbol{\beta}_{2}\right)$
and $\left(\widetilde{\boldsymbol{\beta}_{0}}\xrightarrow{\widetilde{\theta}}\widetilde{\boldsymbol{\beta}_{2}}\right)$
are morally quasi-inverses of each other, where $\Theta_{0}^{+}$
and $\Theta_{2}^{+}$ are the top homological grading generators.
Note that the curves $\widetilde{\boldsymbol{\beta}_{0}}$ and $\widetilde{\boldsymbol{\beta}_{2}}$
are obtained from $\boldsymbol{\beta}_{0}$ and $\boldsymbol{\beta}_{2}$,
respectively, by two handleslides.% https://q.uiver.app/#q=WzAsOCxbMCwwLCJcXGJvbGRzeW1ib2x7XFxiZXRhfV8wIl0sWzEsMCwiXFxib2xkc3ltYm9se1xcYmV0YX1fMiJdLFswLDEsIlxcd2lkZXRpbGRle1xcYm9sZHN5bWJvbHtcXGJldGF9XzB9Il0sWzEsMSwiXFx3aWRldGlsZGV7XFxib2xkc3ltYm9se1xcYmV0YX1fMiB9Il0sWzIsMCwiXFx3aWRldGlsZGV7XFxib2xkc3ltYm9se1xcYmV0YX1fMCB9Il0sWzMsMCwiXFx3aWRldGlsZGV7XFxib2xkc3ltYm9se1xcYmV0YX1fMiB9Il0sWzIsMSwiXFxib2xkc3ltYm9se1xcYmV0YX1fMCJdLFszLDEsIlxcYm9sZHN5bWJvbHtcXGJldGF9XzIiXSxbMCwxLCJcXHRoZXRhIl0sWzIsMywiXFx3aWRldGlsZGV7XFx0aGV0YX0iXSxbMSwzLCJcXFRoZXRhXzIgXisiLDIseyJzdHlsZSI6eyJib2R5Ijp7Im5hbWUiOiJkYXNoZWQifX19XSxbMCwyLCJcXFRoZXRhXzAgXisiLDIseyJzdHlsZSI6eyJib2R5Ijp7Im5hbWUiOiJkYXNoZWQifX19XSxbNCw2LCJcXFRoZXRhXzAgXisiLDIseyJzdHlsZSI6eyJib2R5Ijp7Im5hbWUiOiJkYXNoZWQifX19XSxbNSw3LCJcXFRoZXRhXzIgXisiLDIseyJzdHlsZSI6eyJib2R5Ijp7Im5hbWUiOiJkYXNoZWQifX19XSxbNCw1LCJcXHdpZGV0aWxkZXtcXHRoZXRhfSJdLFs2LDcsIlxcdGhldGEiXV0=
\[\begin{tikzcd}[sep=large]
	{\boldsymbol{\beta}_0} & {\boldsymbol{\beta}_2} & {\widetilde{\boldsymbol{\beta}_0 }} & {\widetilde{\boldsymbol{\beta}_2 }} \\
	{\widetilde{\boldsymbol{\beta}_0}} & {\widetilde{\boldsymbol{\beta}_2 }} & {\boldsymbol{\beta}_0} & {\boldsymbol{\beta}_2}
	\arrow["\theta", from=1-1, to=1-2]
	\arrow["{\Theta_0 ^+}"', dashed, from=1-1, to=2-1]
	\arrow["{\Theta_2 ^+}"', dashed, from=1-2, to=2-2]
	\arrow["{\widetilde{\theta}}", from=1-3, to=1-4]
	\arrow["{\Theta_0 ^+}"', dashed, from=1-3, to=2-3]
	\arrow["{\Theta_2 ^+}"', dashed, from=1-4, to=2-4]
	\arrow["{\widetilde{\theta}}", from=2-1, to=2-2]
	\arrow["\theta", from=2-3, to=2-4]
\end{tikzcd}\]
\end{claim}

\begin{proof}
Let us show that the dashed maps are cycles. First, the maps $\theta,\widetilde{\theta},\Theta_{0}^{+},\Theta_{2}^{+}$
are cycles, and note that $\boldsymbol{\beta}_{i}$ and $\widetilde{\boldsymbol{\beta}_{i}}$
are related by handleslides and $\Theta_{i}^{+}$ is the top homological
grading generator. Since $\mu_{2}(\Theta_{0}^{+},\widetilde{\theta})=\mu_{2}(\theta,\Theta_{2}^{+})$
on homology, they must be equal on the chain level since they are
in the top homological grading ($(\boldsymbol{\beta}_{0},\widetilde{\boldsymbol{\beta}_{0}},\widetilde{\boldsymbol{\beta}_{2}})$
and $(\boldsymbol{\beta}_{0},\boldsymbol{\beta}_{2},\widetilde{\boldsymbol{\beta}_{2}})$
are homologically $\mathbb{Z}$-gradable). Similarly, the dashed map
on the right is a cycle. Their compositions satisfy the conditions
of Lemma \ref{lem:theta-quasi-iso}, with respect to the filtrations
$\{\boldsymbol{\beta}_{2}\}\subset\{\boldsymbol{\beta}_{0},\boldsymbol{\beta}_{2}\}$
and $\{\widetilde{\boldsymbol{\beta}_{2}}\}\subset\{\widetilde{\boldsymbol{\beta}_{0}},\widetilde{\boldsymbol{\beta}_{2}}\}$.
\end{proof}
\begin{claim}
\label{claim:step-3}The mapping cones $\left(\widetilde{\boldsymbol{\beta}_{0}}\xrightarrow{\widetilde{\theta}}\widetilde{\boldsymbol{\beta}_{2}}\right)$
and $\boldsymbol{\beta}_{\infty}^{E}$ are morally quasi-isomorphic
over $\mathbb{F}\llbracket U_{1},U_{2}^{1/2}\rrbracket$, if the almost
complex structure is pinched along the free stabilization region.
\end{claim}

\begin{proof}
We show this in Section \ref{sec:proof-step3}: we define maps $e_{0}\underline{\xi}$,
$e_{1}^{\ast}\underline{\zeta}$ in a destabilized diagram (the names
are unfortunately the same as the maps for Claim \ref{claim:step-1});
$\underline{S(e_{0}\xi)}$ and $\underline{S(e_{1}^{\ast}\zeta)}$
are the wanted morally quasi-inverses.
\end{proof}
\begin{rem}
Claims \ref{claim:step-2} and \ref{claim:step-3} hold over the ring
$\mathbb{F}\llbracket V,W,Z\rrbracket$, where we assign weights $V,W,Z$
to $v,w,z$, respectively. However, we need to assign the same weights
to $w$ and $z$ for Claim \ref{claim:step-1}.
\end{rem}

\begin{rem}
For Theorem \ref{thm:2surgery}, we only need to additionally check
that $\mu_{2}(e_{0}\xi,e_{1}^{\ast}\zeta)=0$ in addition to Claim
\ref{claim:step-3}.
\end{rem}

\begin{rem}
\label{rem:explicit-moral-homotopy-equiv}In fact, the above three
pairs of twisted complexes are morally homotopy equivalent. For Claim
\ref{claim:step-1}, $u_{2}^{-1}\underline{S(\zeta)}$ and $\underline{S(\xi)}$,
where $u_{2}=\sum_{n=1}^{\infty}U_{2}^{n^{2}-n}$, are morally homotopy
inverses. For Claim \ref{claim:step-2}, the given maps are morally
homotopy inverses: one way to show this is to check that $(\mathbb{T}^{2},\boldsymbol{\beta}_{0},\boldsymbol{\beta}_{2},\widetilde{\boldsymbol{\beta}_{0}},\widetilde{\boldsymbol{\beta}_{2}},\boldsymbol{\beta}_{0}',\boldsymbol{\beta}_{2}',\{v,w,z\})$
is homologically $\mathbb{Z}$-gradable. For Claim \ref{claim:step-3},
$\underline{S(e_{0}\xi)}$ and $u_{1}^{-1}\underline{S(e_{1}^{\ast}\zeta)}$,
where $u_{1}=\sum_{n=1}^{\infty}U_{1}^{n^{2}-n}$, are morally homotopy
inverses.
\end{rem}

We believe that finishing off the proof from Claims \ref{claim:step-1},
\ref{claim:step-2}, and \ref{claim:step-3} is standard; we spell
out one way to prove Theorem \ref{thm:sym2-t2} in Section \ref{sec:composition}.

\section{\label{sec:proof-step1}Local computation for Claim \ref{claim:step-1}}

\begin{figure}[h]
\begin{centering}
\includegraphics[scale=1.5]{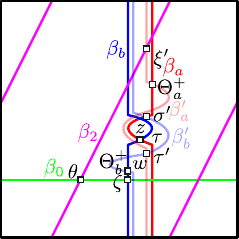}\enskip{}\includegraphics[scale=1.5]{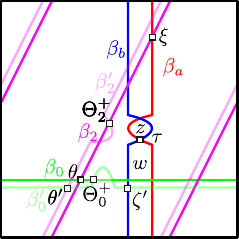}
\par\end{centering}
\caption{\label{fig:ab02ab-t2}Heegaard diagrams $(\mathbb{T}^{2},\beta_{a},\beta_{b},\beta_{0},\beta_{2},\beta_{a}',\beta_{b}',\{w,z\})$
and $(\mathbb{T}^{2},\beta_{0},\beta_{2},\beta_{a},\beta_{b},\beta_{0}',\beta_{2}',\{w,z\})$}
\end{figure}

Consider the Heegaard diagram $(\mathbb{T}^{2},\beta_{a},\beta_{b},\beta_{0},\beta_{2},\{w,z\})$
(together with standard translates) given by Figure \ref{fig:ab02ab-t2},
and work over the ring $\mathbb{F}\llbracket U^{1/2}\rrbracket$ and
assign the weight $U^{1/2}$ to both basepoints $z$ and $w$. Consider
the following vertical maps $\underline{\zeta}$ and $\underline{\xi}$
between the twisted complexes $\beta_{a}\xrightarrow{\tau}\beta_{b}$
and $\beta_{0}\xrightarrow{\theta}\beta_{2}$.% https://q.uiver.app/#q=WzAsOCxbMCwwLCJcXGJldGFfYSJdLFsxLDAsIlxcYmV0YV9iIl0sWzEsMSwiXFxiZXRhXzAiXSxbMiwxLCJcXGJldGFfMiJdLFszLDAsIlxcYmV0YV8wIl0sWzQsMCwiXFxiZXRhXzIiXSxbNCwxLCJcXGJldGFfYSJdLFs1LDEsIlxcYmV0YV9iIl0sWzAsMSwiXFx0YXUiXSxbMiwzLCJcXHRoZXRhIl0sWzEsMiwiXFx6ZXRhIiwyLHsic3R5bGUiOnsiYm9keSI6eyJuYW1lIjoiZGFzaGVkIn19fV0sWzYsNywiXFx0YXUiXSxbNSw2LCJcXHhpIiwyLHsic3R5bGUiOnsiYm9keSI6eyJuYW1lIjoiZGFzaGVkIn19fV0sWzQsNSwiXFx0aGV0YSJdXQ==
\[\begin{tikzcd}[sep=large]
	{\beta_a} & {\beta_b} && {\beta_0} & {\beta_2} \\
	& {\beta_0} & {\beta_2} && {\beta_a} & {\beta_b}
	\arrow["\tau", from=1-1, to=1-2]
	\arrow["\zeta"', dashed, from=1-2, to=2-2]
	\arrow["\theta", from=1-4, to=1-5]
	\arrow["\xi"', dashed, from=1-5, to=2-5]
	\arrow["\theta", from=2-2, to=2-3]
	\arrow["\tau", from=2-5, to=2-6]
\end{tikzcd}\]We claim that they are cycles and that $\mu_{2}(\underline{\zeta},\underline{\xi'})$
and $\mu_{2}(\underline{\xi},\underline{\zeta'})$ are $\Theta^{+}$
modulo $U^{1/2}$. Recall that for instance, for $\mu_{2}(\underline{\zeta},\underline{\xi'})$,
this means that the composition of the vertical maps on the left hand
side is the right hand side modulo $U^{1/2}$.% https://q.uiver.app/#q=WzAsMTAsWzAsMCwiXFxiZXRhX2EiXSxbMSwwLCJcXGJldGFfYiJdLFsxLDEsIlxcYmV0YV8wIl0sWzIsMSwiXFxiZXRhXzIiXSxbNCwxLCJcXGJldGFfYSciXSxbNSwxLCJcXGJldGFfYiAnIl0sWzIsMiwiXFxiZXRhIF9hICciXSxbMywyLCJcXGJldGEgX2IgJyJdLFs0LDAsIlxcYmV0YSBfYSAiXSxbNSwwLCJcXGJldGEgX2IgIl0sWzAsMSwiXFx0YXUiXSxbMiwzLCJcXHRoZXRhIl0sWzEsMiwiXFx6ZXRhIiwyLHsic3R5bGUiOnsiYm9keSI6eyJuYW1lIjoiZGFzaGVkIn19fV0sWzQsNSwiXFx0YXUgJyJdLFs2LDcsIlxcdGF1ICciXSxbMyw2LCJcXHhpICciLDIseyJzdHlsZSI6eyJib2R5Ijp7Im5hbWUiOiJkYXNoZWQifX19XSxbOCw5LCJcXHRhdSAiXSxbOCw0LCJ1XFxUaGV0YSBeKyIsMix7InN0eWxlIjp7ImJvZHkiOnsibmFtZSI6ImRhc2hlZCJ9fX1dLFs5LDUsInVcXFRoZXRhIF4rIiwyLHsic3R5bGUiOnsiYm9keSI6eyJuYW1lIjoiZGFzaGVkIn19fV1d
\[\begin{tikzcd}[sep=large]
	{\beta_a} & {\beta_b} &&& {\beta _a } & {\beta _b } \\
	& {\beta_0} & {\beta_2} && {\beta_a'} & {\beta_b '} \\
	&& {\beta _a '} & {\beta _b '}
	\arrow["\tau", from=1-1, to=1-2]
	\arrow["\zeta"', dashed, from=1-2, to=2-2]
	\arrow["{\tau }", from=1-5, to=1-6]
	\arrow["{\Theta_a ^+}"', dashed, from=1-5, to=2-5]
	\arrow["{\Theta_b ^+}"', dashed, from=1-6, to=2-6]
	\arrow["\theta", from=2-2, to=2-3]
	\arrow["{\xi '}"', dashed, from=2-3, to=3-3]
	\arrow["{\tau '}", from=2-5, to=2-6]
	\arrow["{\tau '}", from=3-3, to=3-4]
\end{tikzcd}\]

Showing the above is equivalent to the following:
\begin{itemize}
\item $\beta_{a}\xrightarrow{\tau}\beta_{b}$ and $\beta_{0}\xrightarrow{\theta}\beta_{2}$
are twisted complexes: $\mu_{1}(\tau)=0$, $\mu(\theta)=0$
\item $\mu_{1}(\underline{\zeta})=0$: $\mu_{1}(\zeta)=0$, $\mu_{2}(\tau,\zeta)=0$,
$\mu_{2}(\zeta,\theta)=0$, $\mu_{3}(\tau,\zeta,\theta)=0$
\item $\mu_{1}(\underline{\xi})=0$: $\mu_{1}(\xi)=0$, $\mu_{2}(\theta,\xi)=0$,
$\mu_{2}(\xi,\tau)=0$, $\mu_{3}(\theta,\xi,\tau)=0$
\item $\mu_{2}(\underline{\zeta},\underline{\xi'})$: $\mu_{3}(\zeta,\theta,\xi')=0$,
$\mu_{4}(\tau,\zeta,\theta,\xi')=\Theta_{a}^{+}$, $\mu_{4}(\zeta,\theta,\xi',\tau')=\Theta_{b}^{+}$,
$\mu_{5}(\tau,\zeta,\theta,\xi',\tau')=0$ modulo $U^{1/2}$
\item $\mu_{2}(\underline{\xi},\underline{\zeta'})$: $\mu_{3}(\xi,\tau,\zeta')=0$,
$\mu_{4}(\theta,\xi,\tau,\zeta')=\Theta_{0}^{+}$, $\mu_{4}(\xi,\tau,\zeta',\theta')=\Theta_{2}^{+}$,
$\mu_{5}(\theta,\xi,\tau,\zeta',\theta')=0$ modulo $U^{1/2}$
\end{itemize}
\begin{rem}
In fact, we can explicitly compute $\mu_{2}(\underline{\zeta},\underline{\xi'})$
and $\mu_{2}(\underline{\xi},\underline{\zeta'})$. All the $\mu_{3}$'s
and $\mu_{5}$'s vanish, and the $\mu_{4}$'s are $u\Theta^{+}$ for
the unit $u=\sum_{n=1}^{\infty}U^{n^{2}-n}$.
\end{rem}

Some of these are easy to verify. Indeed, the $\mu_{1}$'s, $\mu_{2}(\tau,\zeta)$,
and $\mu_{2}(\xi,\tau)$ vanish because there are no relevant bigons
and triangles. The hexagon count $\mu_{5}$ vanishes modulo $U^{1/2}$
because there are no relevant hexagons that do not contain any basepoints.

To verify the rest, we work in the universal cover of $\mathbb{T}^{2}$.
It will be useful to have names for some right triangles.
\begin{defn}
\label{def:triangle-label}Let $\beta_{\infty}$ be the straight,
vertical line that is in the middle of $\beta_{a}$ and $\beta_{b}$.
We abuse notation and label the intersection points as follows: $\beta_{\infty}\cap\beta_{0}=\{\zeta\}$,
$\beta_{2}\cap\beta_{\infty}=\{\xi\}$. Also, let $\beta_{0}\cap\beta_{2}=\{\theta_{0},\theta_{1}\}$,
where $\theta_{1}=\theta$. See Figure \ref{fig:labelled-triangles}.
\begin{itemize}
\item Let triangles $T_{n}^{+}$ and $S_{n}^{+}$, respectively, for $n\in\mathbb{Z}_{\ge1}$
be the $n$th right-side up triangle that has vertices $\theta_{1},\xi,\zeta$
and $\theta_{0},\xi,\zeta$, respectively.
\item Let triangles $T_{n}^{-}$ and $S_{n}^{-}$, respectively, for $n\in\mathbb{Z}_{\ge1}$
be the $n$th upside down triangle that has vertices $\theta_{1},\xi,\zeta$
and $\theta_{0},\xi,\zeta$, respectively.
\end{itemize}
\end{defn}

\begin{figure}
\begin{centering}
\includegraphics[scale=0.65]{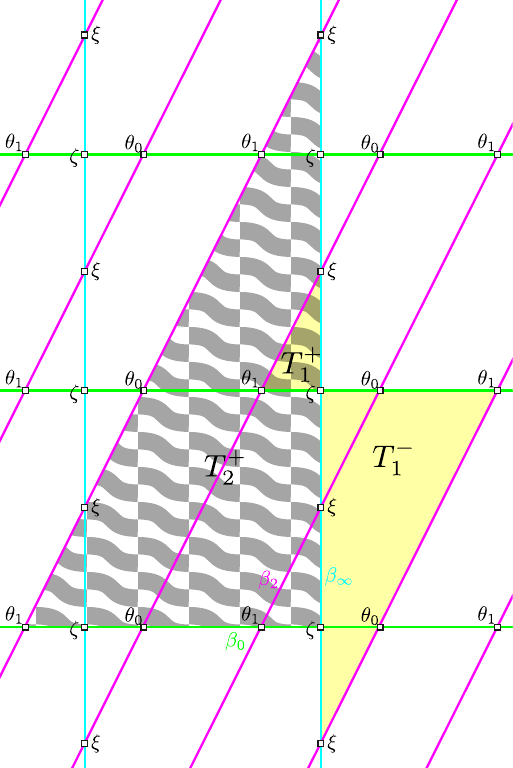}\enskip{}\includegraphics[scale=0.9]{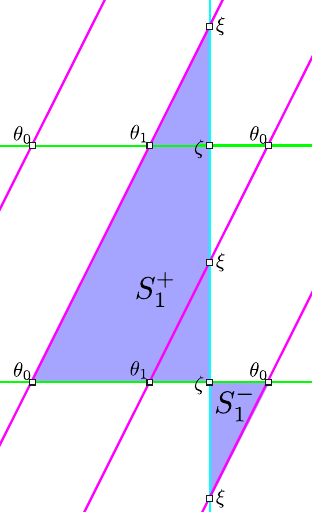}
\par\end{centering}
\caption{\label{fig:labelled-triangles}Right triangles and their names}
\end{figure}

\subsection{\label{subsec:mu2-vanish}The \texorpdfstring{$\mu_2$}{mu2}'s vanish}

For $\mu_{2}(\zeta,\theta)$ and $\mu_{2}(\theta,\xi)$, there are
two families of triangles that cancel each other. These are related
by ``rotation by $\pi$.'' In fact, all of the quadrilaterals (and
hexagons) that cancel in pairs are also related by ``rotation by
$\pi$'' in some sense.

We consider $\mu_{2}(\zeta,\theta)$. The two smallest triangles are
drawn in the left hand side of Figure \ref{fig:small-triangles-quadrilaterals}.
These have one basepoint, and hence have weight $U^{1/2}$. The next
smallest triangles are drawn in Figure \ref{fig:large-triangles},
and they have weight $U^{5/2}$, as they contain five basepoints.

We view these triangles as small perturbations of $T_{n}^{\pm}$.
(We think of $\beta_{a}$ and $\beta_{b}$ as being close to each
other: the region between them is ``small.'') Such triangle $T_{n}^{+}$
or $T_{n}^{-}$ uniquely determines the corresponding triangle that
contributes to $\mu_{2}(\zeta,\theta)$. For each $n$, there are
two of them, one from $T_{n}^{+}$ and one from $T_{n}^{-}$, and
they have weight $U^{n^{2}-n+\frac{1}{2}}$. Hence they all cancel.

\begin{figure}[h]
\begin{centering}
\includegraphics{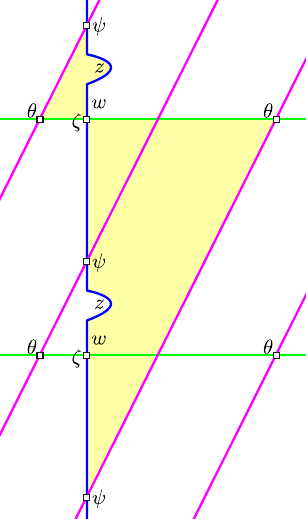}\includegraphics{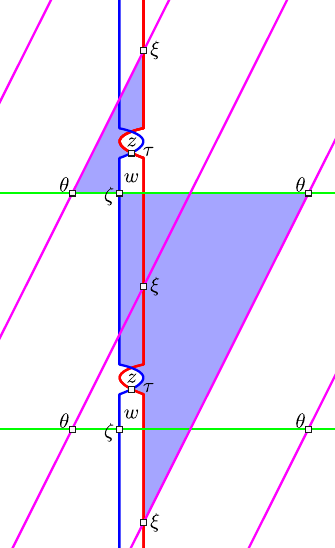}
\par\end{centering}
\caption{\label{fig:small-triangles-quadrilaterals}Left: two small triangles
with weight $U^{1/2}$; Right: Two small quadrilaterals with weight
$1$}
\end{figure}

\subsection{\label{subsec:mu3-vanish}The \texorpdfstring{$\mu_3$}{mu3}'s vanish}

We will write $\zeta$ instead of $\zeta'$ and $\beta_{0}$ instead
of $\beta_{0}'$, etc. since distinguishing them in this subsection
is meaningless.

Let us show that $\mu_{3}(\tau,\zeta,\theta)$, $\mu_{3}(\zeta,\theta,\xi)$,
and $\mu_{3}(\theta,\xi,\tau)$ vanish. Any quadrilateral that contributes
in these quadrilateral count has vertices $\xi,\tau,\zeta,\theta$.
There are two families of quadrilaterals with vertices $\xi,\tau,\zeta,\theta$
that cancel each other. The two smallest quadrilaterals are drawn
in the right hand side of Figure \ref{fig:small-triangles-quadrilaterals}.
These do not have any basepoints, and hence have weight $1$. The
next smallest are drawn in Figure \ref{fig:large-quadrilaterals},
and they have weight $U^{2}$.

We view these quadrilaterals as small perturbations of $T_{n}^{\pm}$'s
as in Subsection \ref{subsec:mu2-vanish}. Given a triangle $T_{n}^{\pm}$,
specifying such a quadrilateral is equivalent to choosing a lift of
$\tau$ as vertices of the quadrilateral. For each of $T_{n}^{+}$
and $T_{n}^{-}$, there is $2n-1$ such choices, and there are exactly
$2n^{2}-2n$ basepoints for each such quadrilateral. Hence they all
cancel.

Computing $\mu_{3}(\xi,\tau,\zeta)$ is slightly more complicated
since the quadrilaterals do not need to have $\theta$ as a vertex,
but we only need to compute $\mu_{3}(\xi,\tau,\zeta)$ modulo $U^{1/2}$.
All the relevant quadrilaterals that do not have any basepoints have
$\theta$ as a vertex, and there are exactly two of them, which cancel.

\subsection{The \texorpdfstring{$\mu_4$}{mu4}'s are the identity modulo $U^{1/2}$}

For the $\mu_{4}$'s, in all the cases, the Maslov index $-2$ pentagons
are small perturbations of the above quadrilaterals (not every quadrilateral
has a small perturbation that contributes to $\mu_{4}$). In each
case, there is exactly one pentagon with vertices $\tau,\zeta,\theta,\xi$
and $\Theta_{i}^{+}$ ($i=a,b,0,2$) that does not contain any basepoints.
This pentagon is a small perturbation of the top quadrilateral in
the right hand side of Figure \ref{fig:small-triangles-quadrilaterals}.
Two of these are drawn in Figure \ref{fig:small-pentagons}.
\begin{figure}[h]
\begin{centering}
\includegraphics{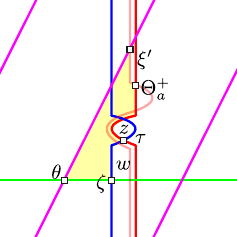}\enskip{}\includegraphics{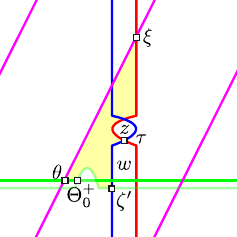}
\par\end{centering}
\caption{\label{fig:small-pentagons}Relevant pentagons for $\mu_{4}(\tau,\zeta,\theta,\xi')$
and $\mu_{4}(\theta,\xi,\tau,\zeta')$ modulo $U^{1/2}$. These are
small perturbations of the top quadrilateral of the right hand side
of Figure \ref{fig:small-triangles-quadrilaterals}.}
\end{figure}

\begin{figure}
\begin{centering}
\includegraphics[scale=0.7]{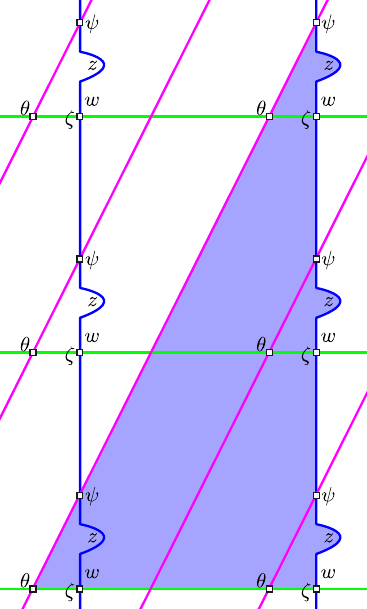}\enskip{}\includegraphics[scale=0.7]{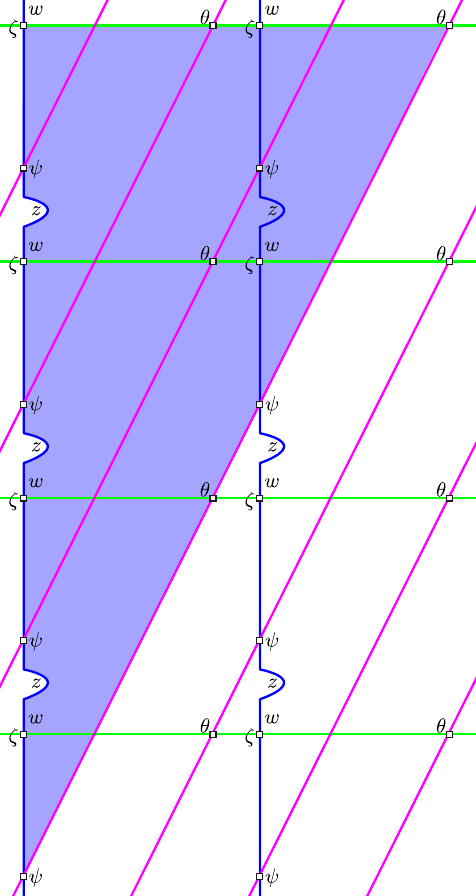}
\par\end{centering}
\caption{\label{fig:large-triangles}Two triangles with weight $U^{5/2}$}
\end{figure}
\begin{figure}
\begin{centering}
\includegraphics[scale=0.7]{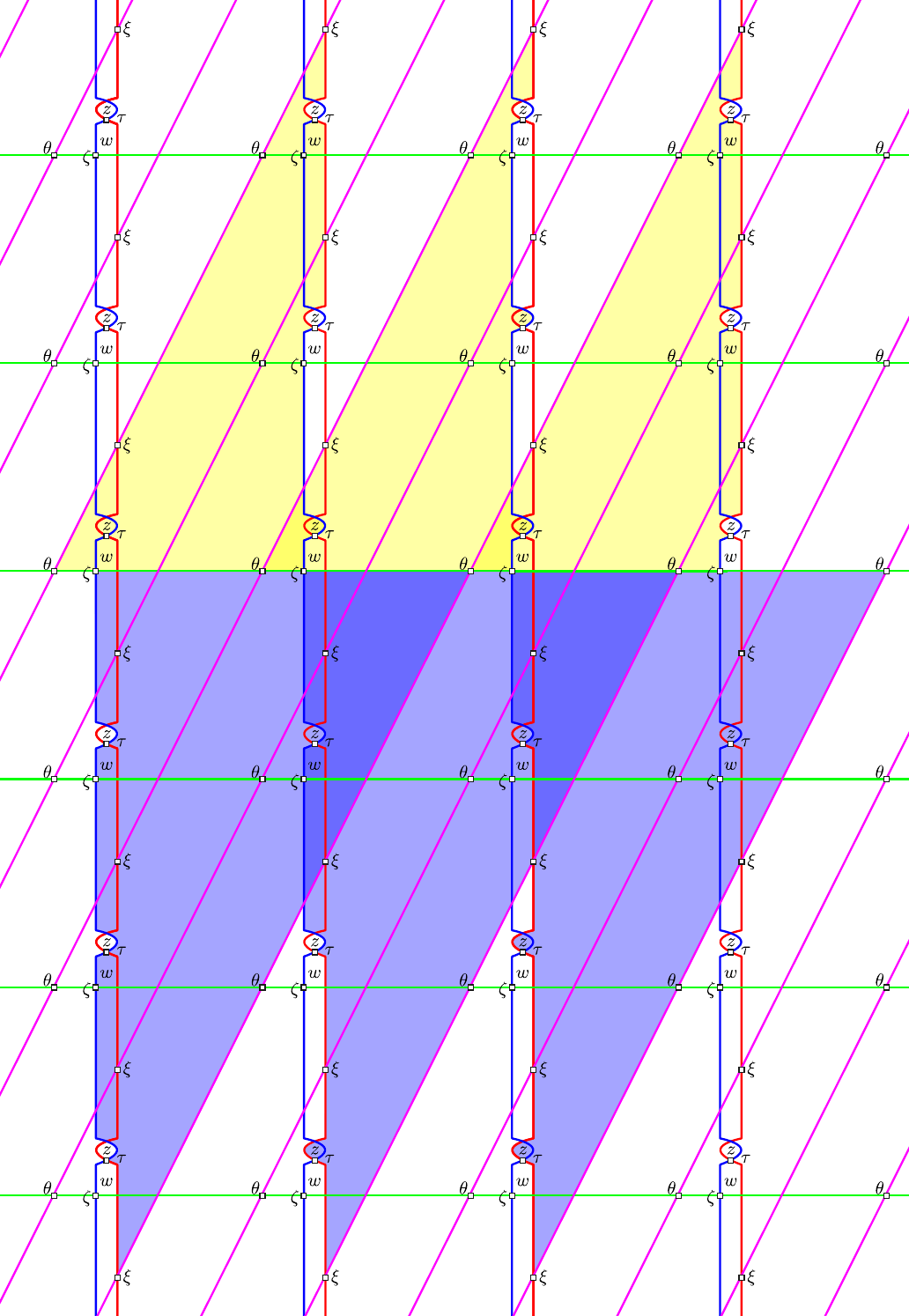}
\par\end{centering}
\caption{\label{fig:large-quadrilaterals}Six quadrilaterals with weight $U^{2}$.
The top quadrilaterals are small perturbations of $T_{2}^{+}$ and
the bottom ones are small perturbations of $T_{2}^{-}.$}
\end{figure}
\FloatBarrier

\section{\label{sec:proof-step3}Local computation for Claim \ref{claim:step-3}}

\begin{figure}
\begin{centering}
\includegraphics{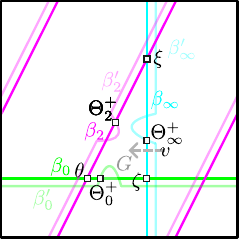}
\par\end{centering}
\caption{\label{fig:local-02inf}The genus $1$ Heegaard diagram $(\mathbb{T}^{2},\beta_{0},\beta_{2},\beta_{\infty},\beta_{0}',\beta_{2}',\beta_{\infty}',v)$}
\end{figure}

Consider the Heegaard diagram Figure \ref{fig:local-02inf}, and work
over the ring $\mathbb{F}\llbracket U\rrbracket$ and assign the weight
$U$ to $v$. The circle $\beta_{\infty}$ has a nontrivial local
system $E=e_{0}\mathbb{F}\llbracket U\rrbracket\oplus e_{1}\mathbb{F}\llbracket U\rrbracket$
whose monodromy is given by the oriented arc $G$.

Consider the following vertical maps $e_{0}\underline{\xi}$ and $e_{1}^{\ast}\underline{\zeta}$
between the twisted complexes $\beta_{0}\xrightarrow{\theta}\beta_{2}$
and $\beta_{\infty}^{E}$. % https://q.uiver.app/#q=WzAsNixbMCwwLCJcXGJldGFfMCJdLFsxLDAsIlxcYmV0YV8yIl0sWzEsMSwiXFxiZXRhX1xcaW5mdHkgXkUiXSxbMiwwLCJcXGJldGFfXFxpbmZ0eSBeRSJdLFsyLDEsIlxcYmV0YSBfMCAiXSxbMywxLCJcXGJldGEgXzIgIl0sWzAsMSwiXFx0aGV0YSAiXSxbNCw1LCJcXHRoZXRhICJdLFszLDQsImVfMSBeXFxhc3QgIFxcemV0YSAiLDAseyJzdHlsZSI6eyJib2R5Ijp7Im5hbWUiOiJkYXNoZWQifX19XSxbMSwyLCJlXzAgXFx4aSIsMix7InN0eWxlIjp7ImJvZHkiOnsibmFtZSI6ImRhc2hlZCJ9fX1dXQ==
\[\begin{tikzcd}[sep=large]
	{\beta_0} & {\beta_2} & {\beta_\infty ^E} \\
	& {\beta_\infty ^E} & {\beta _0 } & {\beta _2 }
	\arrow["{\theta }", from=1-1, to=1-2]
	\arrow["{e_0 \xi}"', dashed, from=1-2, to=2-2]
	\arrow["{e_1 ^\ast  \zeta }", dashed, from=1-3, to=2-3]
	\arrow["{\theta }", from=2-3, to=2-4]
\end{tikzcd}\]As in Section \ref{sec:proof-step1}, we will show that they are cycles
and that $\mu_{2}(e_{0}\underline{\xi},e_{1}^{\ast}\underline{\zeta'})$
and $\mu_{2}(e_{1}^{\ast}\underline{\zeta},e_{0}\underline{\xi'})$
are $\Theta^{+}$ modulo $U$.

As before, this is equivalent to showing the following.
\begin{itemize}
\item $\mu_{1}(\theta)=0$
\item $\mu_{1}(e_{0}\xi)=0$, $\mu_{2}(\theta,e_{0}\xi)=0$
\item $\mu_{1}(e_{1}^{\ast}\zeta)=0$, $\mu_{2}(e_{1}^{\ast}\zeta,\theta)=0$
\item $\mu_{2}(e_{0}\xi,e_{1}^{\ast}\zeta)=0$, $\mu_{3}(\theta,e_{0}\xi,e_{1}^{\ast}\zeta')=\Theta_{0}^{+}$,
$\mu_{3}(e_{0}\xi,e_{1}^{\ast}\zeta',\theta')=\Theta_{2}^{+}$, $\mu_{4}(\theta,e_{0}\xi,e_{1}^{\ast}\zeta',\theta')=0$
modulo $U$
\item $\mu_{3}(e_{1}^{\ast}\zeta,\theta,e_{0}\xi')={\rm Id}_{E}\Theta_{\infty}^{+}$
modulo $U$
\end{itemize}
We also need to check $\mu_{2}(e_{0}\xi,e_{1}^{\ast}\zeta)=0$ for
Theorem \ref{thm:2surgery}.
\begin{rem}
We can explicitly compute $\mu_{2}(e_{0}\underline{\xi},e_{1}^{\ast}\underline{\zeta'})$
and $\mu_{2}(e_{1}^{\ast}\underline{\zeta},e_{0}\underline{\xi'})$.
Indeed, we get $\mu_{2}(e_{0}\xi,e_{1}^{\ast}\zeta)=\mu_{4}(\theta,e_{0}\xi,e_{1}^{\ast}\zeta',\theta')=0$
and the $\mu_{3}$'s are $u\Theta^{+}$ where the $u$ is the same
as in Section \ref{sec:proof-step1}: $u=\sum_{n=1}^{\infty}U^{n^{2}-n}$.
\end{rem}

Similarly, the $\mu_{1}$'s vanish because there are no bigons, and
the $\mu_{4}$ vanish because there are no pentagons with Maslov index
$-2$.

\subsection{The \texorpdfstring{$\mu_2$}{mu2}'s vanish}

We evaluate the triangle counting maps. Recall Definition \ref{def:triangle-label}.

We compute $\mu_{2}(e_{0}\xi,e_{1}^{\ast}\zeta)$, which is the most
complicated case, since it is not a priori obvious that the triangles
$S_{n}^{+}$ and $S_{n}^{-}$ do not contribute. Recall that $\phi=e_{1}e_{0}^{\ast}+Ue_{0}e_{1}^{\ast}$.
\begin{itemize}
\item The triangle $T_{n}^{+}$ contributes 
\[
U^{(n-1)^{2}}e_{1}^{\ast}\phi^{2n-1}e_{0}\theta=U^{n^{2}-n}\theta;
\]
\item The triangle $T_{n}^{-}$ contributes 
\[
U^{(n-1)^{2}}e_{1}^{\ast}\left(U\phi^{-1}\right)^{2n-1}e_{0}\theta=U^{n^{2}-n}\theta;
\]
(note that $U\phi^{-1}=\phi$)
\item The triangle $S_{n}^{\pm}$ contributes some multiple of 
\[
e_{1}^{\ast}\phi^{2k}e_{0}\theta_{0}=0;
\]
\end{itemize}
and so they all cancel and we get $\mu_{2}(e_{0}\xi,e_{1}^{\ast}\zeta)=0$.

The rest vanish as well by direct computation.

\subsection{The \texorpdfstring{$\mu_3$}{mu3}'s are the identity}

\begin{figure}[h]
\begin{centering}
\includegraphics{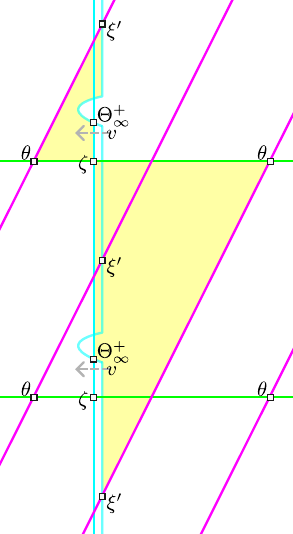}
\par\end{centering}
\caption{\label{fig:2surgery-quad}The two quadrilaterals highlighted in yellow
are the two quadrilaterals that contribute modulo $U$, and they add
up to ${\rm Id}_{E}\Theta_{\infty}^{+}$.}
\end{figure}

We compute $\mu_{3}(e_{1}^{\ast}\zeta,\theta,e_{0}\xi')$, which is
the most complicated case. If we work modulo $U$, then the two small
quadrilaterals highlighted in yellow in Figure \ref{fig:2surgery-quad}
are the only quadrilaterals that contribute. The top one outputs $e_{0}e_{1}^{\ast}\phi\Theta_{\infty}^{+}=e_{0}e_{0}^{\ast}\Theta_{\infty}^{+}$
and the bottom one outputs $U\phi^{-1}e_{0}e_{1}^{\ast}\Theta_{\infty}^{+}=e_{1}e_{1}^{\ast}\Theta_{\infty}^{+}$.
Hence, they add up to ${\rm Id}_{E}\Theta_{\infty}^{+}.$

For the other $\mu_{3}$'s, there is exactly one quadrilateral that
contributes, which gives $\Theta_{0}^{+}$ or $\Theta_{2}^{+}$.

\section{\label{sec:composition}Completing the proof of Theorem \ref{thm:sym2-t2}}

\subsection{\label{subsec:Gradings-1}A homological $\mathbb{Z}$-grading}

A useful observation that we will repeatedly use is that the following
Heegaard datum is homologically $\mathbb{Z}$-gradable. Compare \cite[Section 3]{MR2509750}
and \cite[Section 5.2]{MR2964628}.
\begin{prop}
\label{prop:diagram-graded}Any Heegaard diagram obtained from the
Heegaard diagram $(\mathbb{T}^{2},\boldsymbol{\beta}_{a},\boldsymbol{\beta}_{b},\boldsymbol{\beta}_{c},\{v,w,z\},G)$
given in Figure \ref{fig:drawing-genus1} by adding standard translates
of the attaching curves is homologically $\mathbb{Z}$-gradable.
\end{prop}

\begin{proof}
This follows from Lemma \ref{lem:cornerless-maslov-1}: the $\boldsymbol{\beta}_{i}$'s
are handlebody-equivalent, and the intersection points have $c_{1}=0$.
\end{proof}
Let the top homological $\mathbb{Z}$-grading elements of $\boldsymbol{CF}^{-}(\boldsymbol{\beta}_{a},\boldsymbol{\beta}_{b})$
and $\boldsymbol{CF}^{-}(\boldsymbol{\beta}_{b},\boldsymbol{\beta}_{c}^{E})$
have $\mathbb{Z}$-grading $0$. Then, the top $\mathbb{Z}$-grading
elements of $\boldsymbol{CF}^{-}(\boldsymbol{\beta}_{a},\boldsymbol{\beta}_{c}^{E})$,
$\boldsymbol{CF}^{-}(\boldsymbol{\beta}_{c}^{E},\boldsymbol{\beta}_{a}')$,
$\boldsymbol{CF}^{-}(\boldsymbol{\beta}_{c}^{E},\boldsymbol{\beta}_{b}')$,
$\boldsymbol{CF}^{-}(\boldsymbol{\beta}_{b},\boldsymbol{\beta}_{a}')$
have $\mathbb{Z}$-grading $0$, $-1$, $-1$, $-1$, respectively.

We will consider the twisted complex $\underline{\boldsymbol{\beta}_{ab}}:=\boldsymbol{\beta}_{a}\xrightarrow{\tau}\boldsymbol{\beta}_{b}$
in this graded sense as well, in which case we take $\underline{\boldsymbol{\beta}_{ab}}=\boldsymbol{\beta}_{a}[-1]\xrightarrow{\tau}\boldsymbol{\beta}_{b}$. 

\subsection{\label{subsec:obstruction-class}An Alexander $\mathbb{Z}/2$-splitting}

\begin{figure}[h]
\begin{centering}
\includegraphics{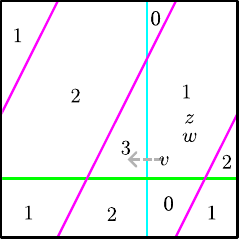}
\par\end{centering}
\caption{\label{fig:2surgery-cornerless}A cornerless two-chain ${\cal D}$
whose boundary has exactly one circle with slope $2$.}
\end{figure}

\begin{prop}
\label{prop:diagram-z/2-splitting}Any Heegaard diagram obtained from
the Heegaard diagram 
\[
(\mathbb{T}^{2},\boldsymbol{\beta}_{a},\boldsymbol{\beta}_{b},\boldsymbol{\beta}_{0},\boldsymbol{\beta}_{2},\widetilde{\boldsymbol{\beta}_{0}},\widetilde{\boldsymbol{\beta}_{2}},\boldsymbol{\beta}_{c}^{E},\{v,w,z\},G)
\]
given in Section \ref{sec:The-plan} by adding standard translates
of the attaching curves is Alexander $\mathbb{Z}/2$-splittable.
\end{prop}

\begin{proof}
Let ${\cal D}$ be a cornerless two-chain. We check that its total
multiplicity $P({\cal D})$ is even. Let ${\cal D}_{1}$ be the cornerless
two-chain in Figure \ref{fig:2surgery-cornerless}, whose boundary
has exactly one circle with slope $2$. We have $P({\cal D}_{1})=6$,
which is even. The cornerless two-chain ${\cal D}$ can be written
as the sum of (1) some multiple of ${\cal D}_{1}$, (2) a linear combination
${\cal D}_{2}$ of some number of small cornerless two-chains between
two circles which are standard translates of each other, and (3) a
cornerless two-chain ${\cal D}_{3}$ whose boundary lies in the union
of the attaching curves $\boldsymbol{\beta}_{a},\boldsymbol{\beta}_{b},\boldsymbol{\beta}_{c}$.
We have $P({\cal D}_{2})=0$, and we can directly check that $P({\cal D}_{3})$
is always even.
\end{proof}
We let all the generators that we have considered in Section \ref{sec:The-plan}
have Alexander $\mathbb{Z}/2$-grading $0$. This is indeed possible:
there are two things to check. One is that we do not get a contradiction
from the maps from Claims \ref{claim:step-2}, and the other is that
the $\Theta^{+}$'s have Alexander $\mathbb{Z}/2$-grading $0$. Consider
the intersection point $s\in\boldsymbol{\beta}_{2}\cap\widetilde{\boldsymbol{\beta}_{0}}$
that corresponds to $\theta$. Then, the former follows from that
${\rm gr}_{A}^{\mathbb{Z}/2}(\Theta_{0}^{+})={\rm gr}_{A}^{\mathbb{Z}/2}(\theta)+{\rm gr}_{A}^{\mathbb{Z}/2}(s)$
and ${\rm gr}_{A}^{\mathbb{Z}/2}(\widetilde{\theta})={\rm gr}_{A}^{\mathbb{Z}/2}(\Theta_{2}^{+})+{\rm gr}_{A}^{\mathbb{Z}/2}(s)$.
The latter also follows from direct computation.

\subsection{Another computation}

Let us briefly study the top homological $\mathbb{Z}$-grading summand
of $\boldsymbol{CF}^{-}(\underline{\boldsymbol{\beta}_{ab}},\underline{\boldsymbol{\beta}_{ab}'})$.
The top homological $\mathbb{Z}$-grading is $0$, and the generators
with homological $\mathbb{Z}$-grading $0$ are $\Theta_{a}^{+}$,
$\Theta_{b}^{+}$, $x$, and $y$ in Figure \ref{fig:betaabab}. (The
intersection points in the free-stabilizing region are not drawn.)
The Alexander $\mathbb{Z}/2$-grading of $x$ is $1$, and the rest
are $0$. One can check that, among linear combinations of $\Theta_{a}^{+}$,
$\Theta_{b}^{+}$, and $y$, the sum $\Theta_{a}^{+}+\Theta_{b}^{+}$
is the only nonzero cycle, regardless of the almost complex structure.

\begin{figure}[h]
\begin{centering}
\includegraphics{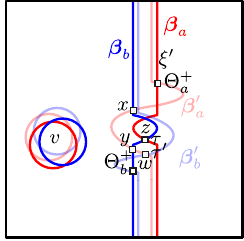}
\par\end{centering}
\caption{\label{fig:betaabab}A Heegaard diagram $(\mathbb{T}^{2},\boldsymbol{\beta}_{a},\boldsymbol{\beta}_{b},\boldsymbol{\beta}_{a}',\boldsymbol{\beta}_{b}',z,w,v)$}
\end{figure}

\subsection{Independence under variation of the almost complex structure}

We show that the statement of Theorem \ref{thm:sym2-t2} is insensitive
of the almost complex structure. Let us first check that the maps
are invariant. In all the cases, this will follow from that there
are no relevant domains with the right Maslov index, which is easy
to check thanks to Propositions \ref{prop:diagram-graded} and \ref{prop:diagram-z/2-splitting}.

The elements $\tau\in\boldsymbol{CF}^{-}(\boldsymbol{\beta}_{a},\boldsymbol{\beta}_{b})$,
${\rm Id}_{E}\Theta^{+}\in\boldsymbol{CF}^{-}(\boldsymbol{\beta}_{c}^{E},\boldsymbol{\beta}_{c}'{}^{E})$,
$\boldsymbol{\sigma}\in\boldsymbol{CF}^{-}(\boldsymbol{\beta}_{b},\boldsymbol{\beta}_{c}^{E})$,
and $\boldsymbol{\rho}\in\boldsymbol{CF}^{-}(\boldsymbol{\beta}_{c}^{E},\boldsymbol{\beta}_{a})$
are invariant under changing the almost complex structure, since the
chain map induced by changing the almost complex structure preserves
both the homological $\mathbb{Z}$-grading and the Alexander $\mathbb{Z}/2$-grading,
and the above elements are the unique nonzero cycle in the corresponding
$\mathbb{Z}\oplus\mathbb{Z}/2$-grading. Hence, in particular, $\underline{\boldsymbol{\beta}_{ab}}$
is preserved under changing the almost complex structure.

Since the top homological $\mathbb{Z}$-grading of $\boldsymbol{CF}^{-}(\boldsymbol{\beta}_{a}[-1],\boldsymbol{\beta}_{c}^{E})$
is $-1$ and $\underline{\boldsymbol{\sigma}}\in\boldsymbol{CF}^{-}(\underline{\boldsymbol{\beta}_{ab}},\boldsymbol{\beta}_{c}^{E})$
has homological $\mathbb{Z}$-grading $0$, the element $\underline{\boldsymbol{\sigma}}\in\boldsymbol{CF}^{-}(\underline{\boldsymbol{\beta}_{ab}},\boldsymbol{\beta}_{c}^{E})$
is invariant under changing the almost complex structure. Similarly,
$\underline{\boldsymbol{\rho}}\in\boldsymbol{CF}^{-}(\boldsymbol{\beta}_{c}^{E},\underline{\boldsymbol{\beta}_{ab}})$
is invariant.

The maps $\Theta_{a}^{+}\in\boldsymbol{CF}^{-}(\boldsymbol{\beta}_{a}[-1],\boldsymbol{\beta}_{a}'[-1])$
and $\Theta_{b}^{+}\in\boldsymbol{CF}^{-}(\boldsymbol{\beta}_{b},\boldsymbol{\beta}_{b}')$
are invariant as well. Since the top homological $\mathbb{Z}$-grading
of $\boldsymbol{CF}^{-}(\boldsymbol{\beta}_{a}[-1],\boldsymbol{\beta}_{b})$
is $-1$, they are invariant as elements of $\boldsymbol{CF}^{-}(\underline{\boldsymbol{\beta}_{ab}},\underline{\boldsymbol{\beta}_{ab}'})$.
Hence, $\Theta^{+}=\underline{\Theta_{a}^{+}}+\underline{\Theta_{b}^{+}}$
is invariant.

Finally, if we know that $\mu_{2}(\underline{\boldsymbol{\sigma}},\underline{\boldsymbol{\rho}'})=\Theta^{+}$
and $\mu_{2}(\underline{\boldsymbol{\rho}},\underline{\boldsymbol{\sigma}'})=\Theta^{+}$
for some choice of almost complex structure, then in any other almost
complex structure, we know that they are equal in homology, but since
they are in the top homological $\mathbb{Z}$-grading, they must be
equal.

\subsection{Finishing off the proof}

We have considered the twisted complexes $\underline{\boldsymbol{\beta}_{ab}}:=\boldsymbol{\beta}_{a}\xrightarrow{\tau}\boldsymbol{\beta}_{b}$,
$\underline{\boldsymbol{\beta}_{02}}:=\boldsymbol{\beta}_{0}\xrightarrow{\theta}\boldsymbol{\beta}_{2}$,
and $\underline{\widetilde{\boldsymbol{\beta}_{02}}}:=\widetilde{\boldsymbol{\beta}_{0}}\xrightarrow{\widetilde{\theta}}\widetilde{\boldsymbol{\beta}_{2}}$,
and we have defined maps between each consecutive pair in the sequence
$\underline{\boldsymbol{\beta}_{ab}},\underline{\boldsymbol{\beta}_{02}},\underline{\widetilde{\boldsymbol{\beta}_{02}}},\boldsymbol{\beta}_{c}^{E}$
in Sections \ref{sec:The-plan}, \ref{sec:proof-step1}, and \ref{sec:proof-step3}
(see Remark \ref{rem:explicit-moral-homotopy-equiv}). However, we
assumed in Claims \ref{claim:step-1} and \ref{claim:step-3} that
the underlying almost complex structures are sufficiently stretched.
Fortunately, the maps $\tau$, $\theta$, and $\widetilde{\theta}$
are invariant under the $A_{\infty}$-functors defined by changing
the almost complex structure since there are no domains that could
contribute. Hence, all the twisted complexes that we have considered
are invariant as well, and so by applying the $A_{\infty}$-functor,
we get pairs of maps (which are cycles) between each consecutive pair
in the sequence $\underline{\boldsymbol{\beta}_{ab}},\underline{\boldsymbol{\beta}_{02}},\underline{\widetilde{\boldsymbol{\beta}_{02}}},\boldsymbol{\beta}_{c}^{E}$,
in any almost complex structure. Let $\underline{\widetilde{\boldsymbol{\rho}}}:\underline{\boldsymbol{\beta}_{ab}}\to\boldsymbol{\beta}_{c}^{E}$
and $\underline{\widetilde{\boldsymbol{\sigma}}}:\boldsymbol{\beta}_{c}^{E}\to\underline{\boldsymbol{\beta}_{ab}}$
be compositions of these maps. Then, $\underline{\widetilde{\boldsymbol{\rho}}}$
and $\underline{\widetilde{\boldsymbol{\sigma}}}$ are cycles, and
have Alexander $\mathbb{Z}/2$-grading $0$.

Let $\boldsymbol{\alpha}$ be any attaching curve, and let $t:\underline{\boldsymbol{\gamma}_{1}}\to\underline{\boldsymbol{\gamma}_{2}}$
be any of the above maps between each consecutive pair in the sequence
$\underline{\boldsymbol{\beta}_{ab}},\underline{\boldsymbol{\beta}_{02}},\underline{\widetilde{\boldsymbol{\beta}_{02}}},\boldsymbol{\beta}_{c}^{E}$.
Then, 
\[
\mu_{2}(-,t):\boldsymbol{CF}^{-}(\boldsymbol{\alpha},\underline{\boldsymbol{\gamma}_{1}})\to\boldsymbol{CF}^{-}(\boldsymbol{\alpha},\underline{\boldsymbol{\gamma}_{2}})
\]
is a quasi-isomorphism by Lemma \ref{lem:theta-quasi-iso}, Remark
\ref{rem:theta-quasi-iso}, and Lemma \ref{lem:derived-nakayama}.
Hence, $\mu_{2}(-,\underline{\widetilde{\boldsymbol{\rho}}})$ and
$\mu_{2}(-,\underline{\widetilde{\boldsymbol{\sigma}}})$ are also
quasi-isomorphisms.

We claim that the top homological $\mathbb{Z}$-grading components
of $\mu_{2}(\underline{\widetilde{\boldsymbol{\rho}}},\underline{\widetilde{\boldsymbol{\sigma}}'}):\underline{\boldsymbol{\beta}_{ab}}\to\underline{\boldsymbol{\beta}_{ab}'}$
and $\mu_{2}(\underline{\widetilde{\boldsymbol{\sigma}}},\underline{\widetilde{\boldsymbol{\rho}}'}):\boldsymbol{\beta}_{c}^{E}\to\boldsymbol{\beta}_{c}'{}^{E}$
are $\Theta^{+}$. Let $\boldsymbol{\beta}=\boldsymbol{\beta}_{a}''$
be another standard translate of $\boldsymbol{\beta}_{a}$. Then,
$(\mathbb{T}^{2},\boldsymbol{\beta},\boldsymbol{\beta}_{a},\boldsymbol{\beta}_{b},\boldsymbol{\beta}_{a}',\boldsymbol{\beta}_{b}',v,w,z)$
is weakly admissible, homologically $\mathbb{Z}$-gradable, and Alexander
$\mathbb{Z}/2$-splittable. Since 
\[
\mu_{2}(-,\mu_{2}(\underline{\widetilde{\boldsymbol{\rho}}},\underline{\widetilde{\boldsymbol{\sigma}}'})):\boldsymbol{CF}^{-}(\boldsymbol{\beta},\underline{\boldsymbol{\beta}_{ab}})\to\boldsymbol{CF}^{-}(\boldsymbol{\beta},\underline{\boldsymbol{\beta}_{ab}'})
\]
is a quasi-isomorphism and $\boldsymbol{HF}^{-}(\boldsymbol{\beta},\underline{\boldsymbol{\beta}_{ab}})\neq0$,
it cannot strictly lower the homological $\mathbb{Z}$-grading. Since
the top homological $\mathbb{Z}$-grading of $\boldsymbol{CF}^{-}(\underline{\boldsymbol{\beta}_{ab}},\underline{\boldsymbol{\beta}_{ab}'})$
is $0$, the top homological $\mathbb{Z}$-grading component of $\mu_{2}(\underline{\widetilde{\boldsymbol{\rho}}},\underline{\widetilde{\boldsymbol{\sigma}}'})$
has homological $\mathbb{Z}$-grading $0$, and so it must be $\Theta^{+}=\underline{\Theta_{a}^{+}}+\underline{\Theta_{b}^{+}}$,
as $\Theta^{+}$ is the unique cycle in homological $\mathbb{Z}$-grading
$0$ and Alexander $\mathbb{Z}/2$-grading $0$. We similarly get
that the top homological $\mathbb{Z}$-grading component of $\mu_{2}(\underline{\widetilde{\boldsymbol{\sigma}}},\underline{\widetilde{\boldsymbol{\rho}}'})$
is ${\rm Id}_{E}\Theta_{c}^{+}$.

Now, we claim that $\underline{\boldsymbol{\rho}}$ and $\underline{\boldsymbol{\sigma}}$
are the top homological $\mathbb{Z}$-grading components of $\underline{\widetilde{\boldsymbol{\rho}}}$
and $\underline{\widetilde{\boldsymbol{\sigma}}}$, respectively.
Since the top homological $\mathbb{Z}$-gradings of $\boldsymbol{CF}^{-}(\underline{\boldsymbol{\beta}_{ab}},\boldsymbol{\beta}_{c}^{E})$
and $\boldsymbol{CF}^{-}(\boldsymbol{\beta}_{c}^{E},\underline{\boldsymbol{\beta}_{ab}'})$
are $0$ and the top homological $\mathbb{Z}$-grading component of
$\mu_{2}(\underline{\widetilde{\boldsymbol{\rho}}},\underline{\widetilde{\boldsymbol{\sigma}}'})$
has homological $\mathbb{Z}$-grading $0$, the top homological $\mathbb{Z}$-grading
component of $\underline{\widetilde{\boldsymbol{\rho}}}$ has homological
$\mathbb{Z}$-grading $0$. Since the top homological $\mathbb{Z}$-grading
of $\boldsymbol{CF}^{-}(\boldsymbol{\beta}_{b},\boldsymbol{\beta}_{c}^{E})$
is $0$ and the top homological $\mathbb{Z}$-grading of $\boldsymbol{CF}^{-}(\boldsymbol{\beta}_{a}[-1],\boldsymbol{\beta}_{c}^{E})$
is $-1$, the top homological $\mathbb{Z}$-grading and Alexander
$\mathbb{Z}/2$-grading $0$ summand of $\boldsymbol{CF}^{-}(\underline{\boldsymbol{\beta}_{ab}},\boldsymbol{\beta}_{c}^{E})$
is spanned by $e_{0}\underline{\rho_{1}}$ and $e_{0}\underline{\rho_{2}}$,
and $e_{0}\underline{\rho_{1}}+e_{0}\underline{\rho_{2}}=\underline{\boldsymbol{\rho}}$
is the unique nonzero element that can be a cycle in $\boldsymbol{CF}^{-}(\underline{\boldsymbol{\beta}_{ab}},\boldsymbol{\beta}_{c}^{E})$.
Since the homological $\mathbb{Z}$-grading $0$ component of $\underline{\widetilde{\boldsymbol{\rho}}}$
is a nonzero cycle, it must be $\underline{\boldsymbol{\rho}}$. Similarly,
$\underline{\boldsymbol{\sigma}}$ is the top homological grading
component of $\underline{\widetilde{\boldsymbol{\sigma}}}$.

Since the top homological $\mathbb{Z}$-grading components of $\underline{\widetilde{\boldsymbol{\rho}}}$,
$\underline{\widetilde{\boldsymbol{\sigma}}}$, $\mu_{2}(\underline{\widetilde{\boldsymbol{\rho}}},\underline{\widetilde{\boldsymbol{\sigma}}'})$,
and $\mu_{2}(\underline{\widetilde{\boldsymbol{\sigma}}},\underline{\widetilde{\boldsymbol{\rho}}'})$
are $\underline{\boldsymbol{\rho}}$, $\underline{\boldsymbol{\sigma}}$,
$\Theta^{+}$, and $\Theta^{+}$, respectively, and they have homological
$\mathbb{Z}$-grading $0$, we have $\mu_{2}(\underline{\boldsymbol{\rho}},\underline{\boldsymbol{\sigma}'})=\Theta^{+}\in\boldsymbol{CF}^{-}(\underline{\boldsymbol{\beta}_{ab}},\underline{\boldsymbol{\beta}_{ab}'})$
and $\mu_{2}(\underline{\boldsymbol{\sigma}},\underline{\boldsymbol{\rho}}')=\Theta^{+}\in\boldsymbol{CF}^{-}(\boldsymbol{\beta}_{c},\boldsymbol{\beta}_{c}'{}^{E})$.

\bibliographystyle{amsalpha}
\bibliography{essay_bib}

\end{document}